\documentclass[leqno]{memoir}

\setsecnumdepth{paragraph}
\chapterstyle{ntglike}
\setsecheadstyle{\large\bfseries}
\setsubsecheadstyle{\normalsize\bfseries}
\nouppercaseheads

\usepackage{amsmath,amssymb,amsthm,slashed,amsfonts,verbatim,mathrsfs}
\usepackage[top = 1in, left = 1.25in, right = 1in, bottom = 1in]{geometry}
\usepackage{cite}
\usepackage{graphicx}
\usepackage[all]{xy}
\usepackage{float}
\usepackage{tikz}
\usepackage{enumerate}

\usepackage{mathtools}

\DeclarePairedDelimiter\floor{\lfloor}{\rfloor}

\usepackage{etex}
\usepackage{xcolor,colortbl}
\usepackage{booktabs}
\usetikzlibrary{matrix}
\usepackage{array}

\def\ca{{\mathcal A}}
\def\cb{{\mathcal B}}
\def\cc{{\mathcal C}}

\def\ce{{\mathcal E}}
\def\cf{{\mathcal F}}

\def\ci{{\mathcal I}}

\def\cl{{\mathcal L}}

\def\cn{{\mathcal N}}

\def\cp{{\mathcal P}}

\def\cu{{\mathcal U}}
\def\cv{{\mathcal V}}

\def\Br{\mathbb R}
\def\Bz{\mathbb Z}
\def\Bn{\mathbb N}

\def\sys{(X,\cb,\mu,T)}

\DeclareMathOperator{\Int}{Int}
\DeclareMathOperator{\Asc}{Asc}

\DeclareMathOperator{\Alt}{Alt}

\newcommand{\htop}{h_{\text{\normalfont top}}}
\DeclareMathOperator{\topo}{\text{\normalfont top}}

\DeclareMathOperator{\card}{card}
\DeclareMathOperator{\Pal}{Pal}
\DeclareMathOperator{\Freq}{Freq}
\DeclareMathOperator{\Sep}{Sep}
\DeclareMathOperator{\mdim}{mdim}

\def\be {\begin{equation}}
\def\en{\end{equation}}

\numberwithin{equation}{section}

\newtheorem{prop}{Proposition}[section]
\newtheorem{thm}[prop]{Theorem}
\newtheorem*{thm*}{Theorem}          
\newtheorem*{cor*}{Corollary}        
\newtheorem*{lem*}{Lemma}

\newtheorem{cor}[prop]{Corollary}

\theoremstyle{definition} 
\newtheorem{defn}[prop]{Definition}
\newtheorem{conj}[prop]{Conjecture}
\newtheorem{exer}{Exercise}[section]
\newtheorem{example}[prop]{Example}
\newtheorem{rem}[prop]{Remark}

\theoremstyle{remark}

\newtheorem*{ack*}{Acknowledgment}

\makeatletter
\renewcommand\@biblabel[1]{#1.}
\makeatother

\title{Measuring Complexity in Cantor Dynamics}

\author{Karl Petersen\\
	{cantorsalta2015: Dynamics on Cantor Sets}\\
    CIMPA Research School, 2-13 November 2015}

\begin{document}
	
		\maketitle
		
		\newpage

	

\tableofcontents
\newpage

\chapter{Introduction}

\section{Preface}

  In these notes we discuss several quantitative definitions of the broad and vague notion of complexity, especially from the viewpoint of dynamical systems, focusing on transformations on the Cantor set, in particular shift dynamical systems. After this introduction, the second part reviews dynamical  entropy, the asymptotic exponential growth rate of the number of patterns or available information, uncertainty, or randomness in a system as the window size grows. The third part treats the more precise complexity function, which counts the actual number of patterns of each size, and then several of its variations. In the fourth part we present a new quantity that measures the balance within a system between coherent action of the whole and independence of its parts. There is a vast literature on these matters, new papers full of new ideas are appearing all the time, and there are plenty of questions to be asked and investigated with a variety of approaches. (Our list of references is in no way complete.) Some of the attractiveness of the subject is due to the many kinds of mathematics that it involves: combinatorics, number theory, probability, and real analysis, as well as dynamics. For general background and useful surveys, see for example \cite{Fogg,Berthe,Ferenczi1999,allouche94,Lothaire,LM,petersen1989ergodic,waltersergodic}.

  Parts of these notes are drawn from earlier writings by the author or from theses of his students Kathleen Carroll and Benjamin Wilson (see also \cite{CarrollPetersen2015,PetersenWilson2015}). Thanks to Val\'{e}rie Berth\'{e}, Francesco Dolce, and Bryna Kra, among others, for pointing out references. I thank the participants, organizers, and supporters of the CIMPA Research School cantorsalta2015, Dynamics on Cantor Sets, for the opportunity to participate and the incentive to produce these notes.
\section{Complexity and entropy}

An elementary question about any phenomenon under observation is, how many possibilities are there. A system that can be in one of a large but finite number of states may be thought to be more complex than one that has a choice among only a few. Then consider a system that changes state from time to time, and suppose we note the state of the system at each time. How many possible histories, or trajectories, can there be in a time interval of fixed length? This is the {\em complexity function}, and it provides a quantitative way to distinguish relatively simple systems (for example periodic motions) from more complicated (for example ``chaotic") ones. In systems with a lot of freedom of motion the number of possible histories may grow very rapidly as the length of time it is observed increases. The exponential growth rate of the number of histories is the {\em entropy}. While it may seem to be a very crude measure of the complexity of a system, entropy has turned out to be the single most important and useful number that one can attach to a dynamical system.

\section{Some definitions and notation}\label{sec:defs}

\indent  A \emph{topological dynamical system} is a pair $(X,T)$, where $X$ is a compact Hausdorff space (usually metric) and $T: X \to X$ is a continuous mapping. In these notes $X$ is usually the Cantor set, often in a specific representation as a subshift or as the set of infinite paths starting at the root in a Bratteli diagram---see below. A {\em measure-preserving system} $(X, \cb, T, \mu)$ consists of a measure space $(X, \cb, \mu)$ and a measure-preserving transformation $T: X \to X$. Often no generality is lost in assuming that $(X, \cb, \mu)$ is the Lebesgue measure space of the unit interval, and in any case usually we assume that $\mu (X)=1$. 
A {\em standard} measure space is one that is measure-theoretically isomorphic to a Borel subset of a complete metric space with its $\sigma$-algebra of Borel sets and a Borel measure. 
$T$ is assumed to be defined and one-to-one a.e., with $T^{-1}\cb \subset \cb$ and $\mu T^{-1}=\mu$. The system is called {\em ergodic} if for every invariant measurable set (every $B \in \cb$ satisfying $\mu(B \triangle T^{-1}B) = 0$) either $\mu(B)=0$ or $\mu(X \setminus B) =0$. 

A \emph{homomorphism} or \emph{factor mapping} between topological dynamical systems $(X,T)$ and $(Y,S)$ is a continuous onto map $\phi : X \to Y$ such that $\phi T = S\phi$. We say Y is a \emph{factor} of X, and X is an \emph{extension} of Y. If $\phi$ is also one-to-one, then $(X,T)$ and $(Y,M)$ are \emph{topologically conjugate} and $\phi$ is a \emph{topological conjugacy}. A homomorphism or factor mapping  between measure-preserving systems $(X, \cb, T, \mu)$ and $(Y,\cb,S,\nu)$ is a map $\phi: X \to Y$ such that $\phi^{-1}\cc \subset \cb$, $\phi T = S \phi$ a.e., and $\mu \phi^{-1}=\nu$. If in addition $\phi$ is one-to-one a.e., equivalently $\phi^{-1}\cc=\cb$ up to sets of measure 0, then $\phi$ is an {\em isomorphism}.

We focus in these notes especially on topological dynamical systems which are \emph{shift dynamical systems}. Let $\ca$ be a finite set called an \emph{alphabet}. The elements of this set are \emph{letters} and shall be denoted by digits. A \emph{sequence} is a one-sided infinite string of letters and a \emph{bisequence} is an infinite string of letters that extends in two directions.
The \emph{full $\ca$-shift}, $\Sigma(\ca)$, is the collection of all bisequences of symbols from $\ca$. If $\ca$ has $n$  elements,
$$\Sigma(\ca) = \Sigma_n= \ca^{\Bz} = \{ x = (x_i)_{i \in \Bz}\hspace{2mm} | \hspace{2mm}x_i \in \ca \text{ for all } i \in \Bz \}.$$

The \emph{one-sided full $\ca$-shift} is the collection of all infinite sequences of symbols from $\ca$ and is  denoted
$$\Sigma^+(\ca) =\Sigma_n^+ = \ca^{\Bn} = \{ x = (x_i)_{i \in \Bn}\hspace{2mm} | \hspace{2mm} x_i \in \ca \text{ for all } i \in \Bn \}.$$
We also define the \emph{shift transformation} $\sigma : \Sigma(\ca) \to \Sigma(\ca)$ and $\Sigma^+(\ca) \to \Sigma^+(\ca)$ by
$$(\sigma x)_i = x_{i+1} \hspace{3mm} \text{for all } i.$$ The pair $(\Sigma_n, \sigma)$ is called the $n$-\emph{shift dynamical system}.

We give $\ca$ the discrete topology and  $\Sigma(\ca)$ and $\Sigma^+(\ca)$ the product topology. Furthermore, the topologies on $\Sigma(\ca)$ and $\Sigma^+(\ca)$ are compatible with the metric $d(x,y) = 1/2^n$, where $n = \inf \{|k|\hspace{2mm} |\hspace{2mm} x_k \neq y_k \}$ \cite{LM}. Thus two elements of $\Sigma(\ca)$ are close if and only if they agree on a long central block. In a one-sided shift, two elements are close if and only if they agree on a long initial block.

A \emph{subshift} is a pair $(X, \sigma)$ (or $(X^+, \sigma)$),  where $X \subset \Sigma_n$ (or $X^+ \subset \Sigma_n^+$) is a nonempty, closed, shift-invariant set. We will be concerned primarily with subshifts of the 2-shift dynamical system.

A finite string of letters from $\ca$ is called a \emph{block} and the length of a block $B$ is denoted $|B|$. Furthermore, a block of length $n$ is an $n$-block. A {\em formal language} is a set $\cl$ of blocks, possibly including the empty block $\epsilon$, on a fixed finite alphabet $\ca$. The set of all blocks on $\ca$, including the empty block $\epsilon$, is denoted by $\ca^*$. Given a subshift $(X, \sigma)$ of a full shift, let $\cl_n(X)$ denote the set of all $n$-blocks that occur in points in $X$. The \emph{language} of $X$ is the collection
\begin{center}
	$\displaystyle \cl(X) = \bigcup_{n=0}^{\infty}\cl_n(X).$
\end{center}
A {\em shift of finite type} (SFT) is a subshift determined by excluding a finite set of blocks.

Let $\ca$ be a finite alphabet. A map $\theta:\ca \to \ca^*$ is called a {\em substitution}.
A substitution $\theta$ is extended to $\ca^*$ and $\ca^\Bn$ by
$\theta(b_1b_2\dots) = \theta(b_1)\theta(b_2)\dots$.
A substitution $\theta$ is called {\em primitive} if there is $m$ such that for all $a \in \ca$ the block $\theta^m(a)$ contains every element of $\ca$.

There is natural dynamical system associated with any sequence. Given a one-sided sequence $u$, we let $X_u^+$ be the closure of $\{\sigma^nu | n \in \Bn\}$, where $\sigma$ is the usual shift. Then $(X_u^+, \sigma)$ is the dynamical system associated with $u$. For any sequence $u$, denote by $\cl(u)$ the family of all subblocks of $u$.

\begin{exer}
	Show that $X_u^+$ consists of all the one-sided sequences on $\ca$ all of whose subblocks are subblocks of $u$: $X_u^+=\{x: \cl(x) \subset \cl(u)\}$.
\end{exer}

In a topological dynamical system $(X,T)$, a point $x \in X$ is called \emph{almost periodic} or \emph{syndetically recurrent} if for every $\epsilon >0$ there is some $N=N(\epsilon)$ such that the set $\{n\geq 0: d(T^nx,x) < \epsilon \}$ has gaps of size at most $N.$ If $X$ is a subshift, then $x\in X$ is almost periodic if and only if every allowed block in $x$ appears in $x$ with bounded gaps.

A topological dynamical system $(X,T)$ is \emph{minimal} if one of the following equivalent properties holds:
\begin{enumerate}[1.]
	\item  $X$ contains no proper closed invariant set;
	\item	  $X$ is the orbit closure of an almost periodic point;
	\item	  every $x \in X$ has a dense orbit in $X$.
\end{enumerate}

The {\em complexity function} of a language $\cl$ is the function $p_\cl (n) = \card (\cl \cap \ca^n), n \geq 0$. This is an elementary, although possibly complicated and informative, measure of the size or complexity of a language and, if the language is that of a subshift, of the complexity of the associated symbolic dynamical system. Properties of this function (for example its asymptotic growth rate, which is the topological entropy of the associated subshift) and extensions and variations comprise most of the subject matter of these notes.

A {\em Bratteli diagram} is a graded graph
whose set $\cv$ of vertices is the disjoint union of finite sets $\cv_n, n=0,1,2,\dots$;
$\cv_0$ consists of a single vertex $v_0$, called the {\em root};
and the set $\ce$ of edges is also the disjoint union of sets $\ce_n, n=1,2,\dots$
such that the source vertex of each edge $e \in \ce_n$ is in $\cv_{n-1}$ and its range vertex is in $\cv_n$.
Denote by $X$ the set of all infinite directed paths$x=(x_n), x_n \in \ce_n$ for all $n \geq 1$, in this graph that begin at the root.
$X$ is a compact metric space when we agree that two paths are close if they agree on a long initial segment.
Except in some degenerate situations the space $X$ is infinite, indeed uncountable, and homeomorphic to the Cantor set.

Suppose we fix a linear order on the set of	edges into each vertex.
Then the set of paths $X$ is partially ordered as follows:
two paths
$x$ and $y$ are {\em comparable} if they agree from some point on, in
which case we say that $x<y$ if at the last level $n$ where they
are different, the edge  $x_n$ of $x$ is smaller than the edge
$y_n$ of $y.$ A map $T$, called the {\em Vershik map}, is defined by letting $Tx$ be the smallest
$y$ that is larger than $x,$ if there is one.
There may be maximal paths $x$ for which $Tx$ is not defined, as well as minimal paths.
In nice situations,
$T$ is a homeomorphism after the deletion of perhaps countably many
maximal and minimal paths and their orbits.
If the diagram is {\em simple}---which means that for every $n$ there is $m>n$ such that there is a path in the graph from every $v \in \cv_n$ to every $w \in \cv_m$---and if there are exactly one maximal path $x_{\max}$ and exactly one minimal path $x_{\min}$, then one may define $Tx_{\max} = x_{\min}$ and arrive at a minimal homeomorphism $T: X \to X$. See \cite{Durand2010, BezuglyiKarpel2015} for surveys on Bratteli-Vershik systems.

\section{Realizations of systems}

There are several results concerning the realization of measure-preserving systems as topological dynamical systems up to measure-theoretic isomorphism.

1. The Jewett-Krieger Theorem states that every non-atomic ergodic measure-preserving system on a Lebesgue space  is measure-theoretically isomorphic to a system $(X,\cb,\mu,T)$ in which $X$ is the Cantor set, $\cb$ is the completion of the Borel $\sigma$-algebra of $X$, $T$ is a minimal homeomorphism (every orbit is dense), and $\mu$ is a unique $T$-invariant Borel probability measure on $X$.

2. The Krieger Generator Theorem says that every ergodic measure-preserving system $\sys$ of finite entropy (see below) is measure-theoretically isomorphic to a subsystem of any full shift which has strictly larger (topological) entropy---see below---with a shift-invariant Borel probability measure. Thus full shifts are ``universal" in this sense. The proof is accomplished by producing a finite measurable partition of  $X$ such that coding orbits according to visits of the members of the partition produces a map to the full shift that is one-to-one a.e.
Recently Seward \cite{Seward2015a,Seward2015b} has extended the theorem to actions of countable groups, using his definition of Rokhlin entropy (see Section \ref{sec:soficent}).

3. Krieger proved also that such an embedding is possible into any mixing shift of finite type (see below) that has strictly larger topological entropy than the measure-theoretic entropy of $\sys$. Moreover, he gave necessary and sufficient conditions that an expansive homeomorphism of the Cantor set be topologically conjugate to s subshift of a given mixing shift of finite type.
``There is a version of the finite generator theorem for ergodic measure preserving
transformations of finite entropy, that realizes such a transformation by means of an
invariant probability measure of any irreducible and aperiodic topological Markov
chain, whose topological entropy exceeds the entropy of the transformation
([4], [2 § 28]). One can say that a corollary of theorem 3 achieves for minimal
expansive homeomorphisms of the Cantor discontinuum what the finite generator
theorem does for measure preserving transformations." \cite{Krieger1982}

4. Lind and Thouvenot \cite{LindThouvenot1977} proved that hyperbolic toral automorphisms (the matrix has no eigenvalue of modulus 1) are universal.
This was extended by Quas and Soo \cite{QuasSoo2014} to quasi-hyperbolic toral automorphisms (no roots of unity among the eigenvalues), and they also showed that the time-1 map of the geodesic flow on a compact surface of constant negative curvature is universal \cite{QuasSoo2012}.

5. Every minimal homeomorphism of the Cantor set is topologically conjugate to the Vershik map on a simple Bratteli diagram with unique maximal and minimal paths \cite{HermanPutnamSkau1992, GiordanoPutnamSkau1995}.

6. Every ergodic measure-preserving system is measure-theoretically isomorphic to a minimal Bratteli-Vershik system with a unique invariant Borel probability measure \cite{Vershik1982, Vershik1985}.

\chapter{Asymptotic exponential growth rate}

\section{Topological entropy}\label{sec:topent}

Let $X$ be a compact metric
space and $T:X\rightarrow X$ a homeomorphism.

{\em First definition} \cite{AdlerKonheimMcAndrew1965}: For an open
cover ${\mathcal U}$ of $X,$ let $N({\mathcal U})$ denote the minimum number
of elements in a subcover of ${\mathcal U},$ $H({\mathcal U})=\log N({\mathcal
	U}),$
\begin{equation*}
h({\mathcal U},T)= \lim_{ n\rightarrow \infty}
{\frac{1}{n}}H({\mathcal U}\vee T^{-1}{\mathcal U}\vee \ldots \vee T^{-n+1}{
	\mathcal U}) ,
\end{equation*}
and
\begin{equation*}
h(T)=\sup_{\mathcal U} h({\mathcal U},T) .
\end{equation*}

{\em Second definition}\label{def:bowenent} \cite{Bowen1971}: For $n\in {\mathbb N}$ and $
\epsilon >0,$ a subset $A\subset X$ is called $n,\epsilon -${\em separated
}  if given $a,b\in A$ with $a\not= b,$ there is $k\in \{
0,\ldots ,n-1\} $ with $d(T^ka,T^kb)\geq \epsilon .$ We
let $S(n,\epsilon )$ denote the {\em maximum}  possible cardinality
of an $n,\epsilon$-separated set. Then
\begin{equation*}
h(T)= \lim_{\epsilon \rightarrow 0^+}
\limsup_{n\rightarrow \infty }{\frac{1}{n}} \log S(n,
\epsilon ) .
\end{equation*}

{\em Third definition} \cite{Bowen1971}: For $n\in {\mathbb N}$ and $\epsilon
>0,$ a subset $A\subset X$ is called $n,\epsilon$-{\em spanning
}  if given $x\in X$ there is $a\in A$ with $d(T^ka,T^kx)\leq
\epsilon $ for all $k=0,\ldots ,n-1.$ We let $R(n,\epsilon )$ denote
the {\em minimum}  possible cardinality of an $n,\epsilon$-spanning set. Then
\begin{equation*}
h(T)= \lim_{\epsilon \rightarrow 0^+}
\limsup_{n\rightarrow \infty }{\frac{1}{n}} \log R(n,
\epsilon ) .
\end{equation*}

\begin{exer}If $(X,T)$ is a subshift ($X=$ a closed shift-invariant
	subset of the set of all doubly infinite sequences on a finite alphabet,
	$T=\sigma=$ shift transformation), then
	\begin{equation*}
	h(\sigma )= \lim_{n\rightarrow \infty }
	\frac{\log(\text{number of }n\text{-blocks seen in sequences in }X)}{n} .
	\end{equation*}
\end{exer}

\begin{thm}[``Variational Principle''] $h(T)=\sup\{
	h_\mu (T):$ $\mu $ is an invariant (ergodic) Borel probability measure
	on $X\} .$
\end{thm}

\section{Ergodic-theoretic entropy}

\subsection{Definition} A finite (or sometimes countable)
measurable partition
\begin{equation*}
\alpha =\{ A_1,\ldots ,A_r\}
\end{equation*}
of $X$ is thought of as the set of possible outcomes of an experiment
(performed at time 0) or as an alphabet of symbols used to form
messages (the experiment could consist of receiving and reading
one symbol). The {\em entropy}  of the partition is
\begin{equation*}
H_\mu (\alpha )=\sum \limits_{A\in \alpha }-\mu (A)\log\mu
(A) \quad\text{ (the logs can be base } e, 2, \text{ or }r);
\end{equation*}
it represents the amount of information gained$=$amount of uncertainty
removed when the experiment is performed or one symbol is received
(averaged over all possible states of the world---the amount of
information gained if the outcome is $A$ (i.e., we learn to which
cell of $\alpha $ the world actually belongs) is $-\log\mu (A))$.
(Note that this is large when $\mu(A)$ is small.)
Notice that the information gained when we learn that an event $A$
occurred is additive for independent events.

The partition
\begin{equation*}
T^{-1}\alpha =\{ T^{-1}A:A\in \alpha \}
\end{equation*}
represents performing the experiment $\alpha $ (or reading a symbol)
at time 1, and $\alpha \vee T^{-1}\alpha \vee \ldots \vee T^{-n+1}
\alpha $ represents the result of $n$ repetitions of the experiment
(or the reception of a string of $n$ symbols). Then $H(\alpha \vee
T^{-1}\alpha \vee \ldots \vee T^{-n+1}\alpha )/n$ is the average
information gain per repetition (or per symbol received), and
\begin{equation*}
h(\alpha ,T)=\lim_{n\rightarrow \infty}
{\frac{1}{n}} H(\alpha \vee T^{-1}\alpha \vee \ldots \vee
T^{-n+1}\alpha )
\end{equation*}
is the long-term time average of the information gained per unit
time.
(This limit exists because of the {\em subadditivity} of $H$:
$H(\alpha \vee \beta) \leq H(\alpha) + H(\beta)$.)

The {\em entropy of the system}  $(X,{\mathcal B},\mu ,T)$ is
defined to be
\begin{equation*}
h_\mu (T)=\sup_{\alpha } h(\alpha,T) ,
\end{equation*}
the maximum information per unit time available from any finite-
(or countable-) state stationary process generated by the system.

\begin{thm}[Kolmogorov-Sinai] If $T$ has a finite {\em generator
	}   $\alpha $---a partition $\alpha $ such that the smallest
	$\sigma$-algebra that contains all $T^j\alpha ,j\in {\mathbb Z},$
	is ${\mathcal B}$---then $h_\mu (T)=h(\alpha ,T).$ (Similarly if $T$
	has a countable generator with finite entropy.)
\end{thm}

\begin{thm} If $\{ \alpha _k\} $ is
	an increasing sequence of finite partitions which generates ${\mathcal
		B}$ up to sets of measure $0$, then $h(\alpha _k,T)\rightarrow h(T)$
	as $k\rightarrow \infty .$
\end{thm}

\begin{rem}
	Entropy comes from {\em coarse-graining}. We admit that we do not have complete information about the state $x$ of a system $X$ but know for a partition $\alpha$ or open cover $\mathcal U$ of $X$ only which cell of $\alpha$ or member of $\mathcal U$ the point $x$ is in. 
	As time clicks forward, there is uncertainty about which cell or partition member the points $Tx, T^2x, \dots$ will be in. 
	Entropy measures the exponential rate at which this uncertainty grows. 
	Because we lack exact information, chance can be present even in completely deterministic systems such as automorphisms of the torus and time-one maps of geodesic flows, which are measure-theoretically isomorphic to Bernoulli shifts. 
\end{rem}

\subsection{Conditioning}  For a countable measurable partition
$\alpha $ and sub-$\sigma$-algebra ${\mathcal F}$ of ${\mathcal B},$ we
define the {\em conditional information function}  of $\alpha
$ given ${\mathcal F}$ by
\begin{equation*}
I_{\alpha \vert {\mathcal F}}(x)=-\sum \limits_{A\in \alpha
}\log\mu (A\vert {\mathcal F})(x)\chi _A(x);
\end{equation*}
this represents the information gained by performing the experiment
$\alpha $ (if the world is in state $x)$ after we already know for
each member of ${\mathcal F}$ whether or not it contains the point $x.$
The {\em conditional entropy} of $\alpha $ given ${\mathcal
	F}$ is
\begin{equation*}
H(\alpha \vert {\mathcal F})=\int _XI_{\alpha \vert {\mathcal
		F}}(x)d\mu (x);
\end{equation*}
this is the average over all possible states $x$ of the information
gained from the experiment $\alpha .$ When ${\mathcal F}$ is the
$\sigma$-algebra generated by a partition $\beta ,$ we often just write
$\beta $ in place of ${\mathcal F}.$

\begin{prop}
	\begin{enumerate}[1.]
	\item	 $H(\alpha \vee \beta \vert {\mathcal F})=H(
		\alpha \vert {\mathcal F})+H(\beta \vert {\mathcal B}(\alpha )\vee {\mathcal
			F}).$
		\item		 $H(\alpha \vert {\mathcal F})$ is increasing in its
		first variable and decreasing in its second.
	\end{enumerate}
\end{prop}

\begin{thm}  For any finite (or countable finite-entropy)
	partition $\alpha ,$
	\begin{equation*}
	h(\alpha ,T)=H(\alpha \vert {\mathcal B}(T^{-1}\alpha \vee
	T^{-2}\alpha \vee \ldots )).
	\end{equation*}
\end{thm}

\subsection{Examples}
\begin{enumerate}

\item{Bernoulli shifts:}\label{item:BS} $h=-\sum p_i\log p_i$ .
Consequently $\mathcal B(1/2,1/2)$ is not isomorphic to $\mathcal
B(1/3,1/3,1/3)$.

\item{Markov shifts:}\label{item:MS} $h=-\sum p_i\sum P_{ij}\log P_{ij}$ .

\item{Discrete spectrum (the span of the eigenfunctions is dense in $L^2$):}\label{item:discspec} $h=0$ . (Similarly for {\em rigid
}  systems---ones for which there is a sequence $n_k\rightarrow
\infty $ with $T^{n_k}f\rightarrow f$ for all $f\in L^2.)$ Similarly
for any system with a {\em one-sided generator}, for then
$h(\alpha ,T)=H(\alpha \vert \alpha _1^\infty )=H(\alpha \vert {
	\mathcal B})=0.$ It's especially easy to see for an irrational
rotation of the circle, for if $\alpha$ is the partition into two
disjoint arcs, then $\alpha_0^n$ only has $2(n+1)$ sets in it.

\item{Products:}\label{item:prod} $h(T_1\times T_2)=h(T_1)+h(T_2)$ .

\item{Factors:}\label{item:factor} If $\pi :T\rightarrow S,$ then $h(T)\geq h(S)$ .

\item{Bounded-to-one factors:} $h(T)=h(S)$. See \cite[p. 56]{ParryTuncel1982}.

\item{Skew products:} $h(T\times \{ S_x\}
)=h(T)+h_T(S)$ . Here the action is $(x,y)\rightarrow (Tx,S_xy),$
with each $S_x$ a m.p.t. on $Y,$ and the second term is the {\em fiber
	entropy}
\begin{equation*}
h_T(S)=\sup\{\int _XH(\beta \vert S_x^{-1}
\beta \vee S_x^{-1}S_{Tx}^{-1}\beta \vee \ldots )d\mu (x):\beta
\text{ is a finite partition of }Y\} .
\end{equation*}

\item{Automorphism of the torus:} $h=\sum \limits_{{
		\left|{\lambda _i}\right|}>1}\log {\left|{\lambda _i}\right|}$
(the $\lambda _i$ are the eigenvalues of the integer matrix with
determinant $\pm 1).$

\item{Pesin's Formula:}
If $\mu \ll m$ (Lebesgue measure on the manifold), then
\begin{equation*}
h_{\mu}(f) = \int \sum_{\lambda_k(x)>0} q_k(x)\lambda_k(x) \, d\mu
(x),
\end{equation*}
where the $\lambda_k(x)$ are the {\em Lyapunov exponents} and
$q_k(x)=\dim (V_k(x) \setminus V_{k-1}(x))$ are the dimensions of the
corresponding subspaces.

\item{Induced transformation (first-return map):} For $A\subset
X,$ $h(T_A)=h(T)/\mu (A)$.

\item{Finite rank ergodic:} $h=0$ .
\begin{proof} Suppose rank $=1,$ let $P$ be a
	partition into two sets (labels 0 and 1), let $\epsilon >0.$ Take
	a tower of height $L$ with levels approximately $P$-constant (possible
	by rank 1; we could even take them $P$-constant)
	and $\mu (\text{junk})<\epsilon .$ Suppose we follow the orbit
	of a point $N\gg L$ steps; how many different $P,N$-names can we
	see? Except for a set of measure $<\epsilon ,$ we hit the junk $n
	\sim \epsilon N$ times. There are $L$ starting places (levels of
	the tower); $C(N,n)$ places with uncertain choices of 0, 1; and
	$2^n$ ways to choose 0 or 1 for these places. So the sum of $\mu
	(A)\log\mu (A)$ over $A$ in $P_0^{n-1}$ is $\leq $ the log of the
	number of names seen in the good part minus the log of $2^N(\epsilon
	/2^N)\log (\epsilon /2^N),$ and dividing by $N$ gives
	\begin{equation*}
	{\frac{\log L}{N}}+NH(\epsilon ,1-\epsilon )+{\frac{N\epsilon
		}{N}}+{\frac{\epsilon (-\log \epsilon +N)}{N}}\sim 0 .
	\end{equation*}
	Similarly for any finite partition $P.$ Also for rank $r$---then
	we have to take care (not easily) about the different possible ways
	to switch columns when spilling over the top.
\end{proof}
\end{enumerate}

\begin{exer}
	Prove the statements in \ref{item:BS}--\ref{item:factor} above.
\end{exer}

\subsection{Ornstein's Isomorphism Theorem}  Two Bernoulli
shifts are isomorphic if and only if they have the same entropy.

\subsection{Shannon-McMillan-Breiman Theorem}  For a finite
measurable partition $\alpha ,$ and $x\in X,$ let $\alpha (x)$ denote
the member of $\alpha $ to which $x$ belongs, and let $\alpha _0^{n-1}=
\alpha \vee T^{-1}\alpha \vee \ldots \vee T^{-n+1}\alpha .$ If $T$
is an ergodic m.p.t. on $X,$ then
\begin{equation*}
{\frac{-\log\mu (\alpha _0^{n-1}(x))}{n}} \rightarrow
h(\alpha ,T) \text{  a.e. and in } L^1.
\end{equation*}

\subsection{Brin-Katok local entropy formula \cite{BrinKatok1983}} 
Let $(X,T)$ be a topological dynamical system (with $X$ compact metric as usual). 
For $x \in X, \delta >0, n \in \Bn$, define
\begin{equation}
\begin{gathered}
B(x,\delta ,n)=\{y: d(T^jx,T^jy)<\delta \text{ for all } j=1,\dots,n\},\\
h^+(x)=\lim_{\delta \to 0}\limsup_{n \to \infty}- \frac{\log \mu(B(x,\delta,n)}{n}, \text{ and }\\
h^-(x)=\lim_{\delta \to 0}\liminf_{n \to \infty}- \frac{\log \mu(B(x,\delta,n)}{n}.
\end{gathered}
\end{equation}
Then $h^+(x)=h^-(x)$ a.e., and
\begin{equation}
h_\mu(T) = \int_X h^+(x)\,d\mu.
\end{equation}
If $\sys$ is ergodic, then $h^+(x)=h^-(x)=h_\mu(T)$ a.e. with respect to $\mu$. 
Thus the Bowen-Dinaburg ball $B(x,\delta,n)$ plays the same role as the partition $\alpha_0^{n-1}$ in the Shannon-McMillan-Breiman Theorem. 
The functions $h^+$ and $h^-$ may be called the {\em upper and lower local entropies of $\sys$}. 

\subsection{Wyner-Ziv-Ornstein-Weiss entropy calculation algorithm}
 For a stationary ergodic sequence $\{ \omega _1,\omega
_2,\ldots \} $ on a finite alphabet and $n\geq 1,$ let
$R_n(\omega )=$ the first place to the right of 1 at which the initial
$n$-block of $\omega $ reappears (not overlapping its first appearance). Then
\begin{equation*}
\frac{\log R_n(\omega )}{n}\rightarrow h \text{ a.e..}
\end{equation*}
This is also related to the {\em Lempel-Ziv parsing algorithm},
in which a comma is inserted in a string $\omega $ each time a block
is completed (beginning at the preceding comma) which has not yet
been seen in the sequence.

\section{Measure-theoretic sequence entropy}\label{sec:mtseqent}

A. G. Kushnirenko \cite{Kushnirenko1967} defined the {\em sequence entropy} of a measure-preserving system $\sys$ with respect to a sequence  $S = (t_i)$ of integers  as follows. Given a (finite or countable) measurable partition $\alpha$ of $X$, as usual $H(\alpha) = - \sum_{A \in \alpha} \mu(A) \log \mu (A)$. Then we let
\be
h_\mu^S(\alpha,T)= \limsup \frac{1}{n}H(\bigvee_{i=1}^nT^{-t_i}\alpha) \quad\text{ and } \quad h_\mu^S(T)=\sup_\alpha h_\mu^S(\alpha,T),
\en
the supremum being taken over all countable measurable partitions of finite entropy. $h_\mu^S(\alpha,T)$ is the asymptotic rate of information gain if a measurement $\alpha$ is made on a dynamical system at the times in the sequence $S$. $h_\mu^S(T)$ is an isomorphism invariant. Kushnirenko
\begin{enumerate}
	\item calculated the $(2^n)$ entropy for some skew product maps on the torus; 
	\item used $(2^n)$ entropy to prove that the time-$1$ map of the horocycle flow on a two-dimensional orientable
manifold of constant negative curvature is not isomorphic to its
Cartesian square, although both have countable Lebesgue spectrum and entropy zero; and 
\item proved that a system $\sys$ has purely discrete spectrum if and only if $h_\mu^S(T)=0$ for every sequence $S$.
\end{enumerate}

D. Newton \cite{Newton1970} (see also \cite{KrugNewton1972}) showed that for each sequence $S$ there is a constant $K(S)$ such that for every system $\sys$, $h_\mu^S(T)=K(S)h(T)$, unless $K(S)=\infty$ and $h(T)=0$; if $K(S)=0$ and $h(T)=\infty$, then $h_\mu^S(T)=0$. thus sequence entropy is no better than ordinary entropy for distinguishing positive entropy systems up to isomorphism.  The constant $K(S)$ is defined as follows:
\be
K(S)= \lim_{k \to \infty}
\limsup_{n \to \infty} \frac{1}{n} \card \bigcup_{i=1}^n[-k+t_i,k+t_i].
\en

\begin{exer}
	Calculate $K(S)$ for the sequences $(i), (2^i), (3i),  (i^2), (p_i)$ ($p_i=i$'th prime).
\end{exer}

Saleski  \cite{Saleski1977, Saleski1979} established connections among sequence entropies and various mixing properties. In particular, he showed that (1) $\sys$ is strongly mixing if and only if for every (countable) measurable partition $\alpha$ with $H(\alpha) < \infty$, and every increasing sequence $S$ in $\Bn$, there is a subsequence $S_0$ for which $h_{S_0}(\alpha,T)=H(\alpha)$; and (2) $\sys$ is weakly mixing if and only if
for every countable measurable partition of finite entropy there is an increasing sequence $S$ in $\Bn$ for which $h_\mu^S(\alpha,T)=H(\alpha)$.

These results were improved by Hulse \cite{Hulse1982} as follows.
(1) $\sys$ is strongly mixing if and only if
every infinite sequence $S$ has a subsequence $S_0$
such that for every measurable partition $\alpha$ of finite entropy, $h_{S_0}(\alpha,T) = H(\alpha )$.
(2) $\sys$ is weakly mixing if and only if
there exists an increasing sequence $S$ in $\Bn$ such that for every measurable partition $\alpha$ of finite entropy, $h_\mu^S(\alpha,T) = H(\alpha)$.

Garcia-Ramos \cite{Garcia-Ramos2015} has recently clarified, in both the measure-theoretic and topological contexts, the connections among discrete spectrum, zero sequence entropy, and some new concepts of equicontinuity.

\section{Topological sequence entropy}\label{sec:topseqent}

T. N. T. Goodman  \cite{Goodman1974} defined sequence entropy for topological dynamical systems $(X,T)$, where $X$ is a compact Hausdorff space and $T:X \to X$ is a continuous map. As in the definition of topological entropy, for any open cover $\mathscr U$ of $X$, $N(\mathscr U)$ denotes the minimum cardinality of subcovers of $\mathscr U$ and $H(\mathscr U)= \log N(\mathscr U)$. For an open cover $\mathscr U$ and sequence $S=(t_i)$ of integers,
\be
\htop ^S(\mathscr U,T)=\lim_{n\to\infty}\frac{1}{n}H(\bigvee_{i=1}^nT^{-t_i}\alpha),
\en
and $\htop ^S (T)=\sup_{\mathscr U}\htop ^S (\mathscr U,T)$. Goodman gives equivalent Bowen-type definitions and establishes some basic properties. He proves the topological  analogue of Newton's result with the same constant $K(S)$ in case $X$ has finite covering dimension and shows that topological sequence entropy does not satisfy the variational principle---it is possible that $\htop ^S(T)>\sup_\mu h_\mu ^S\sys$.

Huang, Li, Shao, and Ye \cite{HuangLiShaoYe2003} studied ``sequence entropy pairs" and ``weak mixing pairs", in the spirit of the study of entropy pairs initiated by F. Blanchard in his investigation of topological analogues of the $K$ property \cite{Blanchard1992, BlanchardEtAl1995, BlanchardHostMaass2000}.
For sets $U, V \subset X$ in a topological dynamical system $(X,T)$, we define $\cn(U,V)=\{ n \geq 0: U \cap T^{-n}V \neq \emptyset\}$.
Then $(x_1, x_2) \in  X \times X$  is defined to be a {\em weak mixing pair}
if for any open neighborhoods $U_1, U_2$ of $x_1,x_2$, respectively, $\cn (U_1, U_1) \cap \cn (U_1, U_2) \neq \emptyset$  (cf. \cite{Petersen1970}).
Denote the set of weak mixing pairs by $WM(X, T )$.

We say that $(x_1, x_2) \in X \times X \setminus \triangle $ is a {\em sequence entropy pair} if whenever
$U_1,U_2$ are closed mutually disjoint neighborhoods of $x_1,x_2$, respectively, there exists a
sequence $S$ in $\Bn$  such that $h_\mu^S( \{U_1^c,U_2^c\},T)>0.$
The system $(X, T )$ is said to have {\em uniform positive sequence entropy (for short s.u.p.e.)} if every
pair $(x_1, x_2) \in X \times X \setminus \triangle$ is a sequence
entropy pair.

\begin{exer}
	$(X, T )$ has s.u.p.e. if and only if for any cover $\mathscr U$ of X by two non-dense open sets, there is a sequence $S$ in $\Bn$ such that $h_\mu^S(\mathscr U,T) > 0$.
\end{exer}

Huang, Li, Shao, and Ye show that $(X, T )$ is topologically weakly mixing if and only if
$WM(X, T ) = X \times X \setminus \triangle$, and that topological weak mixing is also equivalent to s.u.p.e.

Huang, Shao and Ye \cite{HuangShaoYe2005} gave new proofs of many known results about measure-theoretic and topological sequence entropy and characterized mild mixing and rigidity in terms of sequence entropy.
Recall that a measure-preserving system is weakly mixing if and only if its product with any finite measure-preserving system is ergodic. Furstenberg and Weiss \cite{FurstenbergWeiss1978} defined the stronger property of {\em mild mixing}: a finite measure-preserving system is mildly mixing if and only if its product with any ergodic system, even one preserving an infinite measure, is ergodic. A system is mildly mixing if and only if it has no proper rigid factors, i.e. no proper factors for which there exists a sequence $(n_i)$ such that $T^{n_i}f \to f$ in $L^2$ for all $f \in L^2$. Zhang \cite{Zhang1992, Zhang1993} provided sequence entropy characterizations of mild mixing and relative mild mixing.
An {\em IP set} is the set of all finite sums of a sequence of natural numbers. Huang, Shao, and Ye prove that an invertible measure-preserving system $\sys$ (with $\mu(X)=1$) is mildly mixing if and only if for every measurable set $A$ with measure strictly between 0 and 1 and every IP set $F$ there is an infinite sequence $S \subset F$ such that $h_\mu^S(\{B,B^c\},T)>0$; and that it is rigid if and only if there is an IP set $F$ such that $h_\mu^S(T)=0$ for every infinite $S \subset F$.

A topological dynamical system $(X,T)$ is called {\em topologically mildly mixing} if for every transitive system $(X',T')$ the product system $(X \times X', T \times T')$ is transitive \cite{HuangYe2004,GlasnerWeiss2006}. Huang, Shao, and Ye show that a topological dynamical system $(X,T)$ is topologically weakly mixing if and only if for each finite cover $\mathscr U$ of $X$ consisting of non-dense open sets there is an infinite sequence $S$ in $\Bn$ such that $\htop^S(\mathscr U, T)>0$; and that it is topologically mildly mixing if and only if  for each finite cover $\mathscr U$ of $X$ consisting of non-dense open sets and IP set $F$ there is an infinite sequence $S \subset F$ such that $\htop^S(\mathscr U, T)>0$.

\section{Slow entropy}\label{sec:slow}

There have been several efforts to attach to a dynamical system  measures of the rate of growth of $H_\mu(\alpha_0^{n-1})$, in the measure-preserving case, or $N(\cu_0^{n-1})$, in the topological case, that are finer than the exponential growth rate, which gives the ordinary entropies. The quantities may be defined so as to be isomorphism invariants, and then they can be used to try to distinguish systems of zero or infinite entropy up to isomorphism. It seems that more has been done in the zero entropy case, where there are many classes of familiar and interesting examples still requiring further study. The definitions of {\em entropy dimension} and {\em power entropy} are designed especially to detect polynomial growth rates. See \cite[p. 37 ff., p. 92 ff.]{HK2002} and \cite{KongChen2014}.

F. Blume \cite{Blume1995,Blume1997,Blume1998,Blume2000} defined a variety of rates of entropy convergence as follows. Let $\sys$ be a measure-preserving system and $\cp$ a fixed class of finite measurable partitions of $X$. 
Let $a=(a_n)$ be an increasing sequence of positive numbers with limit $\infty$, and let $c >0$. Then $\sys$ is said to be {\em of type $LS \geq c$ for $(a,P)$} if
\be
\limsup_{n \to \infty} \frac{H_\mu(\alpha_0^{n-1})}{a_n} \geq c \text { for all } \alpha \in \cp,
\en
and {\em of type $LI \geq c$ for $(a,P)$} if
\be
\liminf_{n \to \infty} \frac{H_\mu(\alpha_0^{n-1})}{a_n} \geq c \text { for all } \alpha \in \cp.
\en
The types $LS\leq c, LI \leq c, LS <\infty, LS=\infty, LI < \infty, LI=\infty, LS>0, LI>0$ are defined analogously. 

When focusing on zero-entropy systems, it is natural to consider the class $\cp$ of all partitions of $X$ into {\em two} sets of positive measure, since every zero-entropy system has a two-set generator \cite{Krieger1970}. 
Then the quantities defined above are invariants of measure-theoretic isomorphism. 
If $h_\mu(X,T,\alpha)=\lim_{n \to \infty} H_\mu(\alpha_0^{n-1})/n=0$, one should restrict attention to sequences $a$ for whch $a_n/n \to 0$. 

Pointwise entropy convergence rates can be defined by using the {\em information function}
\be
I_\alpha(x)=-\log_2\mu(\alpha(x)),
\en
where $\alpha(x)$ denotes the cell of $\alpha$ to which the point $x$ belongs. Then $\sys$ is defined to be {\em of pointwise type $LS \geq c$ for $(a,\cp)$} if 
\be
\limsup_{n \to \infty} \frac{I_{\alpha_0^{n-1}}(x)}{a_n} \geq c \text{ a.e. for all } \alpha \in \cp,
\en
and the other pointwise types are defined analogously.

Here are some of Blume's results in this area.

\begin{enumerate}
	\item If $a_n/n \to 0$ and $\cp$ is the class of two-set partitions, then there are no aperiodic measure-preserving systems of types $LI < \infty, LS < \infty, LS \leq c$, or $LI \leq c$.
	\item Suppose that $g: [0,\infty) \to \Br$ is a monotone increasing function for which
	\be
	\int_1^\infty \frac{g(x)}{x^2} dx < \infty
	\en
	and $a_n=g(\log_2n)$ for all $n=1,2,\dots$. If $T^k$ is ergodic for all $k \in \Bn$, then $\sys$ is of type $LS = \infty$ for $a,\cp$. 
	\item Every rank-one mixing transformation is of type $LS>0$ for $a_n=\log_2n$ and $\cp$.
	\item No rank-one system can have a convergence rate of type $LI \geq c$ for any sequence $a_n$ that grows faster than $\log_2n$: If $\sys$ is rank one, then there is a partition $\alpha \in \cp$ for which
	\be
	\liminf_{n \to \infty} \frac{H_\mu(\alpha_0^{n-1})}{\log_2n} \leq 2.
	\en
	\item Every totally ergodic system is of pointwise type $LS \geq 1$ for $((\log_2n),\cp)$.
	\end{enumerate}

\begin{exer}
	Show that every measure-preserving system is of pointwise type $LI< \infty$ and $LS<\infty$ for $((n),\cp)$. ({\em Hint}: SMB.)
\end{exer}

\begin{exer}
	Show that every {\em K-system} (all nontrivial factors have positive entropy) is of pointwise type $LS>0$ and $LI>0$ for $((n),\cp)$.
\end{exer}

In order to study the problem of the realizability of measure-preserving actions of amenable groups by diffeomorphisms of compact manifolds, Katok and Thouvenot \cite{KatokThouvenot1997} defined entropy-like measure-theoretic isomorphism invariants in terms of subexponential growth rates. 
One fixes a sequence $(a_n)$ of positive numbers increasing to infinity, or a family $(a_n(t)), 0<t<\infty$, such as $a_n(t)=n^t$ or $a_n(t)=e^{nt}$ (the latter will produce the usual entropy). 
We describe the case when the acting group is $\Bz^2$, i.e. the action of two commuting measure-preserving transformations $S$ and $T$ on a probability space $(X,\cb,\mu)$. Let $\alpha$ be a finite measurable partition of $X$ and $\epsilon > 0$. Let
\be
\alpha_n=\bigvee_{i=0}^{n-1}\bigvee_{j=0}^{n-1}S^{-i}\alpha \vee T^{-j}\alpha
\en
denote the refinement of $\alpha$ by the action of $[0,n-1]^2\subset \Bz^2$. 
The elements of each $\alpha_n$ are given a fixed order, so that they become vectors of length $n^2$, and the $\overline d$ distance between two of them is defined to be the proportion of places at which they disagree. 
We define $S(\alpha,n,\epsilon)$ to be the minimum number of elements of $\alpha_n$ such that the union of $\overline d, \epsilon$ balls centered at them has measure at least $1-\epsilon$. 

Denote by $A(\alpha,\epsilon)$ any one of 
\be
\begin{gathered}
\limsup_{n\to\infty}\frac{S(\alpha,n,\epsilon)}{a_n}, \qquad  \liminf_{n\to\infty}\frac{S(\alpha,n,\epsilon)}{a_n}, \\
\sup\{t: \limsup_{n \to \infty} S(\alpha,n,\epsilon) a_n(t)>0\}, \text{ or } \sup\{t: \liminf_{n \to \infty} S(\alpha,n,\epsilon) a_n(t)>0\},
\end{gathered}
\en
and then define
\be
A(\alpha) = \lim_{\epsilon \to 0} A(\alpha,\epsilon) \text{ and } A= \sup_\alpha A(\alpha),
\en
so that $A$ is an isomorphism invariant.

The authors show that for an ergodic measure-preserving action by diffeomorphisms of $\Bz^2$ on a compact manifold (or even bi-Lipschitz homeomorphisms on a compact metric space) there is a sequence $(a_n)$ for which 
$S(\alpha,n,\epsilon)$ grows at most exponentially. Then they give a cutting and stacking construction that alternates periodic and independent concatenations to produce an action which violates this growth condition and therefore cannot be realized as a smooth action preserving a Borel (not necessarily smooth) measure.

Hochman \cite{Hochman2012} used this type of argument to produce examples of infinite measure-preserving $\Bz^d$ actions for $d>1$ that are not isomorphic to actions by diffeomorphisms on a compact manifold preserving a $\sigma$-finite Borel measure. (Because of Krengel's theorem \cite{Krengel1970} stating that every $\sigma$-finite measure-preserving action of $\Bz$ has a two-set generator, there cannot be such examples when $d=1$---the given measure can be carried to a horseshoe inside a differentiable system.)

{\em Power entropy} is useful as a measure of complexity in zero-entropy systems if the relevant growth rate is polynomial. For example, if, as in Section \ref{sec:topent}, for a topological dynamical system $(X,T)$ one denotes by $S(n,\epsilon)$ the maximum possible cardinality of an $(n,\epsilon)$-separated set, one may define
\be
\begin{aligned}
h^+_\text{pow}(X,T) &= \lim_{\epsilon \to 0}\limsup_{n \to \infty}\frac{\log S(n,\epsilon)}{\log n} \text{ and }\\
h^-_\text{pow}(X,T) &= \lim_{\epsilon \to 0}\liminf_{n \to \infty}\frac{\log S(n,\epsilon)}{\log n},
\end{aligned}
\en
and define $h_\text{pow}$ to be their common value when they are equal.
In this case
\be
S(n,\epsilon) \sim n^{h_\text{pow}} \text{ for small } \epsilon.
	\en
	
	If one replaces in the preceding definition the Bowen-Dinaburg metric
	\be
	d_n(x,y)=\sup \{d(T^kx,T^ky): k=0,1,\dots,n-1\}
	\en
	($d$ denotes the metric on $X$) by the Hamming metric
	\be
\overline{d}_n(x,y)=\frac{1}{n}\sum_{k=0}^{n-1}d(T^kx,T^ky),
	\en
	the resulting entropies are called {\em modified power entropies}. The recent paper \cite{GrogerJager2015} shows the following (see also Section \ref{sec:amorphic}).
	\begin{enumerate}
		\item Neither power entropy nor modified power entropy satisfies a variational principle with respect to any real-valued isomorphism invariant defined for measure-preserving systems.
		\item Recall that a minimal system $(X,T)$ is called {\em almost automorphic} if it is an almost one-to-one extension of a minimal equicontinuous system $(Y,S)$: there is a factor map $\pi:X \to Y$ such that $\Omega=\{y\in Y: \card \pi^{-1}\{y\}=1\}$ is nonempty (in which case it is residual \cite{Veech1970}---see Exercise \ref{exer:residual}). Such a system $(X,T)$ has a unique invariant measure, $\mu$. If $\mu(\Omega)=1$, the extension is called {\em regular}. It is proved in \cite{GrogerJager2015} that the modified power entropy cannot tell the difference between equicontinuous and almost automorphic systems, since it is $0$ for both---and in fact $\sup_n\overline{S}(n,\epsilon)< \infty$ for all $\epsilon$. 
		\item If $(X,T)$ is almost automorphic (but not equicontinuous), then there is $\epsilon>0$ such that $\sup_nS(n,\epsilon)=\infty$. (Cf. Theorem \ref{th:BHM}.) Consequently ``power entropy" redefined for some other growth scale {\em can} distinguish between equicontinuous and almost automorphic systems. 
	\end{enumerate}
	
	\begin{exer}\label{exer:residual}
		Prove the statement from above that if $\pi:X \to Y$ is a factor map between minimal systems for which there exists a singleton fiber, then $\{ y \in Y: \card \pi^{-1}y =1\}$ is residual. ({\em Hint}: Show that the function $g(y)=\sup \{d_X(x,x'):\pi x=\pi x'\}$ is upper semicontinuous on $Y$.)
	\end{exer}

\section{Entropy dimension}\label{sec:entropydimension}
 In 1997 de Carvalho \cite{DeCarvalho1997} proposed an invariant for topological dynamical systems that is possibly finer than topological entropy. For a topological dynamical system $(X,T)$,  define
 \be
 d_\text{top}(X,T)= \inf\{s>0: \sup_\cu \limsup_{n \to \infty} \frac{1}{n^s}\log N(\cu_0^{n-1})=0\},
 \en
 the supremum being taken over all open covers $\cu$ of $X$. 
 Recall that $\cu_0^{n-1}=\cu \vee T^{-1}\cu \vee \dots \vee T^{-n+1}\cu$, and $N(\cu)$ denotes the minimal number of elements in any subcover of $\cu$. 
 
 The author also defines a version for measure-theoretic systems $\sys$:
 \be
  d_\text{mt}(X,T)= \inf\{s>0: \sup_\alpha \limsup_{n \to \infty} \frac{1}{n^s} H_\mu(\alpha_0^{n-1})=0\},
  \en
  the supremum being taken over all finite measurable partitions $\alpha$ of $X$. 
  
 The author shows that if the topological entropy is finite, then the topological entropy dimension is less than or equal to $1$; while if the topological entropy is finite and positive, then the topological  entropy dimension equals 1. Thus $d_\text{top}(X,T)$ is of interest mainly for zero-entropy systems and could perhaps be used to show that some zero-entropy systems are not topologically conjugate. 
 
 Related definitions were given by Dou, Huang, and Park in 2011 \cite{DouHuangPark2011}. The {\em upper entropy dimension} of $(X,T)$ with respect to an open cover is 
\be
\begin{aligned}
\overline{D}(\cu)&=\inf\{s\geq 0:  \limsup_{n \to \infty} \frac{1}{n^s}\log N(\cu_0^{n-1})=0\}\\
&=\sup \{s\geq 0:  \limsup_{n \to \infty} \frac{1}{n^s}\log N(\cu_0^{n-1})=\infty\}.
\end{aligned}
\en
The {\em lower entropy dimension} of $(X,T)$ with respect to an open cover $\cu$ is defined by replacing $\limsup$ by $\liminf$. The {\em upper and lower entropy dimensions} of $(X,T)$ are defined to be 
\be
\overline{D}(X,T)=\sup_\cu \overline{D}(\cu) \text{ and } \underline{D}(X,T)=\sup_\cu \underline{D}(\cu),
\en
respectively. These are invariants of topological conjugacy. When they are equal, they are denoted by $D(X,T)$ and called the {\em entropy dimension} of $(X,T)$. If $\htop (X,T)>0$, then again $D(X,T)=1$.

 Ferenczi and Park \cite{FerencziPark2007} defined another concept of measure-theoretic entropy dimension. Let $\sys$ be a measure-preserving system, $\alpha$ a finite measurable partition of $X$, and $\epsilon >0$. The 
$(\overline{d},n,\epsilon)$-ball around a point $x \in X$ is 
\be
B(x,n,\epsilon)=\{ y \in X: \overline{d}(\alpha_0^{n-1}(y),\alpha_0^{n-1}(x))<\epsilon\}, \quad\text{and}
\en
$K(n,\epsilon)$ is the minimum cardinality of a set of $(\overline{d},n,\epsilon)$-balls that covers a set of measure at least $1-\epsilon$.
Because of the Shannon-McMillan-Breiman Theorem, the entropy of the system coded by the partition $\alpha$ is given by
\be
h_\mu(\alpha,T)=\lim_{\epsilon \to 0^+} \lim _{n \to \infty} \frac{\log K(n,\epsilon)}{n};
\en
see \cite{Katok1980} for the statement in metric spaces.

Thinking in terms of dimension theory suggests the definitions
\be
\begin{gathered}
\overline{D}(\alpha,\epsilon)= \sup\{s\in [0,1]: \limsup_{n \to \infty} \frac{\log K(n,\epsilon)}{n^s} >0\},\\
\overline{D}(\alpha)=\lim_{\epsilon \to 0} \overline{D}(\alpha,\epsilon), \text{ and } \overline{D}\sys=\sup_\alpha \overline{D}(\alpha).
\end{gathered}
\en
The versions with $\underline{D}$ use $\liminf$ instead of $\limsup$. 

$\overline{D}$ and $\underline{D}$ are called the {\em upper and lower entropy dimensions} of $\sys$ If they are equal, they are denoted by $D\sys$ and called the {\em entropy dimension}. The definitions extend easily to measure-preserving actions of other groups, especially $\Bz^d$. 

Ahn, Dou, and Park \cite{AhnDouPark2010} showed that while the topological entropy dimension of a topological dynamical system $(X,T)$ always dominates its measure-theoretic entropy dimension for any invariant measure, the variational principle does not hold: For each $t \in (0,1)$ they construct a uniquely ergodic system $(X,T)$ with the measure-theoretic entropy dimension $D\sys=0$ for the unique invariant measure $\mu$, while the topological entropy dimension $D(X,T)=t$.

See \cite{KuangChengMaLi2014} and its references for some further developments about entropy dimensions.

\section{Permutation entropy}
The permutation entropy of a real-valued time series was introduced in \cite{BandtPompe2002} and studied in the dynamical setting, especially for interval maps, in \cite{BandtKellerPompe2002}. 
It may be regarded as a measure of the amount of information obtained by observing the ordering of an initial orbit piece, and thus of the complexity of the orderings of orbits (or sequences of measurements). 
Instead of counting the number of $n$-blocks that arise from coding orbits with respect to a partition, one counts the number of different order patterns that can appear (or computes the Shannon entropy, with respect to an invariant measure, of the partition according to order patterns). 

Let $T: X \to X$ be a continuous map on a compact interval $X \subset \Br$. For a fixed $n \in \Bn$ one partitions $X$ according to the order patterns of the first $n$ points in the orbit as follows. 
Denote $x_0=T^kx$ for $k=0,1,\dots,n-1$. We define a permutation $\tau (x)=(k_1k_2\dots  k_{n-1})$  of $(0,1,\dots,n-1)$ by 
\be
\begin{gathered}
 x_{k_1} \leq x_{k_2} \leq \dots \leq x_{k_{n-1}} \\
 	\text{and, to break ties, } \tau(i)<\tau(i+1) \text{ in case } x_i=x_{i+1}.
 	\end{gathered}
\en
Denote by $\rho_n$ the partition of $X$ into the nonempty sets $X_\tau$ among all the $n!$ permutations $\tau$. 

Let $\mu$ be an invariant probability measure for $(X,T)$.
The {\em topological and measure-theoretic permutation entropies of order $n$ are defined to be}
\be\begin{aligned}
	\htop^P(n)&=\frac{1}{n-1}\log \card \rho_n \quad\text{ and }\\
	h_\mu^P(n)&=\frac{1}{n-1}\sum_{A \in \rho_n}-\mu(A)\log\mu(A).
\end{aligned}
\en

\begin{thm} \textnormal{\cite{BandtKellerPompe2002}} If $T: X\to X$ is a piecewise monotone map on an interval $X \subset \Br$, then the limits of the permutation entropies of order $n$ exist as $n$ tends to infinity and equal the usual entropies:
	\be
	\lim_{n \to \infty} \htop^P(n)=\htop \quad\text{ and }\quad \lim_{n \to \infty} h_\mu^P(n)=h_\mu(X,T).
	\en
\end{thm}
		This theorem, along with its extensions mentioned just below, is analogous to Theorems \ref{thm:PWTop} and \ref{thm:PWErg}, providing further evidence for a topological version of Theorem \ref{thm:OW}.
	
	The definitions extend to arbitrary systems once one assigns to the elements of a finite partition distinct real numbers or fixes a real or vector-valued function (observable) on $X$ to follow along orbits. To obtain complete information (and reach the ordinary entropy) one should consider a generating sequence of partitions or observables. See \cite{Amigo2012,Keller2012,KellerU2012,HarunaNakajima2011} and their references. 
	
	The assignment of order patterns to strings from a totally ordered set is a potentially many-to-one function. For example, the observation strings $321233$ and $421333$ both determine the permutation $\tau = (324561)$. Haruma and Nakajima \cite{HarunaNakajima2011} clarified and extended previous results by showing (by counting) that for finite-state stationary ergodic processes the permutation entropy coincides with the Kolmogorov-Sinai entropy.

	 For a recent survey of permutation entropies and their relation to ordinary entropies, see \cite{KellerMS2015}. These quantities may be useful in numerical studies of actual systems, since they can be calculated from observed orbits, their behavior as a parameter determining the system changes may reflect the actual entropy even if slow to approach the limit, and since one observes only order relations they may reduce computation time.

\section{Independence entropy}

Louidor, Marcus, and Pavlov \cite{louidor2013independence} introduce a measure of how freely one may change symbols at various places in sequences or higher-dimensional configurations in a symbolic dynamical system and still remain in the system. Part of the motivation comes from practical questions in coding, where one may want to insert check bits without violating desirable constraints on the structure of signals, and another part from questions about the entropy of higher-dimensional subshifts, even shifts of finite type. 

We give the definitions for a one-dimensional subshift $(X,\sigma)$ over a finite alphabet $\ca$; the extensions to $\Bz^d$ subshifts are immediate. The set of nonempty subsets of $\ca$ is denoted by $\hat \ca$. For a configuration $\hat x \in \hat\ca^\Bz$, the {\em set of fillings} of $\hat x$ is 
\be
\Phi(\hat x) = \{x \in \ca^\Bz : \text{ for all } k \in \Bz, x_k \in \hat x_k\},
\en
and the {\em multi-choice shift corresponding to $(X,\sigma)$} is 
\be
\hat X=\{\hat x\in \hat\ca^\Bz : \Phi(\hat x) \subset X\}.
	\en
	Thus $\hat X$ consists of all the sequences on $\hat \ca$ that display the freedom to flex sequences in $X$: for those $k \in \Bz$ for which  $\hat x_k$ is not a singleton, we may freely choose symbols from the sets $\hat x_k$ and generate sequences that are always in the original subshift $X$.
	
	\begin{exer}
		Describe $\hat X$ when $X$ is the golden mean shift of finite type (no adjacent $1$s).
	\end{exer}
	
\begin{exer}
	Show that if $X$ is a shift of finite type, then so is $\hat X$. 
\end{exer}

It can be shown that in dimension $d=1$, if $X$ is sofic, then so is $\hat X$ \cite{PooMarcus2006}; but in higher dimensions the question is open. 

To define independence entropy, we consider finite blocks $\hat B$ on the alphabet $\hat \ca$, considered as elements of $\hat \ca ^{|\hat B|}$, and their fillings
\be
\Phi(\hat B) = \{ x \in \ca^{|\hat B|} : x_k \in \hat x_k \text{ for all } k=0,1,\dots, |\hat B|-1\},
\en
and define the {\em independence entropy of $(X,\sigma)$} to be the exponential growth rate 
\be
h_{\text{ind}}(X,\sigma) = \lim_{n \to \infty} \frac{\log\max\{\card \Phi(\hat B): \hat B \in \cl_n(\hat X,\sigma)\}}{n} 
\en
of the maximal  number of different fillings of $n$-blocks on the alphabet $\hat \ca$. It is a measure of the maximal independence among entries in $n$-blocks, rather than of the number of possible $n$-blocks. 

\begin{exer}
	Show that $h_\text{ind}(X) \leq \htop (X)$ for every subshift $(X,\sigma)$.
\end{exer}

\begin{exer}
	Show that the independence entropy of the golden mean shift of finite type is $\log 2/2$.
\end{exer}

Louidor and Marcus \cite{LouidorMarcus2010} described a natural way to form the {\em axial product} of $d$ one-dimensional shift spaces $X_1, \dots, X_d$ on a common alphabet $\ca$. It is the set $X_1 \otimes \dots  \otimes X_d$ of all configurations $x \in \ca^{\Bz^d}$ such that for every $i=1,\dots,d$ and every choice of $n_1,\dots , n_{i-1},n_{i+1}, \dots ,n_d \in \Bz$, the ``row"
\be 
y_k=x_{(n_1,\dots,n_{i-1},k,n_{i+1},\dots,n_d)}
\en
is in $X_i$, together with the $\Bz^d$ action by the shifts in each coordinate. If all $X_i=X$, the resulting product is the {\em $d$-fold axial power} $X^{\otimes d}$ of $X$.

\begin{thm}\textnormal{\cite{louidor2013independence}} For one-dimensional subshifts $(X,\sigma)$, 
	\begin{enumerate}
		\item
	 $h_\text{ind}(X^{\otimes d})=h_\text{ind}(X)$
\item  $\htop (X^{\otimes d})$ is nonincreasing in $d$ and hence has a limit, $h_\infty(X,\sigma) \geq h_\text{ind}(X,\sigma)$.
\end{enumerate}
\end{thm}

\begin{thm}\textnormal{\cite{MeyerovitchPavlov2014}} For every one-dimensional sushift $(X,\sigma)$, $h_\infty(X,\sigma) = h_\text{ind}(X,\sigma)$.
\end{thm}

\section{Sofic, Rokhlin, na\"{i}ve entropies}\label{sec:soficent}

Lewis Bowen \cite{Bowen2010a,Bowen2010b,Bowen2012Amenable} investigated the possibility of extending the Kolmogorov-Ornstein results (two Bernoulli shifts are measure-theoretically isomorphic if and only if they have the same entropy) to possibly infinite alphabets and possibly nonamenable acting groups. Let $(\Omega,\cf,P)$ be a standard Borel probability space and $G$ a countable discrete group. $G$ acts on the product probability space $(\Omega,\cf,P)^G$ by $gx(h)=x(g^{-1}h)$ for all $x \in \Omega^G$ and $g,h \in G$. First one needs a definition of entropy. If there is a finite or countable subset $\Omega_0 \subset \Omega$ of full measure, then the entropy of the measure $P$ is defined to be
\be
H(P)= -\sum_{\omega \in \Omega_0} P(\omega)\log P(\omega),
\en
and otherwise $H(P)=\infty$. The theorems of Kolmogorov and Ornstein state that two Bernoulli shifts $(\Omega,\cf,P)^\Bz$ and $(\Omega',\cf',P')^\Bz$ are measure-theoretically isomorphic if and only if $H(P)=H(P')$. 

Ornstein and Weiss \cite{OrnsteinWeiss1987} extended these results to countably infinite amenable groups. They also noted that there were some potential problems for nonamenable groups, by showing that for Bernoulli actions of the free group on two generators the 2-shift factors onto the 4-shift (each with the product measure of its uniform measure). 
 
L. Bowen  \cite{Bowen2010b} gave a definition of entropy for actions by sofic groups 
(defined by Gromov \cite{Gromov1999Surjunctive} and Weiss \cite{Weiss2000SoficGroups})
and showed that it is an isomorphism invariant for Bernoulli actions of sofic groups for countable alphabets with finite entropies, and a complete isomorphism invariant for such Bernoulli actions by countably infinite linear groups (groups of invertible matrices with entries from a field). In \cite{Bowen2010a}, using his $f$-invariant (a special case of the sofic entropy), he showed that two Bernoulli actions by a free group on a finite number of generators on countable alphabets of finite entropy are isomorphic if and only if the alphabet entropies are equal. 

A {\em sofic group} \cite{Gromov1999Surjunctive,Weiss2000SoficGroups} is a countable group $G$ that can be approximated by finitely many permutations acting on subsets, as described in the following. (Recall that an amenable group is one that can be approximated by nearly invariant subsets). 
For $\epsilon >0$ and $F \subset G$, a map $\phi: G \to S_n=$ the group of permutations of $\{1,\dots,n\}$ is called an {\em $(F,\epsilon)$-approximation to $G$} if the set 
\be
\begin{gathered}
V(F)=\{v \in \{1,\dots,n\}: \phi(f_1)\phi(f_2)v=\phi(f_1f_2)v \text{ for all } f_1,f_2 \in F \\
\text{ and } f_1 \neq f_2 \text{ implies } \phi(f_1)v \neq \phi(f_2)v\}
\end{gathered}
\en
satisfies 
\be
\card V(F) \geq (1-\epsilon) n.
\en
Then $G$ is defined to be {\em sofic} if there is a sequence of $(F_k,\epsilon_k)$-approximations to $G$ such that $\epsilon_k \to 0$, $F_k \subset F_{k+1}$ for all $k$, and $\cup_{k=1}^\infty F_k =G$. 
Amenable groups and residually finite groups (the intersection of all the finite index nomal subgroups is trivial) are sofic. It seems to be an open question whether every countable group is sofic. 

The sofic entropy of a measure-preserving action of a sofic group $G$ on a probability space $(X,\cb,\mu)$ is defined with respect to a fixed sofic approximation $\Phi=\{\phi_k\}$ of $G$ and a generating finite or countable ordered measurable partition $\alpha$ of $X$. According to L. Bowen, it is a measure of the exponential growth rate of the number of partitions of $\{1,\dots,n_k\}$ that approximate $\alpha$. A main theorem is that the definition does not depend on the choice of the partition $\alpha$.

Let $\alpha=\{A_1,A_2,\dots\}$ be an ordered finite measurable partition of $X$, $\phi:G \to S_n$ a map, 
$\beta=\{B_1,B_2,\dots\}$ a partition of $\{1,\dots,n\}$, $\nu$ the uniform probability measure on $\{1,\dots,n\}$, and $F \subset G$ a finite subset. For any function $g:F \to \Bn$, let
\be
A_g = \cap_{f \in F} fA_{g(f)} \quad\text{ and } B_g=\cap_{f \in F} \phi(f)B_{g(f)}.
\en
Define the {\em $F$-distance} between $\alpha$ and $\beta$ to be
\be
d_F(\alpha, \beta) = \sum_{g:F \to \Bn} |\mu(A_g) - \nu(B_g)|.
\en
Then for $\epsilon >0$ $AP(\phi,\alpha;F,\epsilon)$ 
is defined to be the set of all  ordered partitions $\beta$ of $\{1,\dots,n\}$ with the same number of atoms as $\alpha$ for which $d_F(\alpha,\beta)\leq \epsilon$.

Fix the sofic approximation $\Phi=\{\phi_k\}$ of $G$. For a finite ordered measurable partition $\alpha$ of $X$, $\epsilon >0$, and finite $F \subset G$, define
\be
\begin{aligned}
H(\Phi,\alpha;F,\epsilon) &= \limsup_{k \to \infty} \frac{1}{n_k} \log \card AP(\phi_k,\alpha;F,\epsilon),\\
H(\Phi,\alpha;F)&=\lim_{\epsilon \to 0}H(\Phi,\alpha;F,\epsilon),\\
h(\Phi,\alpha)&=\inf_{\text{finite }F\subset G} H(\Phi,\alpha;F).
\end{aligned}
\en
The definitions extend also to countable partitions $\alpha$.

\begin{thm}	\textnormal{\cite{Bowen2010b}}
 If $G$ is a countable sofic group with sofic approximation $\Phi=\{\phi_k\}$ acting by measure-preserving transformations on a probability space $(X,\cb,\mu)$, then $h(\Phi,\alpha)=h(\Phi,\beta)$ for any two generating partitions $\alpha, \beta$ of $X$.
\end{thm}
If $G,(X,\cb,\mu),\Phi$ are as in the Theorem and there is a generating partition $\alpha$ (the smallest $\sigma$-algebra containing all sets $g\alpha, g \in G$ is, up to sets of measure $0$, the full $\sigma$-algebra $\cb$) with $H_\mu(\alpha)<\infty$, then $h_\text{sofic}(X,\cb,\mu,G,\Phi)$ is defined to be $h(\Phi,\alpha$); otherwise it is undefined.
In principle $h_\text{sofic}(X,\cb,\mu,G,\Phi) \in [-\infty,\infty]$. Notice that there is not a single sofic entropy of the system, but a family depending on the sofic approximation $\Phi$ to $G$. 

L. Bowen has also proved \cite{Bowen2012Amenable} that if $G$ is amenable, then for every sofic approximation $\Phi$ the sofic entropy $h_\text{sofic}(X,\cb,\mu,G,\Phi)$ coincides with the ordinary entropy \cite{OrnsteinWeiss1987}. 
Greatly generalizing the Ornstein-Weiss example mentioned above, L. Bowen also showed \cite{Bowen2011Weak} that {\em any} two nontrivial Bernoulli actions by a countable group that contains a nonabelian free subgroup are {\em weakly isomorphic}---each is a measure-preserving factor of the other.

Kerr and Li (see \cite{KerrLi2013} and its references) have studied sofic entropy for both measure-preserving and topological sofic group actions as well as the relation of a kind of combinatorial independence with mixing properties and positive entropy. (Maybe there is a connection with the ideas in \cite{LouidorMarcus2010}?)

P. Burton \cite{Burton2015} picked up on a suggestion of L. Bowen and pursued a definition of entropy for actions of nonamenable groups that at first seems mybe useless, because for such group actions it is always either $0$ or $\infty$. However, the quantity does thus serve to divide systems into two disjoint classes, and its relation to sofic entropy is of interest. 
There are both measure-theoretic (already defined by L. Bowen) and topological (due to P. Burton) versions.

Let $(X,\cb,\mu,G)$ be a measure-preserving system, with $(X,\cb,\mu)$ a standard probability space and $G$ a countable discrete group. 
For a finite measurable partition $\alpha$ of $X$ and finite $F \subset G$, $\alpha^F = \vee_{g \in F} (g\alpha)$.
Then one may define
\be
\begin{aligned}
h_\text{nv}(X,\mu,G,\alpha) &= \inf \left\{ \frac{H_\mu (\alpha^F)}{\card F}:F \text{ a finite nonempty subset of }G\right\} \quad\text{and}\\
h_\text{nv}(X,\mu,G) &= \sup \{h_\text{nv}(X,\mu,G,\alpha): \alpha \text{ a finite partition of } X\}.
\end{aligned}
\en

Similarly in the topological case. If $G$ is a countable group acting by homeomorphisms on a compact metric space $X$, then for an open cover $\cu$ of $X$ and finite $F \subset G$ one defines $\cu^F=\vee_{g \in F}(g\cu)$,
\be
\begin{aligned}
	h_\text{tnv}(X,G,\cu)&= \inf\left\{ \frac{\log N(\cu^F)}{\card F}:  F \text{ a finite nonempty subset of }G\right\} \quad\text{and}\\
	h_\text{tnv}(X,G)&=  \sup \{h_\text{tnv}(X,G,\cu): \cu \text{ an open cover of } X\}.
\end{aligned}
\en
L. Bowen proved that if $G$ is nonamenable, then $h_\text{nv}(X,\mu,G)$ is always either $0$ or $
\infty$, and P. Burton noted that for $G$ nonamenable also always $h_\text{tnv}(X,G)$ is $0$ or $\infty$.

Verifying conjectures of L. Bowen, P. Burton proved that for finitely generated $G$, if $h_\text{tnv}(X,G)=0$, then the topological sofic entropy \cite{KerrLi2011} of $(X,G)$ is nonnegative for every sofic approximation $\Phi$ for $G$. The analogous statement in the measure-preserving case has been proved by Abert, Austin, Seward, and Weiss (see \cite{Burton2015}).

Seward \cite{Seward2015a,Seward2015b} defined the {\em Rokhlin entropy} of a measure-preserving system $(X,\cb,\mu,G)$, with $(X,\cb,\mu)$ a standard probability space and $G$ a countably infinite group, to be
\be
h_\text{Rok}(X,\cb,\mu,G) = \inf\{H_\mu(\alpha): \alpha \text{ a countable Borel generating partition}\}.
\en
 Rokhlin \cite{Rohlin1967} proved that for a measure-preserving $\Bz$ action the Rokhlin entropy coincides with the ordinary Kolmogorov-Sinai entropy: $h_\text{Rok}(X,\cb,\mu,G)=h(X,\cb,\mu,\Bz)$. Now it is known that the Rokhlin entropy and ordinary entropy coincide for all free ergodic measure-preserving actions of amenable groups. 
With this definition Seward was able 
 to extend Krieger's finite generator theorem to actions of arbitrary countable groups.
 \begin{thm}\textnormal{\cite{Seward2015a,Seward2015b}} Let $G$ be a countably infinite group, suppose that the system $(X,\cb,\mu,G)$, with $(X,\cb,\mu)$ a non-atomic standard probability space, is ergodic, and $p=(p_i)$ is a finite or countable probability vector. If 
 	\be
 	h_\text{Rok}(X,\cb,\mu,G) < H(p) = \sum -p_i \log p_i,
 	\en
 	then there is a generating partition $\alpha=\{A_i\}$ with $\mu(A_i)=p_i$ for all $i$. 
 	In particular, if $h_\text{Rok}(X,\cb,\mu,G) < \log k$ for some integer $k$, then there is a generator $\alpha$ with $k$ elements, and $(X,\cb,\mu,G)$ embeds in the $k$-shift $(\{0,1,\dots,k-1\}^G,G)$.
 \end{thm}

\section {Kolmogorov complexity}
The {\em (Kolmogorov) complexity
} $K(w)$ of a finite sequence $w$ on a finite alphabet is defined
to be the length of the shortest program that when input to a fixed
universal Turing machine produces output $w$ (or at least a coding
of $w$ by a block of 0's and 1's). For a topological dynamical system
$(X,T)$, open cover ${\mathcal U}=\{ U_0,\ldots ,U_{r-1}
\} $ of $X,$  $x\in X$, and $n\geq 1$, we consider the set of codings of the initial $n$ points in the orbit of $x$ according to the cover $\mathcal U$:  let ${\mathcal C}(x,n)=$ the
set of $n$-blocks $w$ on $\{ 0,\ldots ,r-1\} $
such that $T^jx\in U_{w_j},$ $j=0,\ldots ,n.$ Then we define the
{\em upper and lower complexity of the orbit}  of a point
$x\in X$ to be
\begin{equation}
 K^+(x,T)=\sup_{{\mathcal U}} \limsup_{n\rightarrow \infty }
\min \{{\frac{K(w)}{n}}:
w\in {\mathcal C}(x,n)\}
\end{equation}
and
\begin{equation}
 K^-(x,T)=\sup_{{\mathcal U}} \liminf_{n\rightarrow \infty }
\min \{{\frac{K(w)}{n}}:
w\in {\mathcal C}(x,n)\}.
\end{equation}

\begin{thm}[Brudno \cite{Brudno1982}, White \cite{White1991Thesis,White1993}] If $\mu $ is an ergodic invariant
	measure on $(X,T),$ then
	\begin{equation}
	K^+(x,T)=K^-(x,T)=h_\mu (X,T) \text{ a.e. }d\mu (x).
	\end{equation}
\end{thm}

Brudno showed that 
$K^+(x,T)=h_\mu(X,T)$ a.e. for every ergodic measure $\mu$ on $(X,T)$. He also noted that attempting to use finite measurable partitions $\alpha$ instead of open covers,
\be
\sup_\alpha \limsup_{n \to \infty} \frac{K(\alpha_0^{n-1}(x))}{n},
\en
gives $\infty$ for all nonperiodic points $x \in X$.

White also explains how ingredients of the proofs, some going back to Brudno and even Kolmogorov, produce a universal coding scheme: an effective (achievable by a Turing machine) algorithm for coding any finite-state ergodic stationary source that almost surely achieves a long-term data compression rate equal to the entropy of the source (the best possible). Lempel-Ziv, Wyner-Ziv, Ornstein-Shields, etc. have also produced universal coding schemes.  

\chapter{Counting patterns}

  A {\em complexity function} $p_x(n)$ counts the number of patterns of ``size" $n$ that appear in an object $x$ under investigation. One of the simplest situations (one might suppose) is that of a one-dimensional sequence $u$ on a finite alphabet $\ca$.
If $u$ is a sequence or bisequence, the \emph{complexity function} of $u$, denoted $p_u$, maps $n$ to the number of blocks of length $n$ that appear in $u$.
If $X$ is a subshift, then $p_X(n)$ is the number of blocks of length $n$ that appear in $\cl(X).$
In higher-dimensional symbolic dynamical systems one may count the number of configurations seen in rectangular regions, and in tilings one may count the number of patches of tiles of a fixed size that are equivalent under translations, or, if preferred, under translations and rotations. The asymptotic exponential growth rate of the complexity function,
\be
\limsup_{n \to \infty} \frac{\log p_x(n)}{n},
\en
is a single number that measures the complexity of $x$ in one sense, while the function $p_x$ itself is a precise measurement of how the complexity or richness of the object grows with size. There is a huge literature on complexity functions of various kinds for various structures; see for example \cite{Lothaire,Fogg,Berthe,allouche94, Ferenczi1999}. Here we look at a few representative examples.

\begin{exer}\label{exer:compequiv} \cite{Ferenczi1999} Show that the growth rate of the complexity function $p_X$ of a subshift $(X,\sigma)$ is an invariant of topological conjugacy by proving that if $(X,\sigma)$ and $(Y,\sigma)$ are topologically conjugate subshifts on finite alphabets, then there is a constant $c$ such that 
	\be
	p_X(n-c) \leq p_Y(n) \leq p_X(n+c) \text{ for all } n > c.
		\en
		\end{exer}

\section{The complexity function in one-dimensional symbolic dynamics}\label{sec:complexityfunction}
 Let $u$ be a one or two-sided infinite sequence on a finite alphabet $\ca$, and let $p_u(n)$ denote the number of $n$-blocks in $u$.
Since every block continues in at least one direction, $p_u(n+1) \geq p_u(n)$ for all $n$.

\begin{exer}
	Find the complexity functions of the bisequence $u=\dots 121232121\dots$ and the one-sided sequence $v=321212121\dots$.
\end{exer}

Hedlund and Morse \cite[Theorems 7.3 and 7.4]{MorseHedlund1938SymbolicDynamics} showed that a two-sided sequence $u$ is periodic if and only if there is a $k$ such that $p_u(k+1)=p_u(k)$, equivalently if and only if there is an $n$ such that $p_u(n) \leq n$.

\begin{exer}\label{exer:evper}
	Show that for one-sided sequences $u$ the following conditions are equivalent:
	(1) there is an $n$ such that $p_u(n) \leq n$;\\	
	(2)  there is a $k$ such that $p_u(k+1)=p_u(k)$.\\	
	(3) $u$ is eventually periodic;\\	
	(4) $p_u$ is bounded. \\
	(Hint: For (2) implies (3), note that each $k$-block in $u$ must have a unique right extension to a $(k+1)$-block, and that some $k$-block must appear at least twice in $u$.)
\end{exer}

\begin{exer}
	Show that for a two-sided sequence $u$, if there is an $n$ such that $p_u(n) \leq n$, then $u$ is periodic.
\end{exer}

\section{Sturmian sequences}
Hedlund and Morse \cite{MorseHedlund1940Sturmian} defined Sturmian sequences as those that have the smallest possible complexity among non-eventually-periodic sequences.
\begin{defn}
	A sequence $u$ is called \emph{Sturmian} if it has complexity $p_u(n) = n+1$ for all $n$.
\end{defn}
If $u$ is Sturmian, then $p_u(1) = 2$. This implies that Sturmian sequences are over a two-letter alphabet. For the duration of this discussion on Sturmian systems, we fix the alphabet $\ca = \{0,1\}$.

\begin{exer}
	\label{fibex}
	The Fibonacci substitution is defined by: 
	\begin{align*}
	\phi: 0 & \mapsto 01
	\\1 & \mapsto 0.
	\end{align*}
	The fixed point  $f = 0100101001001010010100100101...$ of the Fibonacci substitution is called the {\em Fibonacci sequence}. Show that $f$ is a Sturmian sequence.
\end{exer}

\begin{defn}
	A set $S$ of blocks is \emph{balanced} if for any pair of blocks $u$, $v$ of the same length in $S$, $||u|_1 - |v|_1| \leq 1$, where $|u|_1$ is the number of occurrences of 1 in $u$ and $|v|_1$ is the number of occurences of 1 in $v$.
\end{defn}

It immediately follows that if a sequence $u$ is balanced and not eventually periodic then it is Sturmian. This is a result of the fact that if $u$ is aperiodic, then $p_u(n) \geq n+1$ for all $n$, and if $u$ is balanced then $p_u(n) \leq n+1$ for all $n$. In fact, it can be proved that a sequence $u$ is balanced and aperiodic if and only if it is Sturmian \cite{Lothaire}. Furthermore, it immediately follows that any shift of a Sturmian sequence is also Sturmian.

Sturmian sequences also have a natural association to lines with irrational slope. To see this, we introduce the following definitions.
\begin{defn}
	Let $\alpha$ and $\beta$ be real numbers with $0 \leq \alpha, \beta \leq 1$. We define two infinite sequences $x_{\alpha,\beta}$ and ${x'}_{\alpha,\beta}$ by
	\begin{eqnarray*}
		(x_{\alpha,\beta})_n&=& \lfloor \alpha(n+1) + \beta \rfloor - \lfloor \alpha n + \beta \rfloor    \nonumber \\
		({x'}_{\alpha,\beta})_n&=& \lceil \alpha(n+1) + \beta \rceil - \lceil \alpha n + \beta \rceil    \nonumber \\
	\end{eqnarray*}
	for all $n\geq 0$.
	The sequence $x_{\alpha,\beta}$ is the \emph{lower mechanical sequence} and ${x'}_{\alpha,\beta}$ is the \emph{upper mechanical sequence} with slope $\alpha$ and intercept $\beta$.
\end{defn}

The use of the words slope and intercept in the above definitions stems from the following graphical interpretation.
Consider the line $y= \alpha x + \beta$. The points with integer coordinates that sit just below this line are $F_n = (n, \lfloor \alpha n + \beta \rfloor).$ The straight line segment connecting two consecutive points $F_n$ and $F_{n+1}$  is horizontal if  $x_{\alpha,\beta} = 0$ and diagonal if  $x_{\alpha,\beta} = 1$. Hence, the lower mechanical sequence can be considered a coding of the line $y= \alpha x + \beta$ by assigning to each line segment connecting $F_n$ and $F_{n+1}$ a 0 if the segment is horizontal and a 1 if the segment is diagonal. Similarly, the points with integer coordinates that sit just above this line are $F'_n = (n, \lceil \alpha n + \beta \rceil)$. Again, we can code the line  $y= \alpha x + \beta$ by assigning to each line segment connecting $F'_n$ and $F'_{n+1}$ a 0 if the segment is horizontal and a 1 if the segment is diagonal. This coding yields the upper mechanical sequence \cite{Lothaire}.

The sequence $x_{\alpha,\beta}$ codes the forward orbit of $\beta \in [0,1)$ under the translation $x \to x+\alpha \mod 1$ with respect to the partition $A=[0,1-\alpha), A^c=[1-\alpha,1)$: $x_{\alpha,\beta}=0$ if $\beta + n\alpha \in A$, $x_{\alpha,\beta}=1$ if $\beta +n\alpha \in A^c$; and the sequence $x'_{\alpha,\beta}$ codes the orbit of $\beta$ with respect to the partition $B=(0,1-\alpha], B^c=(1-\alpha,1]$. 

A mechanical sequence is $rational$ if the line $y= \alpha x + \beta$ has rational slope and  $irrational$ if $y= \alpha x + \beta$ has irrational slope. In \cite{Lothaire} it is proved that a sequence $u$ is Sturmian if and only if $u$ is irrational mechanical. In the following example we construct a lower mechanical sequence with irrational slope, thus producing a Sturmian sequence.

\begin{example}
	Let $\alpha = 1/\tau^2$, where $\tau = (1+ \sqrt{5})/2$ is the golden mean, and $\beta = 0$. The lower mechanical sequence $x_{\alpha,\beta}$ is constructed as follows:
	\begin{align*}
	(x_{\alpha, \beta})_0 = & \lfloor 1/\tau^2 \rfloor = 0  \\
	(x_{\alpha, \beta})_1= & \lfloor 2/\tau^2  \rfloor - \lfloor 1/\tau^2  \rfloor = 0 \\
	(x_{\alpha, \beta})_2= & \lfloor 3/\tau^2  \rfloor - \lfloor 2/\tau^2  \rfloor = 1 \\
	(x_{\alpha, \beta})_3= & \lfloor 4/\tau^2  \rfloor - \lfloor 3/\tau^2  \rfloor =  0\\
	(x_{\alpha, \beta})_4= & \lfloor 5/\tau^2  \rfloor - \lfloor 4/\tau^2  \rfloor =  0\\
	(x_{\alpha, \beta})_5= & \lfloor 6/\tau^2  \rfloor - \lfloor 5/\tau^2  \rfloor =  1 \\
	\vdots
	\end{align*}
	Further calculation shows that $x_{\alpha,\beta} = 0010010100... = 0f$. Note that a similar calculation gives ${x'}_{\alpha,\beta} = 1010010100...= 1f,$ hence the Fibonacci sequence is a shift of the lower and upper mechanical sequences with slope $1/\tau^2$ and intercept $0$.
\end{example}

\begin{exer}
	 Show that while Sturmian sequences are aperiodic, they are {\em syndetically recurrent}: every block that occurs in a Sturmian sequence occurs an infinite number of times with bounded gaps.
\end{exer}

As a result of the preceding Exercise, any block in $\cl_n(u)$ appears past the initial position and can thus be extended on the left.  Since there are $n+1$ blocks of length $n$, it must be that exactly one of them can be extended to the left in two ways.
\begin{defn}
	In a Sturmian sequence $u$, the unique block of length $n$ that can be extended to the left in two different ways is called a \emph{left special block}, and is denoted $L_n(u)$. The sequence $l(u)$ which has the $L_n(u)$'s as prefixes is called the \emph{left special sequence} or \emph{characteristic word} of $X_u^+$ \cite{Fogg, Lothaire}.
\end{defn}
In a similar fashion, we define the right special blocks of $\cl_n(u)$.
\begin{defn}
	In a Sturmian sequence $u$, the unique block of length $n$ that can be extended to the right in two different ways is called a \emph{right special block}, and is denoted $R_n(u)$. The block $R_n(u)$ is precisely the reverse of $L_n(u)$ \cite{Fogg}.
\end{defn}
We now address how to determine the left special sequence in a Sturmian system.

Since every Sturmian sequence $u$ is irrational mechanical, there is a line with irrational slope $\alpha$ associated to $u$. We use this $\alpha$ to determine the left special sequence of $X_u^+$.

Let $(d_1, d_2, ..., d_n,...)$ be a sequence of integers with $d_1 \geq 0 $ and $d_n > 0 $ for $n>1$. We associate a sequence $(s_n)_{n\geq -1}$ of blocks to this sequence by
\begin{center}
	$s_{-1} = 1$, \hspace{3mm} $s_0 = 0$, \hspace{3mm} $s_n = s^{d_n}_{n-1}s_{n-2}. $
\end{center}
The sequence $(s_n)_{n\geq -1}$ is a \emph{standard sequence}, and  $(d_1, d_2, ..., d_n,...)$  is its \emph{directive sequence}. We can then determine the left special sequence of  $X_u^+$ with the following proposition stated in \cite{Lothaire}.
\begin{prop}
	Let $\alpha = [0, 1+ d_1, d_2, ....]$ be the continued fraction expansion of an irrational $\alpha$ with $0< \alpha < 1$, and let $(s_n)$ be the standard sequence associated to $(d_1, d_2,...)$. Then every $s_n$, $n \geq 1$, is a prefix of $l$ and
	$$l = \lim_{n\to \infty} s_n.$$
	\label{lss}
\end{prop}
This is illustrated in the following two examples.

\begin{example}
	\label{fibex2}
	Let  $\alpha = 1/\tau^2$, where $\tau = (1+ \sqrt{5})/2$ is the golden mean. The continued fraction expansion of $1/\tau^2$ is $[ 0,2,\overline{1} ]$. By the above proposition $d_1 = 1, d_2= 1, \\ d_3 = 1, d_4 = 1, ... $. The standard sequence associated to  $(d_1, d_2,...)$ is constructed as follows:
	\begin{center}
		\begin{align*}
		s_1 =  & s_0^{d_1}s_{-1} =  01
		\\ s_2 = & s_1^{d_2}s_0 =   010
		\\ s_3 = & s_2^{d_3}s_1 =   01001
		\\ s_4 = & s_3^{d_4}s_2 =   01001010
		\\ \vdots
		\end{align*}
		
	\end{center}
	Continuing this process, the left special sequence of $X_u^+$, where $u$ is a coding of a line with slope $1/\tau^2$, is
	$$l = 010010100100101001 ... = f.$$ It follows that the left special sequence of $X_f^+$ is $f$.
\end{example}

\section{Episturmian sequences}

Sturmian sequences have many equivalent definitions and characterizations as well as many amazing properties. Some of the properties allow generalization to sequences on alphabets of more than two symbols. {\em Episturmian sequences} were defined in \cite{DroubayJustinPirillo2001} as follows.
The {\em right palindrome closure} of a block $B$ is the unique shortest palindrome which has $B$ as a prefix.

\begin{defn}
	A one-sided sequence $u$ is {\em standard episturmian} if the right palindrome closure of every prefix of $u$ is also a prefix of $u$. A sequence $v$ is {\em episturmian} if there is a standard episturmian sequence $u$ such that $\cl(v)=\cl(u)$ .
\end{defn}

\begin{exer}
	Prove that every Sturmian sequence is episturmian.
\end{exer}

\begin{exer}
	Prove that every episturmian sequence is syndetically recurrent.
\end{exer}

We say that an infinite sequence $u$ is {\em closed under reversals} if for each subblock $B=b_1\dots b_n$ of $u$, its {\em reversal} $B'=b_n \dots b_1$ is also a subbblock of $u$. Recall that a {\em right-special} subblock of an infinite sequence $u$ is a subblock $B$ of $u$ that has two distinct continuations: there are symbols $b,c \in \ca$ with $b \neq c$ and both $Bb$ and $Bc$ appearing in $u$ as subblocks.

\begin{thm}
	\textnormal{\cite{DroubayJustinPirillo2001}}A one-sided sequence $u$ is episturmian if and only if it is closed under reversals and for each $n \geq 1$ it has at most one right-special subbblock of length $n$.
\end{thm}

For a review of properties of episturmian sequences, including their complexity, see \cite{GlenJustin2009}.

\section{The Morse sequence}
The Morse sequence, more properly called the Prouhet-Thue-Morse sequence, is the fixed point 
\begin{center}
	$\omega  = 0110100110010110....$
\end{center}
of the substitution $0 \to 01, 1 \to 10$. 
The complexity function of the Morse sequence is more complicated than that of the Fibonacci sequence. For the Morse sequence, $p_\omega(1) =2$, $p_\omega(2) =4$, and, for $n \geq 3$, if $n=2^r+q+1$, $r \geq 0$, $0<q\leq 2^r$, then
\begin{center}
	$p_\omega(n) =
	\begin{cases}
	6(2^{r-1})+4q& \text{if } 0 < q \leq 2^{r-1} \\
	8(2^{r-1})+2q & \text{if } 2^{r-1}<q \leq 2^r.\\
	\end{cases}$
\end{center}
The complexity function of the Morse sequence is discussed in more detail in \cite[Chapter 5]{Fogg}

\section{In higher dimensions, tilings, groups, etc.}

The complexity of configurations and tilings in higher-dimensional spaces and even groups is an area of  active investigation. A central question has been the possibility of generalizing the observation of Hedlund and Morse (Exercise \ref{exer:evper}) to higher dimensions: any configuration of low enough complexity, in some sense, should be eventually periodic, in some sense. A definite conjecture in this direction was stated in 1997 by M. Nivat in a lecture in Bologna (see \cite{EpifanioKoskasMignosi2003}):
For a $d$-dimensional configuration $x: \Bz^d \to \ca$ on a finite alphabet $\ca$, define its {\em rectangular complexity function} to be the function $P_x(m_1,\dots,m_d)$ which counts the number of different $m_1 \times \dots \times m_d$ box configurations seen in $x$. The {\em Nivat Conjecture} posits that 
 if $x$ is a two-dimensional configuration on a finite alphabet for which there exist $m_1,m_2 \geq 1$ such that $P_x(m_1,m_2)\leq m_1m_2$, then $x$ is periodic: there is a nonzero vector $w \in \Bz^2$ such that $x(v+w)=x(v)$ for all $v \in \Bz^2$. 
 
 Cassaigne \cite{Cassaigne1999} characterized all two-dimensional configurations with complexity function $P_x(m_1,m_2)=m_1m_2+1$. 
 
Vuillon \cite{Vuillon1998} considered tilings of the plane generated by a cut-and-project scheme. Recall (see \cite{ArnouxBertheEiIto2001}) that Sturmian sequences code (according to the two possible tile=interval lengths) tilings of a line obtained by projecting onto it the points in the integer lattice that are closest to it along another, transverse, line. 
Vuillon formed tilings of the plane with three types of diamonds by projecting onto a plane points of the cubical lattice in $\Bz^3$ and proved that the number of different $m \times n$ parallelograms is $mn+m+n$.

Berth\'{e} and Vuillon \cite{BertheVuillon2000a} showed that these two-dimensional configurations code the $\Bz^2$ action of two translations on the circle. By applying a one-block code from the three-letter alphabet to a two-letter alphabet, they produced for each $m$ and $n$ a two-dimensional configuration  which is syndetically recurrent and is not periodic in any rational direction but has the relatively low 
 rectangular complexity function $P(m,n)=mn+n$. Two-dimensional configurations with this complexity function were characterized in \cite{BertheVuillon2000b}.

Sander and Tijdeman \cite{SanderTijdeman2000a,SanderTijdeman2000b,SanderTijdeman2002} studied the complexities of configurations $x: \Bz^d \to \{0,1\}$ in terms of the number of distinct finite configurations seen under a sampling window. Let $A
=\{a_1,\dots,a_n\}$, each $a_i \in \Bz^d$, be a fixed non-empty {\em sampling window}, and define
\be
P_x(A)=\card\{(x(v+a_1),\dots,x(v+a_n)):v \in \Bz^d\}
\en
to be the number of distinct {\em $A$-patterns} in $x$ (written here as ordered $|A|$-tuples). A natural extension of (\ref{exer:evper}) might be that if there is a nonempty set $A \subset \Bz^d$ for which $P_x(A) \leq |A|$, then $x$ must be periodic: there is a $w \in \Bz \setminus \{0\}$ such that $x(v+w)=x(v)$ for all $v \in \Bz$. 
Sander and Tijdeman proved the following.

\begin{enumerate}
	\item If $P_x(A) \leq |A|$ for some $A \subset \Bz^d$ with $|A| \leq 3$, then $x$ is periodic.
	\item In dimension 1, the observation of Hedlund and Morse generalizes from sampling windows that are intervals to arbitrary sampling windows: If $x \in \{0,1\}^\Bz$ satisfies $P_x(A) \leq |A|$ for some (non-empty) sampling window $A$, then $x$ is periodic.
	\item There are a non-periodic two-dimensional configuration $x: \Bz^2 \to \{0,1\}$ and a sampling window $A \subset\Bz^2$ of size $|A|=4$ such that $P_x(A)=4=|A|$.
	\item {\em Conjecture}: If $A \subset \Bz^2$ is a (non-empty) sampling window that is the restriction to $\Bz^2$ of a {\em convex} subset of $\Br^2$ and $x: \Bz^2 \to\{0,1\}$ satisfies $P_x(A) \leq |A|$, then $x$ is periodic.
	\item If there is a sampling window $A \subset \Bz^2$ that consists of all points in  a rectangle (with both sides parallel to the coordinate axes) with one side of length $2$, and $P_x(A) \leq |A|$, then $x$ is periodic. 
	\item In dimension 3 and above, there are rectangular box sampling windows $A$ with sides parallel to the coordinate axes and nonperiodic configurations $x$ with $P_x(A) \leq |A|$.
	\end{enumerate}
	
	The last statement above was recently improved by Cyr and Kra \cite{CyrKra2013}:  If there is a sampling window $A$ that consists of all points in  a rectangle (with both sides parallel to the coordinate axes) with one side of length $3$, and $P_x(A) \leq |A|$, then $x$ is periodic.
	
	Kari and Szabados \cite{KariSzabados2015} (see also \cite{KariSzabados2015b}) represented configurations in $\Bz^d$ as formal power series in $d$ variables with coefficients from $\ca$ and used results from algebraic geometry to study configurations in $\Bz^d$ which have low complexity in the sense that for some sampling windows $A$ they satisfy $P_x(A) \leq |A|$. They proved that in dimension two any non-periodic configuration $x$ can satisfy such an estimate for only finitely many {\em rectangular} sampling windows $A$.
	
	Epifanio, Koskas, and Mignosi \cite{EpifanioKoskasMignosi2003} made some progress on the Nivat Conjecture by showing that if $x$ is a configuration on $\Bz^2$
 for which there exist $m,n \geq 1$ such that $P_x(m,n)< mn/144$, then $x$ is periodic. The statement was improved by Quas and Zamboni \cite{QuasZamboni2004} by combinatorial and geometrical arguments to replace $1/144$ by $1/16$, and by Cyr and Kra \cite{CyrKra2015} by arguments involving subdynamics  
to replace it by $1/2$. In \cite{CyrKra2016x2x3}, Cyr and Kra give an application to Furstenberg's $ \times 2 \times 3$ problem: If $\mu$ is a probability measure on $[0,1)$ that is invariant under multiplication mod $1$ by two multiplicatively independent integers and ergodic under the joint action, then any Jewett-Krieger model for the natural extension of the joint action (as a $\mathbb Z^2$ subshift) either has complexity function $P_x(n,n)$ bounded in $n$ (for all $x$) and the subshift is finite and $\mu$ is atomic, or else 
\be
\liminf_{n \to \infty} \frac{P_x(n,n)}{n^2} \geq \frac{1}{2}.
\en

We do not give definitions of all the terminology associated with tilings and tiling dynamical systems---see for example \cite{Frank2008,Robinson2004,Solomyak1997} for background. For a tiling $x$ of $\Br^d$ that has finite local complexity, one may define its complexity function $P_x(r)$ to be the number of different patches (identical up to translation, or perhaps translation and rotation) seen in $x$ within spheres of radius $r$. In analogy with Exercise \ref{exer:compequiv} for subshifts, Frank and Sadun \cite{FrankSadun2014a} and A. Julien \cite{Julien2012} (see also \cite{Julien2010}) showed that if two minimal tiling dynamical systems are aperiodic and have finite local complexity, then their complexity functions are equivalent---within bounded multiples of each other up to bounded translations (or dilations---see the cited papers for precise statements). 

The investigation of the complexity function and the calculation or even estimation of entropy are extending to subshifts on groups (see for example \cite{Piantadosi2008,Aubrun2014}) and even on trees \cite{AubrunBeal2012,AubrunBeal2013,BanChang2015}. 

Analogues of the Nivat Conjecture for general Delaunay sets in $\Br^d$ were proved by Lagarias and Pleasants \cite{LagariasPleasants2002,LagariasPleasants2003}. Huck and Richard \cite{HuckRichard2015} estimate the pattern entropy of ``model sets" (certain point sets that result from cut and project schemes) in terms of the size of the defining window.

Durand and Rigo \cite{DurandRigo2013} proved a reformulation of Nivat's Conjecture by redefining periodicity and using a different complexity function: A subset $E \subset \Bz^d$ is ``definable in Presburg arithmetic $(\Bz;<,+)$" if and only if the number $R_E(n)$ of different blocks that appear infinitely many times in $E$ is $O(n^{d-1})$  and ``every section is definable in $(\Bz;<,+)$". We do not attempt to explain the terminology here, but just note that the subsets of $\Bn$ definable in $(\Bn;<,+)$ correspond exactly to the eventually periodic sequences, so the theorem of Durand and Rigo may be regarded as an extension of the Hedlund-Morse theorem to all dimensions.

\section{Topological complexity}\label{sec:topologicalcomplexity}
Let $(X,T)$ be a topological dynamical system. If $(X,T)$ is a subshift and $\cu$ is the time-$0$ cover (also partition) consisting of the cylinder sets determined by fixing a symbol at the origin,  
then the complexity function $p_X(n)$ (which by definition is the number of distinct $n$-blocks in all sequences in the system)
 is the minimal possible cardinality of any subcover of $\cu_0^{n-1}=\cu \vee T^{-1}\cu \vee \dots \vee T^{-n+1}\cu$;
  i.e., in this case $p_X(n)$ equals the $N(\cu_0^{n-1})$ of the definition of topological entropy (see Sections \ref{sec:complexityfunction} and \ref{sec:topent}).
Blanchard, Host, and Maass \cite{BlanchardHostMaass2000} took this as the definition of the {\em topological complexity function}: $p_\cu (n)=N(\cu_0^{n-1})=$ the minimum possible cardinality of a subcover of $\cu_0^{n-1}$.

\begin{thm}\label{th:BHM}
	\textnormal{\cite{BlanchardHostMaass2000}} A topological dynamical system is equicontinuous if and only if every finite open cover has bounded complexity function. (Cf. Sections \ref{sec:complexityfunction} and \ref{sec:mtcomplexity}.)
\end{thm}

\begin{exer}
	Discuss this theorem in relation to a Sturmian subshift and the irrational translation on $[0,1]$ of which it is an almost one-to-one extension.
\end{exer}

They also related the complexity function to concepts of mixing and chaos.

\begin{defn}
	A topological dynamical system is {\em scattering} if every covering by non-dense open sets has unbounded complexity function. It is {\em 2-scattering} if every covering by two non-dense open sets has unbounded complexity function.
\end{defn}
The following results are from \cite{BlanchardHostMaass2000}.
\begin{enumerate}
	\item Every topologically weakly mixing system is scattering.
	\item For minimal systems, 2-scattering, scattering, and topological weak mixing are equivalent.
	\item If every non-trivial closed cover $\cu$ of $X$ has complexity function satisfying $p_\cu (n) \geq n+2$ for all $n$, then $(X,T)$ is topologically weakly mixing.
  \item If $(X,T)$ has a point of equicontinuity, then there exists an open cover $\cu$ of $X$ with 
   $ p_\cu (n) \leq n+1$ for all $n$.
   \item A system is scattering if and only if its Cartesian product with every minimal system is transitive.
   \item Every scattering system is disjoint from every minimal distal system. (Recall that $(X,T)$ and $(Y,S)$ are {\em disjoint} if the only closed invariant subset of their Cartesian product that projects onto both $X$ and $Y$ is all of $X \times Y$.)
  \end{enumerate}
  
  The topological complexity of nilsystems has been studied for example in \cite{DongDonosoMaassSongYe2013,HostKraMaass2014,Qiao2015}.

\section{Low complexity, the number of ergodic measures, automorphisms}
Irrational translations on the unit interval and their generalizations to interval exchanges have natural codings as subshifts on finite alphabets which have low complexity and a small number of ergodic measures. 
A subshift $(X,\sigma)$ is said to have {\em minimal block growth} if there is a constant $c<\infty$ such that its complexity function $p_X(n)=|\cl_n(X)|$ satisfies $p_X(n)-n<c$ for all $n$. Such subshifts were studied and characterized by Coven and Hedlund \cite{CovenHedlund1973}, Coven \cite{Coven1975}, and Paul \cite{Paul1975}.

Cassaigne \cite{Cassaigne1996,Cassaigne1997} proved that a sequence $u$ satisfies $p_u(n) \leq Kn$ for some constant $K$ and all $n$ if and only if the sequence $p_u(n+1)-p_u(n)$ is bounded. Rote \cite{Rote1994} constructed sequences with $p_u(n)=2n$ for all $n$ and proved that a sequence $u$ on $\{0,1\}$ has complexity function $p_u(n)=2n$ and language closed under switching $0$ and $1$ if and only if its difference sequence $u_{n+1}-u_n \mod 2$ is Sturmian.

The coding of an exchange of $r$ intervals satisfies $p_X(n)=(r-1)n+1$ for all $n$ (generalizing the Sturmian case, when $r=2$). Veech \cite{Veech1978} and Katok  \cite{Katok1973} showed independently that an exchange of $r$ intervals has at most $\floor{r/2}$ ergodic measures. Boshernitzan \cite{Boshernitzan1985} showed that a general minimal subshift with {\em linear block growth} has a bounded number of ergodic measures. Denoting by $\mathscr E(X,\sigma)$ the set of ergodic invariant Borel probability measures on $(X,\sigma)$, his main results are as follows:
\begin{enumerate}
	\item If $\displaystyle{\liminf \frac{p_X(n)}{n}=\alpha}$, then $|\mathscr E(X,\sigma)| \leq \alpha$.
		\item If there is an integer $K \geq 3$ such that $\displaystyle{\limsup \frac{p_X(n)}{n}<K}$, then $|\mathscr E(X,\sigma)| \leq K-2$.
		\item If there is a real $\alpha \geq 2$ such that $\displaystyle{\limsup \frac{p_X(n)}{n}=\alpha}$, then $|\mathscr E(X,\sigma)| \leq \floor{\alpha} -1$.
\end{enumerate}

\begin{exer}
Show that (2) implies (3).
\end{exer}

\begin{exer}
	Show that for a (primitive, proper) substitution $\theta: A \to A^*$ ($A^*$ is the set of all blocks on the alphabet $A$) that defines a minimal subshift $(X_\theta,\sigma)$ as the orbit closure of a sequence fixed under $\theta$, the complexity function $p_X$ satisfies $p_X(n) \leq Kn$ for all $n$, where $K=\max\{|K(a)| :a \in A\}$.
\end{exer}

Cyr and Kra \cite{CyrKra2015GenericMeasures} showed that these results of Boshernitzan are sharp by constructing for each integer $d \geq 3$ a minimal subshift $(X,\sigma)$ for which
\be
\liminf \frac{p_X(n)}{n}=d, \quad \limsup \frac{p_X(n)}{n}=d+1, \text{ and } |\mathscr E(X,\sigma)|=d.
\en

Damron and Fickenscher \cite{DamronFickenscher2015} considered minimal subshifts with {\em eventually constant block growth}: there are constant $K$ and $n_0$ such that
\be
p_X(n+1) - p_X(n) =K \text{ for all } n \geq n_0, \text{ equivalently } p_X(n) = Kn + C
\text{ for all } n \geq n_0.
\en
This condition is satisfied by codings of interval exchanges but is stronger than linear block growth. They improved Boshernitzan's bound (of $K-1$) for this smaller class of subshifts by showing that if a minimal subshift $(X,\sigma)$ has eventually constant block growth with constant $K \geq 4$, then $|\mathscr E(X,\sigma)|\leq K-2$.
The proof involved careful study of the frequencies of blocks and the relation between Rauzy graphs for $n$-blocks and for $(n+1)$-blocks.

There is recent activity showing that if the complexity function $p_X(n)$ of a minimal subshift $(X,\sigma)$ grows slowly, then there are not very many {\em automorphisms} (shift-commuting homeomorphisms) of the system. This is in contrast to nontrivial mixing shifts of finite type, for which the automorphism group is huge \cite{Hedlund1969,BoyleLindRudolph1988}. 
The papers of Coven \cite{Coven1972}, Olli \cite{Olli2013}, Salo-T\"{o}rm\"{a} \cite{SaloTorma2014},  Coven-Quas-Yassawi \cite{CovenQuasYassawi}, Donoso-Durand-Maass-Petite \cite{DonosoDurandMaassPetite2015}, and Cyr-Kra \cite{CyrKra2014Lin} contain such results. In the latter two papers it is proved that if 
\be
\liminf _{n \to \infty} \frac{p_X(n)}{n} < \infty,
\en
then the automorphism group of $(X,\sigma)$ is {\em virtually $\Bz$}, meaning that its quotient by the subgoup generated by $\sigma$ is a finite group.

The recent paper of Cyr and Kra \cite{CyrKra2015auts} shows that if $p_X$ has {\em stretched exponential growth} in the sense that there is $\beta < 1/2$ such that
\be
\limsup_{n \to \infty} \frac{\log p_X(n)}{n^\beta} = 0,
\en
then the automorphism group of $(X,\sigma)$ is amenable. (The linear and quadratic complexity cases are discussed in \cite{CyrKra2014Lin, CyrKra2014Quad}).

The recent paper of Salo \cite{Salo2014} gives an example of a minimal subshift $(X,\sigma)$ for which 
$p_X(n)=O(n^{1.757})$ (subquadratic complexity), and the automorphism group is {\em not finitely generated}.

\section{Palindrome complexity}
  A {\em palindrome} is a block $B=b_1 \dots b_n$ which reads the same forwards as backwards: denoting by $B'=b_n \dots b_1$ the {\em reversal} of $B$, we have $B'=B$. For a (one or two-sided) sequence $u$ and each $n\geq 0$ we denote by $\Pal_u(n)$ the number of $n$-palindromes (palindromes of length $n$) in $u$. (We may define $\Pal_u(0)=\Pal_u(1)=1$.) For a language $\cl$, $\Pal_{\cl}(n)$ denotes the number of $n$-palindromes in $\cl$, and for a subshift $(X,\sigma)$, $\Pal_X(n)$ denotes the number of $n$-palindromes in $\cl(X)$, i.e. the number of palindromes of length $n$ found among all sequences in $X$. For any $\alpha$, denote by $\Pal_\alpha$ the sum of $\Pal_\alpha(n)$ over all $n$.

  Besides the inherent combinatorial interest of finding palindromes and estimating how many occur, such knowledge may have applications elsewhere, for example in mathematical physics, as explained in \cite{AlloucheBaakeCassaigneDamanik2003}. Let $A$ be a finite alphabet and $f:A \to \Br$ a one-to-one ``potential function". Let $u$ be a sequence on $A$. Then the {\em one-dimensional discrete Schr\"{o}dinger operator} $H_u$ acting on $l^2(\Bz$) is defined by
  \be
  (H_u \phi)(n)  = \phi(n-1)+\phi(n+1) + f(u_n)\phi(n) \text{ for all } n \in \Bz.
  \en
  It is thought that the nature of the spectrum of $H_u$ can indicate behavior of a material (perhaps a quasicrystal) described by the sequence $u$. Absolutely continuous spectrum (with respect to Lebesgue measure) might indicate behavior similar to a conductor, pure point spectrum might indicate insulating behavior, and singular continuous spectrum might be more interesting than either of the others.

  \begin{thm}\textnormal{\cite{HofKnillSimon1995}}
 Let $u$ be a sequence on a finite alphabet whose orbit closure $(\overline{\mathscr O(u)}, \sigma)$ is strictly ergodic and infinite and which contains arbitrarily long palindromes, i.e. $\Pal_u(n)$ is not eventually $0$. Then for uncountably many $x \in \overline{\mathscr O(u)}$, the operator $H_x$ has purely singular continuous spectrum.
    \end{thm}

    \begin{exer}
    	Show that the fixed points of the Fibonacci substitution $0 \to 01, 1 \to 0$ and Morse substitution $
 0 \to 01, 1 \to 10$ contain arbitrarily long palindromes.
    \end{exer}

It is interesting that Sturmian sequences can be characterized not just by their complexity function (it is as small as possible for a sequence that is not eventually periodic: $p_u(n)=n+1$ for all $n$),  but also by their palindrome complexity functions. Droubay \cite{Droubay1995} showed that the fixed point of the Fibonacci substitution satisfies the conclusion of the following theorem of Droubay and Pirillo \cite{DroubayPirillo1999}. A generalization to two dimensions was given in \cite{BertheVuillon2001}.

\begin{thm}\textnormal{\cite{DroubayPirillo1999}}
	A one-sided sequence $u$ on a finite alphabet is Sturmian if and only if $\Pal_u(n)=1$ for all even $n \geq 2$ and $\Pal_u(n)=2$ for all odd $n$.
\end{thm}

Damanik and Zare \cite{DamanikZare2000} studied the palindorome complexity of fixed points of primitive substitutions. Recall that the orbit closure of any such sequence is strictly ergodic, so that every block $B$ has a uniform frequency of occurrence $\Freq_u(B)$ in $u$, which equals the measure of the cylinder set consisting of all sequences that have $B$ occurring at the origin.

\begin{thm}\textnormal{\cite{DamanikZare2000}}
	Let $u$ be a one-sided sequence that is a fixed point of a primitive substitution on a finite alphabet. Then
	\begin{enumerate}
		\item  $\Pal_u(n)$ is bounded;
		\item There are constants $c_1$ and $c_2$ such that if $B$ is a palindrome of length $n$ in $u$, then
		\be
		\frac{c_1}{n} \leq \Freq_u(B) \leq \frac{c_2}{n} \text{ for all } n.
		\en
			\end{enumerate}
\end{thm}

The palindrome complexity function is studied for many examples in \cite{AlloucheBaakeCassaigneDamanik2003}, where it is also proved that if a sequence $u$ has linear complexity function (there is a constant $C$ such that $p_u(n)\leq Cn$ for all $n$), then $\Pal_u$ is bounded.
\begin{thm}\textnormal{\cite{AlloucheBaakeCassaigneDamanik2003}}
	Let $u$ be a one-sided sequence on a finite alphabet that is not eventually periodic. Then for all $n  \geq 1$,
	\be
	\Pal_u(n) < \frac{1}{n} 16 p_u\left( n+\floor{\frac{n}{4}}\right).
	\en
\end{thm}

These results were generalized in \cite{BalaziMasakovaPelantova2007} as follows. Recall that for a sequence $u$, $\cl(u)$ denotes the family of all subblocks of $u$. A formal language $\cl$ is called {\em closed under reversal} if $\cl' \subset \cl$.
\begin{thm}\textnormal{\cite{BalaziMasakovaPelantova2007}}
Let $u$ be a syndetically recurrent sequence on a finite alphabet.
\begin{enumerate}
\item If $\cl(u)$ is not closed under reversal, then $u$ does not contain arbitrarily long palindromes: $\Pal_u(n)=0$ for all sufficiently large $n$.
\item If $\cl(u)$ is closed under reversal, then
\be\label{eq:palineq}
\Pal_u(n) + \Pal_u(n+1) \leq p_u(n+1) - p_u(n) + 2 \text{ for all } n.
\en
\item Suppose that $u$ is the natural coding of a ``nondegenerate" exchange of $r$ intervals, so that $p_u(n)=n(r-1)+1$ for all $n$. If $\cl(u)$ is closed under reversal, then equality holds in the preceding estimate, and in fact
\be\label{eq:paleq}
\Pal_u(n)=\begin{cases}
1 \text{ if } n \text{ is even}\\
r \text{ if } n \text{ is odd}.
\end{cases}
\en
\end{enumerate}
\end{thm}

\begin{exer}
	Check that Equation \ref{eq:paleq} implies that equality holds in Equation \ref{eq:palineq}
\end{exer}

The paper \cite{GlenJustinWidmerZamboni2009} studies blocks and sequences that contain a lot of palindromes. The following observation was made in \cite{DroubayJustinPirillo2001}.

\begin{exer}
	Let $B$ be a block of length $|B|=n$. Show that the number of different palindromes that are subblocks of $B$ (including the empty block) is at most $n+1$. ({\em Hint}: Define a block $B$ to have {\em Property J} if there is a suffix of $B$ that is a palindrome and appears only once in $B$. For $x \in \ca$, what are the possibilities for $\Pal_{Bx}$, depending on whether $B$ has property $J$ or not? Then use induction on the lengths of prefixes of $B$ to show that $\Pal_B$ is the cardinality of the set of prefixes of $B$ that have property $J$.)
\end{exer}

In view of the preceding observation, the authors of \cite{GlenJustinWidmerZamboni2009} define a finite block $B$ to be {\em rich} if it contains the maximum possible number, $|B|+1$, of palindromes as subblocks. An infinite sequence is defined to be rich if every one of its subblocks is rich. (See also \cite{AmbrozMasakovePelangtovaFrougny2006} and \cite{BrlekHamelNivatReutenauer2004}.)

\begin{exer}
		 Prove that a block $B$ is rich if and only if each of its prefixes has a suffix that is a palindrome and appears exactly once in $B$.
		\end{exer}
		\begin{exer}
		 If $B$ is rich then each of its subblocks is also rich.
\end{exer}

\begin{thm}\textnormal{\cite{GlenJustinWidmerZamboni2009}} A finite block or infinite sequence $w$ is rich if and only if for every subblock $B$ of $w$, if $B$ contains exactly two occurrences of a palindrome as a prefix and a suffix, then $B$ is itself a palindrome.
\end{thm}

In \cite{DroubayJustinPirillo2001} it is proved that every episturmian sequence is rich.

\begin{exer}
	Let $B$ be a block of length $n=|B|$ on an alphabet $\ca$ of cardinality $|\ca|=r$. What is the maximum possible cardinality of the family $\cl_B$ of all (different) subblocks of $B$?	
	\end{exer}
	
	\begin{exer}
		What is the minimum length $L(k,r)$ of a block on an alphabet $\ca$ of $r$ symbols that contains all $k$-blocks on $\ca$? What about configurations on $r$ symbols in $\Bz^2$?
	\end{exer}
	
	\begin{exer}
	How many rich words of length $5$ are there on an alphabet of $3$ letters? How many are there of length $n$ on an alphabet of $r$ letters?
	\end{exer}
	
	\begin{exer}
		Formulate a definition of palindrome in $\Bz^2$. Does the support of a ``palindrome" in $\Bz^2$ have to be a rectangle? How about a definition in $\Bz^d$?
	\end{exer}

\section{Nonrepetitive complexity and Eulerian entropy}
T. K. S. Moothathu \cite{moothathu2012eulerian} defined the {\em nonrepetitive complexity function} $P_u^N$ of a sequence $u$ on a finite alphabet $\ca$ in terms of how long an initial block of $u$ could be before it contained a repeat of a subblock of length $n$:
\be
P_u^N(n) = \max \{m \in \Bn: u_i^{i+n-1} \neq u_j^{j+n-1} \text { for all } 0 \leq i < j \leq m-1\}.
\en
The exponential growth rate may be called the {\em Eulerian entropy} of $u$, because there is a connection with Eulerian circuits in directed graphs, and denoted by $h_E(u)$:
\be
h_E(u)=\limsup_{n \to \infty}\frac{\log P_u^N(n)}{n}.
\en

Referring to Bowen's definition of topological entropy in terms of separated sets (\ref{def:bowenent}), the concept of Eulerian entropy extends to points in arbitrary topological dynamical systems $(X,T)$ when $X$ is a metric space. For a point $x \in X$, denote by $\beta(x,n,\epsilon)$ the maximum $m \in \Bn$ for which the initial orbit segment of $x$ of length $m$ is an $(n,\epsilon)$-separated set, so that
\be
\beta(x,n,\epsilon) = \max \{m \in \Bn : \{x,Tx,\dots ,T^{m-1}x\} \text{ is } (n,\epsilon)\text{-separated}\},
\en
define the {\em Eulerian entropy at $x$} to be
\be
h_E(X,T,x) = \lim_{\epsilon \to 0} \limsup_{n \to \infty} \frac{\log \beta(x,n,\epsilon)}{n},
\en
and define the {\em Eulerian entropy of the system} to be
\be
h_E(X,T)= \sup_{x \in X} h_E(X,T,x).
\en
The exact relationships among $h_E(X,T,x), h_E(X,T)$, and $\htop(X,T)$ are not clear, but here are some results from \cite{moothathu2012eulerian}.
\begin{enumerate}
\item $h_E(X,T)$ is an invariant of topological conjugacy.
\item Let $(X,\sigma)$ be a one-step mixing shift of finite type on a finite alphabet $\ca$ which has a {\em safe symbol}, a symbol $s \in \ca$ such that for all $a \in \ca$ both blocks $as,sa \in \cl(X)$. Then there is a sequence $x \in X$ such that $h_E(X,\sigma,x)=\htop (X,\sigma)$.
\item If $(X,\sigma)$ is a mixing shift of finite type such that for all large enough $n$ the de Bruijn graph of $(X,\sigma)$  is irreducible, then $h_E(X,\sigma,x)=\htop(X,\sigma)$ for a residual set of $x \in X$. 
(Recall that the de Bruijn graph of $(X,\sigma)$ has for its vertices the set of $(n-1)$-blocks, and there is an edge from $B$ to $C$ if and only if there is an $n$-block that has prefix $B$ and suffix $C$.)
\item There is a (nontransitive) homeomorphism $T$ on the Cantor set $X$ such that $h_E(X,T,x) < \htop (X,T)$ for all $x \in X$. 
\end{enumerate}
\begin{exer}
	Construct some de Bruijn graphs for several mixing shifts of finite type. How can one tell whether they will be irreducible for all large enough $n$?
\end{exer}
\begin{exer}
	Determine $P_u^N(n)$ and $h_E(u)$ to the extent possible when $u$ is the fixed point of the Fibonacci substitution, the Prouhet-Thue-Morse sequence, the Champernowne sequence, etc.
\end{exer}

\section{Mean topological dimension}

The complexity function $p_X$ and most of its variants are especially useful in the case of systems that have topological entropy zero. In a study of holomorphic and harmonic maps, M. Gromov \cite{Gromov1999} proposed an invariant, which he called {\em mean topological dimension}, that is especially useful in the study of systems with infinite entropy, such as the shift on $[0,1]^\Bz$. 
This invariant was studied in \cite{LindenstraussWeiss2000} (based on Lindenstrauss' Ph.D. dissertation) and used there and in subsequent papers \cite{Lindenstrauss1999,LindenstraussTsukamoto2014,GutmanLT2015} to settle questions about the possibility of embedding a given topological dynamical system $(X,T)$ in a standard target such as $(([0,1])^d)^\Bz,\sigma)$.

We deal with a compact metric space $X$ and finite open covers $\cu,\cv$ of $X$. Recall that $\cv$ {\em refines} $\cu$, written $\cv \succ \cu$, if every member of $\cv$ is contained in some member of $\cu$, and the {\em join} of $\cu$ and $\cv$ is $\cu \vee \cv= \{U \cap V: U \in \cu, V \in \cv\}$. 

The {\em dimension} of a finite open cover $\cu$ is defined to be
\be
D(\cu) = \min_{\cv \succ \cu}\left( \max_{x \in X} \sum_{V \in \cv}1_V(x)\right)-1,
\en
the smallest $n$ such that for every refinement $\cv$ of $\cu$ no point of $X$ is in more than $n+1$ elements of $\cv$.
The {\em cover dimension} of the compact metric space $X$ is $D(X) = \sup_\cu D(\cu)$. 
One can show that $D(\cu)$ is subadditive, $D(\cu \vee \cv) \leq D(\cu) + D(\cv)$, and therefore the limit
\be
D(\cu,T)= \lim_{n \to \infty} \frac{1}{n} D\left(\cu_0^{n-1}\right)=\lim_{n \to \infty}\frac{1}{n}D(\cu \vee T^{-1}\cu \vee \dots \vee T^{-n+1}\cu)
\en
exists. The {\em mean topological dimension} of $(X,T)$ is defined to be
\be
\mdim(X,T) = \sup_\cu D(\cu,T)
\en
and is an invariant under topological conjugacy.

Here are a few results from \cite{LindenstraussWeiss2000,Lindenstrauss1999}.
\begin{enumerate}
	\item If $D(X)<\infty$ or $\htop(X,T)<\infty$, then $\mdim(X,T)=0$.
	\item $\mdim([0,1]^\Bz,\sigma) =D([0,1])=1$.
	\item If $Y \subset X$ is closed and $T$-invariant ($TY=Y$), then $\mdim(Y,T)\leq \mdim(X,T)$.
	\item It was an open question of J. Auslander whether every minimal topological dynamical system $(X,T)$ can be embedded in $([0,1]^\Bz,\sigma)$ (in analogy with Beboutov's theorem on the embedding of $\Br$ actions in $C(\Br)$ with the translation action---see \cite{Kakutani1968}). Because of the preceding result the answer is negative, since for any $r\in[0,\infty]$ there is a minimal system $(X,T)$ with $\mdim(X,T)=r$.
	\item If $d \in \Bn$ and $(X,T)$ is a topological dynamical system with $\mdim(X,T)<d/36$, then $(X,T)$ embeds in $(([0,1]^d)^\Bz,\sigma)$.
\end{enumerate}

\begin{exer}
	Prove that $D([0,1])=1$.
\end{exer}

\section{Amorphic complexity via asymptotic separation numbers}\label{sec:amorphic} 
To measure the complexity of zero-entropy systems, in \cite{FuhrmannGrogerJager2015} the authors propose the concepts of {\em separation numbers} and {\em amorphic complexity}. Let $(X,T)$ be a topological dynamical system, with $X$ a compact metric space.  For $\delta > 0$, $n \in \Bn$, and points $x,y  \in X$, 
let
\be
S_n(x,y,\delta)=\card \{k=0,1,\dots,n-1: d(T^kx,T^ky)>\delta\}
\en
be the number of times that the first $n$ iterates of $x$ and $y$ are separated by a distance greater than $\delta$; and, for $\nu \in (0,1]$, define $x$ and $y$ to be {\em $(\delta,\nu)$-separated} if 
\be
\limsup_{n \to \infty} \frac{S_n(x,y,\delta)}{n} \geq \nu.
\en
A subset $E \subset X$ is called {\em $(\delta,\nu)$-separated} if every pair $x,y \in E$ with $x \neq y$ is $(\delta,\nu)$-separated. The maximum cardinality of a $(\delta,\nu)$-separated subset of $X$ is called the {\em (asymptotic) separation number for distance $\delta$ and frequency $\nu$} and is denoted by $\Sep(\delta,\nu)$. (See also Section \ref{sec:slow}.)

\begin{exer}
Prove that if $(X,T)$ is equicontinuous, then $\Sep(\delta,\nu)$ is uniformly bounded in $\delta$ and $\nu$.
\end{exer}

For a fixed $\delta > 0$, the upper and lower scaling rates of the separation numbers as the frequency tends to $0$ are defined to be
\be
A^+(\delta)=\limsup_{\nu \to 0}\frac{\log\Sep(\delta,\nu)}{-\log\nu} \text{ and } A^-(\delta)=\liminf_{\nu \to 0}\frac{\log\Sep(\delta,\nu)}{-\log\nu}.
\en
The polynomial scaling rate has been chosen as possibly the most interesting one to apply to zero-entropy examples. And notice also that the rate is in terms of shrinking frequency instead of increasingly long time interval as usual, since the time asymptotics were already included before.

The {\em upper and lower amorhic complexities} are defined to be
\be
A^+(X,T) = \sup_{\delta >0} A^+(\delta) \text{ and } A^-(X,T)= \sup_{\delta > 0} A^-(\delta).
\en
If $A^+(X,T)=A^-(X,T)$, then their common value is called the {\em amorphic complexity} of $(X,T)$ and is denoted by $A(X,T)$.

Here are some of the results in \cite{FuhrmannGrogerJager2015}.
\begin{enumerate}
\item If $(X,T)$ has positive topological entropy or is weakly mixing for some invariant measure with nontrivial support, then $\Sep(\delta,\nu)$ is infinite for some values of $\delta$, $\nu$.
\item Recall that a minimal system $(X,T)$ is called {\em almost automorphic} if it is an almost one-to-one extension (meaning that there exists a singleton fiber) of a minimal equicontinuous system \cite{Veech1970}. If $(X,T)$ is almost automorphic but not equicontinuous, then  $\Sep(\delta,\nu)$ is not uniformly bounded in $\delta$ and $\nu$.
\item If $(X,T)$ is an extension of an equicontinuous system $(Y,S)$ with respect to a factor map $\pi: (X,T) \to (Y,S)$ such that $\mu \{y \in Y: \card\pi^{-1}\{y\}>1\} =0$ for every invariant measure $\mu$ on $Y$, then  $\Sep(\delta,\nu)<\infty$ for all $\delta,\nu$. (Regular Toeplitz systems provide such examples.)
\item If $(X,T)$ is a Sturmian subshift, then $A(X,T)=1$.
\item The set of values of $A(X,T)$ for regular Toeplitz systems $(X,T)$ is dense in $[1,\infty)$.
\end{enumerate}

\section{Inconstancy}
Inspired by nineteenth-century ideas of Crofton and Cauchy on the fluctuations of curves and of M. Mend\`{e}s France \cite{MendesFrance1981/1982} on the entropy or temperature of a curve, Allouche and Maillard-Teyssier \cite{AlloucheMaillard-Teyssier2011} define the {\em inconstancy} of a plane curve to be twice its length divided by the perimeter of its convex hull. Mend\`{e}s-France had suggested the logarithm of this quantity as the {\em entropy} of the curve. (According to Fechner's law in psychophysics, the magnitude of a sensation is proportional to the logarithm of the intensity of the stimulus.)
The inconstancy $\ci (u)$ of a sequence $u$ of real numbers is defined to be the inconstancy of the piecewise linear curve obtained by connecting the points on its graph. 
See \cite{AlloucheMaillard-Teyssier2011} for the references, interesting examples, and discussion of possible applications of inconstancy.  

For sequences $u_0u_1\dots$ on the alphabet of two symbols $0$ and $h>0$ for which the frequencies $\mu[ab]$ exist for all 2-blocks $ab$, the authors provide the following results.
\begin{enumerate}
	\item $\ci(u) = 1+ (\sqrt{h^2+1}-1)(\mu[0h]+\mu[h0])$.
	\item If $d \in \Bn$, the inconstancy of the periodic sequence $(0^d1)^\infty$ is
	\be
	\ci((0^d1)^\infty)= \frac{d-1+2\sqrt{2}}{d+1}.
	\en
	This is largest ($\sqrt{2}$) for $d=1$ and tends to $1$ as $d \to \infty$ and the curve becomes flat.
	\item\label{item:randomseq} For a random binary sequence, the inconstancy is $(1+\sqrt{2})/2=1.207\dots$.
	\item\label{item:PTM} For the Prouhet-Thue-Morse sequence, the inconstancy is $(1+2\sqrt{2})/3=1.276\dots$. This follows from the fact that $\mu[00]=\mu[11]=1/6, \mu[01]=\mu[10]=1/3$. The inconstancy of this sequence is relatively high because it does not contain long strings (length more than two) of $0$ or $1$.
  \item If $u$ is a Sturmian sequence for which the frequency of $1$ is $\alpha$ and which does not contain the block $11$, then
  \be
  \ci(u)=1 + 2(\sqrt{2}-1) \alpha .
  \en
\end{enumerate}

  \begin{exer}
  	Show that no Sturmian sequence contains both blocks $00$ and $11$.
  \end{exer}

\begin{exer}
	Prove statements (\ref{item:randomseq}) and (\ref{item:PTM}) above.
\end{exer}

\section{Measure-theoretic complexity}\label{sec:mtcomplexity}

Ferenczi \cite{Ferenczi1997} has proposed upper and lower measure-theoretic complexities, $P_T^+$ and $P_T^-$, as asymptotic growth rates of the number of $\overline{d}\text{-}n\text{-}\epsilon$ balls needed to cover $(1-\epsilon)$ of the space. A similar idea is used in \cite{KatokThouvenot1997}
 for the construction of some $\Bz^2$ actions and in \cite{Ratner1981}, with the $\overline{f}$ rather than $\overline{d}$ metric, to show that different Cartesian powers of the horocycle flow are not Kakutani equivalent.

Let $T:X \to X$ be an ergodic measure-preserving transformation on a probability space $(X,\cb,\mu)$. For a finite (or sometimes countable) measuable partition of $X$ and a point $x \in X$, denote by $\alpha(x)$ the cell of $\alpha$ to which $x$ belongs. For integers $i \leq j$ we denote by $\alpha_i^j$ the partition $T^{-i}\alpha \vee \dots T^{-j}\alpha$. The $\overline{d}$, or Hamming, distance between two blocks $B=b_1 \dots b_n$ and $C=c_1 \dots c_n$ is defined to be
\be
\overline{d}(B,C)= \frac{1}{n} \card\{i=1,\dots,n:b_i \neq c_i\} .
\en
For $x \in X, n \in \Bn$, and $\epsilon >0$, the {\em $\overline{d}\text{-}n\text{-}\epsilon$ ball centered at $x$ is defined to be
	\be
	B(\alpha,x,n,\epsilon)=\{y \in X: \overline{d}(\alpha_0^{n-1}(x),\alpha_0^{n-1}(y))<\epsilon\} .
	\en}
Define
\be
\begin{gathered}
K(\alpha,n,\epsilon,T)= \min\{K: \text{there are } x_1,\dots,x_K \in X
\text{ such that } \\ 
\mu\left(\cup_{i=1}^KB(\alpha,x_i,n,\epsilon)\right) \geq 1-\epsilon \}.
\end{gathered}
\en
Let $g: \Bn \to \Bn$ be an increasing function. Define
\be
\begin{gathered}
P_{\alpha,T}^+ \prec g \text{ to mean } \lim_{\epsilon \to 0} \limsup_{n \to \infty} \frac{K(\alpha,n,\epsilon,T)}{g(n)} \leq 1 \text{  and}\\
P_{\alpha,T}^+\succ g \text{ to mean } \lim_{\epsilon \to 0} \limsup_{n \to \infty} \frac{K(\alpha,n,\epsilon,T)}{g(n)} \geq 1 .
\end{gathered}\en
If both relations hold, we write $P_{\alpha,T}^+ \sim g$. Similar notation defines $P_{\alpha,T}^-$ when the $\limsup$s are replaced by $\liminf$s.

The suprema over all partitions $\alpha$ are defined as follows.
For an increasing function $g:\Bn \to \Bn$, $P_T^+ \prec g$ means that $P_{\alpha,T}^+ \prec g$ for every partition $\alpha$. And $P_T^+ \succ g$ means that for every increasing $h:\Bn \to \Bn$ such that $h(n) \leq g(n)$ for all large enough $n$ and with $\limsup h(n)/g(n)<1$ there is a partition $\alpha$ such that $P_{\alpha,T}^+ \succ h$. Similar careful defining produces $P_T^-$. These asymptotic growth rate equivalence classes, $P_T^+$ and $P_T^-$, are called the {\em upper and lower measure-theoretic complexities of $(X,\cb,\mu,T)$}, respectively.

Ferenczi establishes the following properties of these measure-theoretic complexity ``functions".
\begin{enumerate}
\item If $\alpha$ is a generating partition, then $P_T^+ \sim P_{\alpha,T}^+$ and $P_T^- \sim P_{\alpha,T}^-$.
\item $P_T^+$ and $P_T^-$ are invariants of measure-theoretic isomorphism.
\item $\lim \log P_T^+(n)/n=\lim \log P_T^-(n)/n = h_\mu (T)$, in the sense that $P_T^+$ and $P_T^-$ are dominated by $e^{n(h+\epsilon)}$ for every $\epsilon >0$ and dominate $e^{n(h-\epsilon)}$ for every $\epsilon >0$.
\item $(X,\cb,\mu,T)$ is measure-theoretically isomorphic to a translation on a compact group if and only if $P_T^+ \prec g$ for every unbounded increasing $g: \Bn \to \Bn$. (Same for $P_T^-$.) This is to be compared with the characterization of topological dynamical systems with bounded {\em topological complexity functions} as the translations on compact groups---see Sections \ref{sec:complexityfunction} and \ref{sec:topologicalcomplexity}.
\item For the Morse system, $P_T^+  \sim 10n/3$ and $P_T^- \sim 3n$.
\item For rank one systems, $P_T^- \prec an^2$ for every $a>0$.
\item For the Chacon system (the orbit closure of the fixed point of the substitution $0 \to 0010, 1 \to 1$), $P_T^+ \sim P_T^-\sim 2n$.
\end{enumerate}

\section{Pattern complexity}
Kamae and Zamboni \cite{KamaeZamboni2002a,KamaeZamboni2002b}
have defined and studied {\em maximal pattern complexity} for infinite sequences, which shares some features with sequence entropy (Sections \ref{sec:mtseqent} and \ref{sec:topseqent}) and average sample complexity (Section \ref{sec:averagesamplecomplexity}).
A {\em pattern} is defined to be a finite increasing sequence $\tau=\tau(0)\tau(1)\tau(2)\dots\tau(k-1)$ of integers with $\tau(0)=0$.
If $u=u_0u_1\dots$ is an infinite one-sided sequence on a finite alphabet, onee considers the set of words seen at places in $u$ along the pattern $\tau$, namely
\be
F_\tau(u)=\{u_iu_{i+\tau(1)}\dots u_{i+\tau(k-1)}: i=0,1,2,\dots\}.
\en
The {\em maximal pattern complexity function } of $u$ is defined to be
\be
P_u^*(k)=\sup\{|F_\tau(u)|: |\tau|=k\}, k=1,2,\dots .
\en
 The authors establish the following results.
 \begin{enumerate}
 	\item Let $u$ be an infinite sequence on a finite alphabet, and suppose that the weak* limit of $\sum_{i=0}^{n-1}\delta_{\sigma^iu}/n$ exists and so defines a shift-invariant measure $\mu$ on the orbit closure $X$ of $u$. Then for any sequence $S=(t_i)$, the sequence entropy (see Section \ref{sec:mtseqent}) is bounded in terms of the maximal pattern complexity: with $\alpha$ the usual time-0 partition of $X$,
 	\be
 	h_\mu^S(\alpha,\sigma) \leq \limsup_{k \to \infty} \frac{\log P_u^*(k)}{k}.
 	\en
 	Hence, if $(X,\sigma,\mu)$ does not have purely discrete spectrum, the maximal pattern complexity finction $P_u^*(k)$ has exponential growth rate.
 	\item If there exists $k$ such that $P_u^*(k) < 2k$, then $u$ is eventually periodic (cf. Exercise \ref{exer:evper}).
 	\item For Sturmian (and generalized Sturmian) sequences, $P_u^*(k)=2k$ for all $k$.
 	\item For the Prouhet-Thue-Morse sequence, $P_u^*(k)=2^k$ for all $k$.
 \end{enumerate}
 
 Further developments concerning pattern complexity, including possible applications to image recognition, can be found in \cite{Kamae2012, FerencziHubert2012,XueKamae2012} and their references.
 
 \begin{exer}
 	The {\em window complexity function} $P_u^w(n)$ of an infinite sequence $u$ on a finite alphabet $\ca$, studied in \cite{CassaigneKaboreTapsoba2010}, counts the number of different blocks of length $n$ seen in the sequence $u$ beginning at coordinates that are multiples of $n$:
 	\be
 	P_u^w(n) = \card \{u_{kn}u_{kn+1}\dots u_{(k+1)n-1}: k \geq 0\}.
 	\en
 	 In the spirit of pattern complexity, one could also consider counting blocks of length $n$ that appear in a sequence $u$ along an arithmetic progression of coordinates:
 \be
 P_u^a(n) = \card \{u_iu_{i+k}u_{i+2k}\dots u_{i+(n-1)k}: i \geq 0, k \geq 1\}.
 \en
 Discuss the relationships among the ordinary complexity function $p_u$ (Section \ref{sec:complexityfunction}), $P_u^w$, $P_u^a$, and $P_u^*$. Compute them for Sturmian sequences and the Prouhet-Thue-Morse sequence. 
\end{exer}

\begin{exer}
	Consider the {\em Oldenburger-Kolakoski sequence} 
	\be
	\omega = 122112122122112 \dots ,
	\en
	which coincides with its sequence of run lengths. Can you say anything about any measure of complexity for this sequence? Anything about its orbit closure as a subshift? (It seems to be an open problem whether the frequency of $1$ exists, although numerical studies show that it seems to be $1/2$. So finding any invariant measure that it defines could be difficult... )
	
	What about a sequence on the alphabet $\{2,3\}$ that has the same self-similarity property? 
	
	Or consider all such sequences on $\{1,2,3\}$? Since now some choices seem possible, the orbit closure should have positive topological entropy?
\end{exer}

\chapter{Balancing freedom and interdependence}

\section{Neurological intricacy}

In any system that consists of individual elements there is a tension between freedom of action of the individuals and coherent action of the entire system. There is maximum complexity, information, disorder, or randomness when all elements act independently of one another. At the other extreme, when the elements of the system are strongly linked and the system acts essentially as a single unit, there is maximum order. In the first situation, it seems that there is little advantage to the individual elements in being part of a larger system, and the system does not benefit from possible concerted action by its constituents. In the second situation, most individual elements could be superfluous, and the system does not benefit from any diversity possibly available from its parts. This tension between the one and the many, the individual and the state, the soloist and the orchestra, the part and the whole, is ancient and well known.

It is natural to think that evolving organisms or societies may seek a balance in which the benefits of diversity and coherence are balanced against their disadvantages. 
Abrams and Panaggio \cite{AbramsPanaggio2011} constructed a differential equations model to describe the balance between competitive and cooperative pressures to attempt to explain the prevalence of right-handedness in human populations. (Left-handers may have an advantage in physical competitions, where opponents are accustomed to face mostly right-handers, and the population as a whole may benefit from diversity. But left-handers will be at a disadvantage when faced with objects and situations designed for the comfort of the prevalent right-handers.) 
Blei \cite{Blei2011} defined measures of {\em interdependence} among families of functions in terms of functional dependence among subfamilies using combinatorics, functional analysis, and probability. 

Neuroscientists G. Edelman, O. Sporns, G. Tononi \cite{tononi1994measure} proposed a measure, which they called ``neural complexity", of the balance between specific and mass action in the brains of higher vertebrates.
		High values of this quantity are associated with non-trivial organization of the network; when this is the case, segregation coexists with integration.
			 Low values are associated with systems that are either completely independent (segregated, disordered) or completely dependent (integrated, ordered).
			 We will see that beneath this concept of intricacy there is another (new) basic notion of complexity, that we call {\em average sample complexity}.
			 The definitions and study of intricacy and average sample complexity in dynamics were initiated in \cite{Wilson2015,PetersenWilson2015}.
			
			 One considers a model that consists of a family
	 	 $X=\{X_i:i=0,1,\dots,n-1\}$ of random variables representing an isolated neural system with $n$ elementary components (maybe groups of neurons), each $X_i$ taking values in a discrete (finite or countable) set $E$.
	 	 For each $n \in \Bn$ we define $n^*=\{0,1,\dots,n-1\}$.
	 	 The set $n^*$ represents the set of {\em sites}, and $E$ represents the set of {\em states}. It may seem that we are assuming that the brain is one-dimensional, but not so: the sites may be arranged in some geometrically important way, but at this stage we only number them and will take the geometry, distances, connections, etc. into account maybe at some later stage.
	 	 The elements of the set $E$ (often $E=\{0,1\}$) may encode (probably quantized) levels of excitation or something analogous.
		 For $S\subset n^*$, $X_S=\{X_i:i\in S\}$,
	 ${S^c}=n^*\setminus S$. Neural complexity measures the level of interdependence between action at the sites in $S$ and those in $S^c$, averaged over all subset $S$ of the set of sites, with some choice of weights.

	The	\emph{entropy} of a random variable $X$ taking values in a discrete set $E$ is
		\[
		H(X)=-\sum_{x\in E}Pr\{X=x\}\log Pr\{X=x\}.
		\]
		The {\em mutual information} between two random variables $X$ and $Y$ over the same
		probability space $(\Omega, \cf, P)$ is
		\[
		\begin{aligned}
		MI(X,Y)&=H(X)+H(Y)-H(X,Y).\\
		&=H(X)-H(X|Y)=H(Y)-H(Y|X)
		\end{aligned}
		\]
 Here $(X,Y)$ is the random variable on $\Omega$ taking values in $E \times E$ defined by $(X,Y)(\omega)=(X(\omega), Y(\omega))$.
	 $MI(X,Y)$ is a measure of how much $Y$ tells about $X$ (equivalently, how much $X$ tells about $Y$).
	$MI(X,Y)=0$ if and only if $X$ and $Y$ are independent.

The
{\em neural complexity, $C_N$}, of the family  $X=\{X_i:i=0,1,\dots,n-1\}$ is defined to be the following
		average, over all subfamilies  $X_S=\{X_i:i \in S\}$, of the mutual information between $X_S$ and $X_{S^c}$:
		\[
		C_N(X)=\frac{1}{n+1}\sum_{S\subset n^*}\frac{1}{\binom{n}{|S|}}MI(X_S,X_{{S^c}}).
		\]
		The weights are chosen to be uniform over all subsets of the same size, and then uniform over all sizes.

J. Buzzi and L. Zambotti \cite{BZ12} studied neural complexity in a general probabilistic setting, considering it as one of a family of functionals on processes that they called {\em intricacies}, allowing more general systems of weights for the averaging of mutual informations.	
They define a
		 \emph{system of coefficients}, $c_S^n$, to be a family of numbers satisfying for all $n\in\mathbb{N}$ and $S\subset n^*$
		\begin{enumerate}[1.]
\item			 $c_S^n\ge 0$;
	\item	 	 $\sum_{S\subset n^*}c_S^n=1$;
		 \item	 $c_{{S^c}}^n=c_S^n$.
		\end{enumerate}

			 For a fixed $n\in\mathbb{N}$ let $X=\{X_i:i\in n^*\}$ be a collection of random variables all taking values in the same finite set.
			 Given a system of coefficients, $c_S^n$, the corresponding \emph{mutual information functional}, $\mathcal{I}^c(X)$, is defined by
			\[
			\mathcal{I}^c(X)=\sum_{S\subset n^*}c_S^n MI(X_S,X_{{S^c}}).
			\]

\begin{defn}
			An \emph{intricacy} is a mutual information functional satisfying:
		\begin{enumerate}[1.]
	\item		  Exchangeability: invariance under permutations of $n^*$;
		\item	 Weak additivity: $\mathcal{I}^c(X,Y)=\mathcal{I}^c(X)+\mathcal{I}^c(Y)$ for any two  independent systems $X=\{X_i:i\in n^*\}$ and $Y=\{Y_j:j\in m^*\}$.
		\end{enumerate}
	\end{defn}

	\begin{thm}[Buzzi-Zambotti]\label{thm:BZ}
		Let $c_S^n$ be a system of coefficients and $\mathcal{I}^c$ the associated mutual information functional. $\mathcal{I}^c$ is an intricacy if and only if there exists a symmetric probability measure $\lambda_c$ on $[0,1]$ such that
		\[
		c_S^n=\int_{[0,1]}x^{|S|}(1-x)^{n-|S|}\lambda_c(dx)
		\]
	\end{thm}
	
	\begin{example}
		\begin{enumerate}[1.]
\item			 $\displaystyle c_S^n=\frac{1}{(n+1)}\frac{1}{\binom{n}{|S|}}$ (Edelman-Sporns-Tononi);
\item			 For $0<p<1$,
			\[
			c_S^n=\frac{1}{2}(p^{|S|}(1-p)^{n-|S|}+(1-p)^{|S|}p^{n-|S|}) \text{ (}p\text{-symmetric)};
			\]
\item			 For $p=1/2$, $c_S^n=2^{-n}$ (uniform).
		\end{enumerate}
	\end{example}
	
\begin{exer}
	Prove that the neural (Edelman-Sporns-Tononi) weights correspond to $\lambda$ being Lebesgue measure on $[0,1]$.
\end{exer}

\section{Topological intricacy and average sample complexity}

%
%

	Let $(X,T)$ be a topological dynamical system and $\mathscr{U}$ an open cover of $X$. Given $n\in\mathbb{N}$ and a subset $S\subset n^*$ define
	\[
	\mathscr{U}_S=\bigvee_{i\in S}T^{-i}\mathscr{U}.
	\]
	\begin{defn}\cite{Wilson2015,PetersenWilson2015}
		Let $c_S^n$ be a system of coefficients. Define the \emph{topological intricacy of $(X,T)$ with respect to the open cover} $\mathscr{U}$ to be
		\[
		\Int(X,\mathscr{U},T):=\lim_{n\rightarrow\infty}\frac{1}{n}\sum_{S\subset n^* }c_S^n\log\left( \frac{N(\mathscr{U}_S)N(\mathscr{U}_{{S^c}})}{N(\mathscr{U}_{n^*})}\right).
		\]
	\end{defn}

Applying the laws of logarithms and noting the symmetry of the sum with respect to sets $S$ and their complements leads one to define the following quantity.

	\begin{defn}\cite{Wilson2015,PetersenWilson2015}
		The \emph{topological average sample complexity of $T$ with respect to the open cover $\mathscr{U}$} is defined to be
		\[
		\Asc(X,\mathscr{U},T):=\lim_{n\rightarrow\infty}\frac{1}{n}\sum_{S\subset n^* } c_S^n\log N(\mathscr{U}_S).
		\]
	\end{defn}
	
	\begin{prop}
	$\displaystyle{
	\Int(X,\mathscr{U},T)=2\Asc(X,\mathscr{U},T)-\htop(X,\mathscr{U},T)}$.
		\end{prop}
	
	
	%

	\begin{thm}
		The limits in the definitions of $\Int(X,\mathscr{U},T)$ and $\Asc(X,\mathscr{U},T)$ exist .
	\end{thm}
	As usual this follows from subadditivity of the sequence
	\be
	b_n:=\sum_{S\subset n^*}c_S^n\log N(\mathscr{U}_S)
	\en
	and Fekete's Subadditive Lemma: For every subadditive sequence $a_n$, $\lim_{n\rightarrow\infty} a_n/n$ exists and is equal to $\inf_n a_n/n$.
	
	\begin{exer}
		Prove Fekete's Lemma.
	\end{exer}
	
	\begin{prop}
		For each open cover $\mathscr{U}$, $\Asc(X,\mathscr{U},T)\le \htop(X,\mathscr{U},T)\le \htop(X,T)$, and hence $\displaystyle \Int(X,\mathscr{U},T)\le \htop(X,\mathscr{U},T)\le \htop(X,T).$
	\end{prop}
	
	In particular, a dynamical system with zero (or relatively low) topological entropy (one that is coherent or ordered) has zero (or relatively low) topological intricacy.

  The intricacy of a subshift $(X,\sigma)$ with respect to the ``time zero open cover" $\mathscr U_0$ by cylinder sets defined by the first (or central) coordinate is determined by counting the numbers of different blocks that can be seen along specified sets of coordinates:

	\[
		\Int(X,\mathscr{U}_0,\sigma)=\lim_{n\rightarrow\infty}\frac{1}{n}\sum_{S\subset n^*}c_S^n\log\left(\frac{|\mathscr{L}_S(X)||\mathscr{L}_{{S^c}}(X)|}{|\mathscr{L}_{n^*}(X)|}\right)
		\]

	\begin{example}[Computing $|\mathscr{L}_S(X)|$ for the golden mean SFT]
		Let $n=3$, so that $n^*=\{0,1,2\}$. The following figure shows how different numbers of blocks can appear along different sets of coordinates of the same cardinality:  if $S=\{0,1\}$ then $N(S)=3$, whereas if $S=\{0,2\}$, $N(S)=4$.
		
		\begin{center}
			\begin{tikzpicture}
			\node at (.8,2.5) {$S=\{0,1\}$};
			\draw (0,1.8)--(.5,1.8) [line width=.6mm,red];
			\draw (.7,1.8)--(1.2,1.8) [line width=.6mm,red];
			\draw (1.4,1.8)--(1.9,1.8) [line width=.6mm];
			\node at (.25,1.4) [red] {$0$};\node at (.95,1.4)[red] {$0$};
			\node at (.25,1) [red] {$0$};\node at (.95,1)[red] {$1$};
			\node at (.25,.6) [red] {$1$};\node at (.95,.6)[red] {$0$};
			\node at (5.8,2.5) {$S=\{0,2\}$};
			\draw (5,1.8)--(5.5,1.8) [line width=.6mm,red];
			\draw (5.7,1.8)--(6.2,1.8) [line width=.6mm];
			\draw (6.4,1.8)--(6.9,1.8) [line width=.6mm, red];
			\node at (5.25,1.4) [red] {$0$};\node at (6.65,1.4) [red]{$0$};
			\node at (5.25,1) [red] {$0$};\node at (6.65,1) [red]{$1$};
			\node at (5.25,.6) [red] {$1$};\node at (6.65,.6) [red]{$0$};
			\node at (5.25,.2) [red] {$1$};\node at (6.65,.2) [red]{$1$};
			\node at (1,-.5) {$|\mathscr{L}_S(X)|=3$};
			\node at (6,-.5) {$|\mathscr{L}_S(X)|=4$};
			\end{tikzpicture}
		\end{center}
	\end{example}

When we average over all subsets $S \subset n^*$, we get an approximation (from above) to $\Int(X,\mathscr{U}_0,\sigma)$:

	\begin{example}[Computing $|\mathscr{L}_S(X)|$ for the golden mean SFT]
		\begin{table}[h]
			\[
			\begin{array}{cccc}
			\toprule
			S&{S}&|\mathscr{L}_S(X)|&|\mathscr{L}_{S^c}(X)|\\
			\midrule
			\emptyset&\{0,1,2\}&1&5\\
			\{0\}&\{1,2\}&2&3\\
			\{1\}&\{0,2\}&2&4\\
			\{2\}&\{0,1\}&2&3\\
			\{0,1\}&\{2\}&3&2\\
			\{0,2\}&\{1\}&4&2\\
			\{1,2\}&\{0\}&3&2\\
			\{0,1,2\}&\emptyset&5&1\\
			\bottomrule\\
			\end{array}
			\]
		\end{table}

		\begin{center}
			$\displaystyle \frac{1}{3\cdot2^3}\sum_{S\subset 3^*}\log\left(\frac{|\mathscr{L}_S(X)||\mathscr{L}_{S^c}(X)|}{|\mathscr{L}_{n^*}(X)|}\right)=\frac{1}{24}\log\left(\frac{6^4\cdot 8^2}{5^6}\right)\approx0.070$.
		\end{center}
	\end{example}

Apparently one needs better formulas for $\Int$ and $\Asc$ than the definitions, which involve exponentially many calculations as $n$ grows. Here is a formula  that applies to many SFT's and permits accurate numerical estimates.

	\begin{thm}\label{AscForSFT}
		Let $X$ be a shift of finite type with adjacency matrix $M$ such that $M^2>0$. Let $c_S^n=2^{-n}$ for all $S$. Then
		\[
		\Asc(X,\mathscr{U}_0,\sigma)=\frac{1}{4}\sum_{k=1}^\infty\frac{\log |\mathscr{L}_{k^*}(X)|}{2^k}.
		\]
	\end{thm}
	
	This formula shows that, as expected,
	$\Asc$ is sensitive to word counts of all lengths and thus is a finer measurement than $\htop$, which just gives the asymptotic exponential growth rate. Below we will see examples of systems that can be distinguished by $\Asc$ and $\Int$ but not by their entropies, or even by their symbolic complexity functions.
	
 The main ideas of the proof are:
 \begin{enumerate}
 	\item Each $S \subset n^*$ decomposes into a union of disjoint intervals $I_j$ separated by gaps of length at least $1$.
 	\item Because $M^2>0$, $\log N(S) = \sum_j \log N(I_j)$.
 	\item We may consider the $S \subset n^*$ that do not contain $n-1$, then those that do, and use induction.
 \end{enumerate} 	
	
	\begin{cor}
		For the full $r$-shift with $c_S^n=2^{-n}$ for all $S$,
		\[
		\Asc(\Sigma_r,\mathscr{U}_0,\sigma)=\frac{\log r}{2}\quad\text{and}\quad \Int(\Sigma_r,\mathscr{U}_0,\sigma)=0.
		\]
	\end{cor}

In the following table we compare $\htop, \Int$, and $\Asc$ for the full 2-shift, the golden mean SFT, and the subshift consisting of a single periodic orbit of period two. The first is totally disordered, while the third is completely deterministic, so each of these has intricacy zero, while the golden mean SFT has some balance between freedom and discipline.

	\[
	\begin{tabular}{@{}m{.14\textwidth}m{.38\textwidth} m{.11\textwidth}m{.10\textwidth}m{.08\textwidth}@{}}
	\toprule
	&\hspace{.3in} Adjacency Graph&Entropy&$\Asc$&$\Int$\\
	\midrule
	Disordered&\includegraphics[width=1.75in]{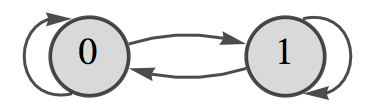}&$0.693$&$0.347$&$0$\\
	&\includegraphics[width=1.575in]{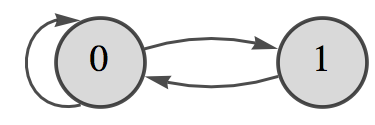}&$0.481$&$0.286$&$0.090$\\
	Ordered&\hspace{.175in}\includegraphics[width=1.4in]{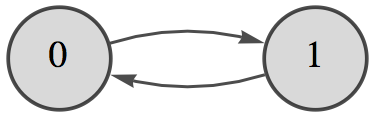}&$0$&$0$&$0$\\
	\bottomrule
	\end{tabular}
	\]

As with the definitions of topological and measure-theoretic entropies, one may seek to define an isomorphism invariant by taking the supremum over all open covers (or partitions). But this will lead to
 nothing new.

	\begin{thm}\label{thm:PWTop}
		Let $(X,T)$ be a topological dynamical system. Then
		\[
		\sup_{\mathscr{U}}\Asc(X,\mathscr{U},T)=\htop(X,T).
		\]
	\end{thm}
	
	\begin{proof}
			The proof depends on the structure of average subsets of $n^*=\{ 0, 1, \dots ,n-1\}$:
			most $S \subset n^*$ have size about $n/2$, so are not too sparse. 
			
			When computing the ordinary topological entropy of a subshift, to get at the supremum over open covers it is enough to start with the (generating) time-0 partition (or open cover) $\alpha$, then iterate and refine, replacing $\alpha$ by $\alpha_{k^*}=\alpha_0^{k-1}$. Then for fixed $k$, when we count numbers of blocks (configurations), we are looking at $\alpha_{(n+k )^*}$ instead of  $\alpha_{n^*}$; and when $k$ is fixed, as $n$ grows the result is the same.
		
		When computing $\Asc$ and $\Int$, start with the time-0 partition, and code by $k$-blocks. Then $S \subset n^*$ is replaced by $S+k^*$, and the effect on $\alpha_{S+k^*}$ as compared to $\alpha_S$ is similar, since it acts similarly on each of the long subintervals comprising $S$.

	\begin{tikzpicture}[xscale=8]
	\draw[-][draw, very thick] (.2,0) -- (1.2,0);
	
	\draw [thick] (.2,-.1)  -- (.2,0.1);
	\draw [thick] (.4,-.1)  -- (.4,0.1);
	\draw [thick] (.6,-.1)  -- (.6,0.1);
	\draw [thick] (.8,-.1)  -- (.8,0.1);
	\draw [thick] (1,-.1)  -- (1,0.1);
	\draw [thick] (1.2,-.1)  -- (1.2,0.1);
	\node at (.3,-.3) {$k/2$};
	\node at (.38,.6) {$s_1$};
	<7->{
		\draw [very thick] (.4,.6) -- (.8,.6);}
	<8->{
		\node at (.62,.3) {$s_2$};}
	<9->{
		\draw [very thick] (.64, .3) -- (1.04,.3);}
	\end{tikzpicture}

	Here is a still sketchy but slightly more detailed indication of the idea.		
		Fix a $k$ for coding by $k$-blocks (or looking at $N((\mathscr{U}_k)_S)$ or $H((\alpha_k)_S)$).
				Cut $n^*$ into consecutive blocks of length $k/2$.
				When $s \in S$ is in one of these intervals of length $k/2$, then $s+k^*$ {\em covers} the next interval of length $k/2$.
				So if $S$ hits many of the intervals of length $k/2$, then $S+k^*$ starts to look like a union of long intervals, say each with $|E_j|>k$.
				By shaving a little off each of these relatively long intervals, we can assume that also the gaps have length at least $k$.

		Given $\epsilon >0$, we may assume that $k$ is large enough that for every interval $I \subset \mathbb N$ with $|I| \geq k/2$,
		\be\label{happrox}
		0 \leq \frac{\log N(I)}{\card (I)} - h_{\topo} (X,\sigma) < \epsilon.
		\en
		
		We let $\mathfrak{B}$ denote the set of $S \subset n^*$ which miss at least $2n\epsilon/k$ of the intervals of length $k/2$
				and show that
				\[\displaystyle{\lim_{n \to \infty} \frac{\card (\mathfrak{B})}{2^n} = 0}.\]
				If $S \notin \mathfrak{B}$, then $S$ hits many of the intervals of length $k/2$,
				and hence $S+k^*$ is the union of intervals of length at least $k$, and we can arrange that the gaps are also long enough to satisfy the estimate in {\ref{happrox}} comparing (average of log of) number of blocks to $h_{\topo} (X,\sigma)$.
	\end{proof}

\begin{exer}
	Prove that 	$\sup_{\mathscr{U}}\Int (X,\mathscr{U},T)=\htop(X,T)$.
\end{exer}

\section{Ergodic-theoretic intricacy and average sample complexity}

%
%

We turn now to the formulation and study of the measure-theoretic versions of intricacy and average sample complexity.
	For a partition $\alpha$ of $X$ and a subset $S\subset n^*$ define
	\[ 
	\alpha_S=\bigvee_{i\in S}T^{-i}\alpha.
	\]
	
	\begin{defn}\cite{Wilson2015,PetersenWilson2015}
		Let $(X,\mathscr{B},\mu,T)$ be a measure-preserving system, $\alpha=\{A_1,\dots,A_n\}$ a finite measurable partition of $X$, and $c_S^n$ a system of coefficients.
		The \emph{measure-theoretic intricacy of $T$ with respect to the partition $\alpha$} is
		\[
		\Int_\mu(X,\alpha,T)=\lim_{n\rightarrow \infty}\frac{1}{n}\sum_{S\subset n^*}c_S^n\left[H_\mu(\alpha_S)+H_\mu(\alpha_{{S^c}})-H_\mu(\alpha_{n^*})\right].
		\]
		The \emph{measure-theoretic average sample complexity of $T$ with respect to the partition $\alpha$} is
		\[
		\Asc_\mu(X,\alpha,T)=\lim_{n\rightarrow\infty}\frac{1}{n}\sum_{S\subset n^*}c_S^nH_\mu(\alpha_S).
		\]
	\end{defn}

	\begin{thm}
		The limits in the definitions of measure-theoretic intricacy and measure-theoretic average sample complexity exist.
	\end{thm}
	
	\begin{thm}\label{thm:PWErg}
		Let $(X,\mathscr{B},\mu,T)$ be a measure-preserving system. Then
		\[
		\sup_{\alpha}\Asc_{\mu}(X,\alpha,T)=\sup_{\alpha}\Int_{\mu}(X,\alpha,T)=h_\mu(X,T).
		\]
	\end{thm}
	
	The proofs are similar to those for the corresponding theorems in topological setting.
	These observations indicate that there may be a topological analogue of the following result.

	\begin{thm}[Ornstein-Weiss, 2007]\label{thm:OW}
		If $J$ is a finitely observable functional defined for ergodic finite-valued processes that is an isomorphism invariant, then $J$ is a continuous function of the measure-theoretic entropy.
	\end{thm}
	
	Here $X=(x_1,x_2,\dots)$ is a finite-state ergodic stochastic process, and a ``finitely observable functional" is the a.s. limit $F(X)$ of a sequence of functions $f_n(x_1,x_2,\dots,x_n)$ taking values in some metric space which for every such process converges almost surely. The integral of $x_1$ and the entropy of the process are examples of finitely observable functionals.

			\section{The average sample complexity function}\label{sec:averagesamplecomplexity}

		The observations in the preceding situation suggest that one should examine these $\Asc$ and $\Int$ {\em locally}. For example, for a fixed open cover $\mathscr U$, fix a $k$ and find the topological average sample complexity $Asc(X, \mathscr{U}_k,\sigma)=\lim_{n \to \infty} \frac{1}{n}\sum_{S\subset n^*}c_S^n\log N((\mathscr{U}_k)_S)$.
				Or, do not take the limit on $n$, and study the quantity as a function of $n$,
				analogously to the symbolic or topological complexity functions.
				Similarly for the measure-theoretic version: fix a partition $\alpha$ and study the limit, or the function of $n$.
		\[\Asc_\mu(X,T,\alpha)=\lim_{n\rightarrow\infty}\frac{1}{n}\sum_{S\subset n^*}c_S^n H_\mu(\alpha_S).
		\]

		So we begin study of the Asc of a fixed open cover as a function of $n$,
	\[
	\Asc(X,\sigma,\mathscr{U}_k,n)=\frac{1}{n}\sum_{S\subset n^*}c_S^n\log N(S),
	\]
	especially for SFT's and $\mathscr U = \mathscr U_0$, the natural time-0 cover (and partition).  
	
	Figure \ref{DifferentAscs} shows two SFT's that have the same number of $n$-blocks for every $n$ but different $\Asc$ functions.
	
	\begin{example}
		\begin{figure}\label{DifferentAscs}
			\[
			\begin{array}{cc}
			\text{I}&\text{II}\\
			&\\
			\xymatrix{
				&*+[o][F-]{0}\ar@(ur,ul)[]\ar@/_/[dl]&\\
				*+[o][F-]{1}\ar@/^/[rr]&&*+[o][F-]{2}\ar@/^/[ll]\ar@/_/[ul]
			}&
			\xymatrix{
				&*+[o][F-]{0}\ar@/^/[dr]\ar@/^/[dl]&\\
				*+[o][F-]{1}\ar@/^/[ur]\ar@/_/[rr]&&*+[o][F-]{2}\ar@/^/[ul]
			}\\
			&\\
			(\text{M}_I)^3>0&(\text{M}_{II})^4>0
			\end{array}
			\]
			\caption{Graphs of two subshifts with the same complexity function but different average sample complexity functions.}
		\end{figure}
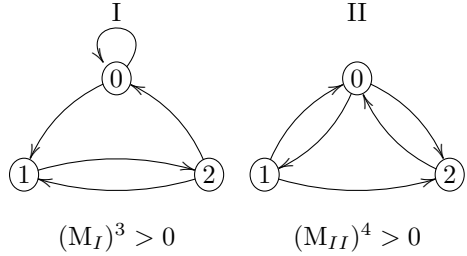
	\end{example}

	\begin{equation*}
	\Asc(n)=\frac{1}{n}\frac{1}{2^n}\sum_{S\subset n^*}\log N(S)
	\end{equation*}
			\begin{figure}\label{AscGraphs}
			\begin{center}
		\includegraphics[width=4.5in]{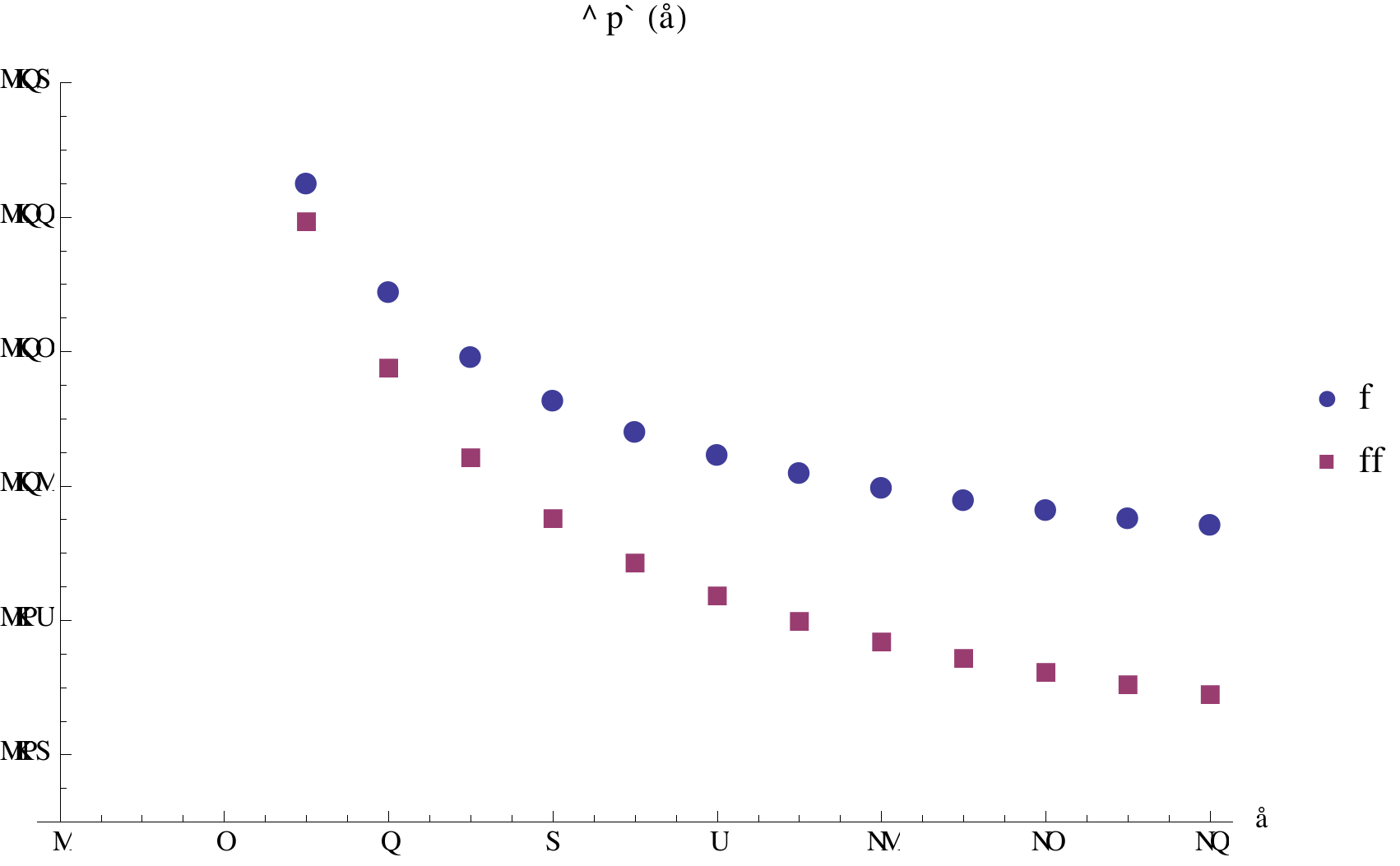}
			\end{center}
\end{figure}

Numerical evidence (up to $n=10$) indicates that these two SFT's have different values of $\Asc$ and $\Int$, although they have identical complexity functions and hence the same topological entropy..

		\[
	\begin{tabular}{@{}m{.25\textwidth} m{.12\textwidth}m{.12\textwidth}m{.12\textwidth}@{}}
	\toprule
	Adjacency Graph&$h_{\text{top}}$&$\Asc(10)$&$\Int(10)$\\
	\midrule
	\includegraphics[width=.9in,keepaspectratio]{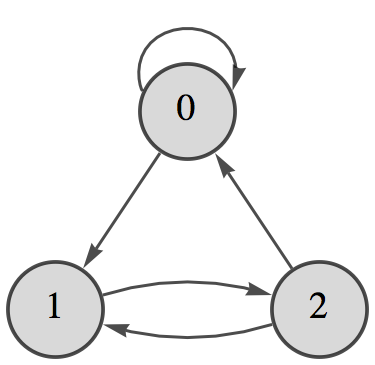}&$ 0.481$&$0.399$&$0.254$\\
	\includegraphics[width=.9in,keepaspectratio]{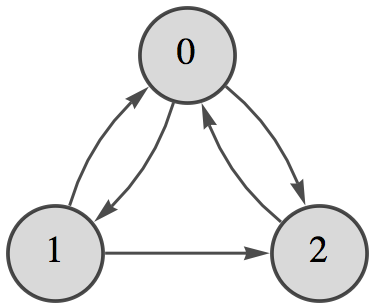}&$0.481$&$0.377$&$0.208$\\
	\bottomrule
	\end{tabular}
	\]

\section{Computing measure-theoretic average sample complexity}

%
%

 We identify $\Asc_\mu(X,\alpha,T)$ as ($1/2$ of) a conditional entropy, or fiber entropy, of a skew product of $X$ with the $2$-shift. 
 For a 1-step Markov process with respect to a fixed generating partition $\alpha$, 
 then $\Asc_\mu(X,\alpha,T)$ is given by 
  a series which  serves as a computational tool analogous to the series in Theorem \ref{AscForSFT}. 
  The next two results extend to general weights by applying Theorem \ref{thm:BZ}.

The idea is to view a subset $S\subset n^*$ as corresponding to a random binary string of length $n$ generated by the Bernoulli $(1/2,1/2)$ measure $P$ on the full $2$-shift $(\Sigma_2^+,\sigma)$.
For example $\{0,2,3\}\subset 5^*\leftrightarrow 10110 $.
The average of the entropy $H_\mu(\alpha_S)$, over all $S\subset n^*$, is then an integral and can be interpreted in terms of the entropy of a first-return map to the cylinder $A=[1]$ in a cross product of our system $X$ and the full $2$-shift.

Denote by $T_{X \times A}$ the first-return map to $X \times [1] \subset X \times A$ and by $\sigma_A$ the first-return map to $A$ under the shift $\sigma$. 
Let $P_A=P/P(A)$ denote Bernoulli $(1/2,1/2)$ measure normalized on $A$.  
For $\xi \in A$ define $n_A(\xi)=\inf\{n \geq 1: \sigma^n \xi \in A\}$.
The system $(X \times A, T_{X \times A}, \mu \times P_A)$ is a skew product over $(A, \sigma_A,P_A)$ with $T_{X \times A}(x,\xi)=(T^{n_A(\xi)}x, \sigma_A\xi)$. (Here the base system is written in the second coordinate.)
For each $n=1,2,\dots$ let $A_n=\{\xi \in A: n_A(\xi)=n\}$, and let $\mathcal A$ denote the $\sigma$-algebra generated by the sets $X \times A_n, n\geq 1$. 
Denote by $h_{\sigma_A}(X,T,\mu)$ the fiber entropy of the skew product $(X \times A,T_{X\times A},\mu \times P_A)$.

\begin{thm}\label{AscForMPT}
	Let $(X,\mathscr{B},\mu,T)$ be a measure-preserving system and $\alpha$ a finite measurable partition of $X$. Let $A=[1]=\{\xi\in\Sigma_2^+:\xi_0=1\}$ and $\beta=\alpha\times A$ the related finite partition of $X\times A$. Denote by $T_{X\times A}$ the first-return map on $X\times A$ and let $P_A=P/P[1]$ denote the measure $P$ restricted to $A$ and normalized. Let $c_S^n=2^{-n}$ for all $S\subset n^*$. Then
	\be\begin{aligned}
	\Asc_\mu(X,\alpha,T)&=\frac{1}{2} \lim_{n \to \infty} \frac{1}{n} H_{\mu \times P_A}\left(\bigvee_{k=0}^{n-1}T_{X \times A}^{-k}\beta \Big\vert \bigvee_{k=0}^{n-1}T_{X \times A}^{-k}\mathcal A\right)\\
	&=\frac{1}{2} h_{\sigma_A}(X,T,\mu) 
	\leq \frac{1}{2}h_{\mu\times P_A}(X\times A,\beta,T_{X\times A}).
\end{aligned}\en
\end{thm}

Applying the definition of the entropy of a transformation with respect to a fixed partition as the integral of the corresponding information function and breaking up the integral into a sum of integrals over sets where the first-return time to $X \times A$ takes a fixed value produces the following result.

\begin{thm}\label{Ascseries}
	Let $(X,\mathscr{B},\mu,T)$ be a measure-preserving system, $\alpha$ a finite measurable partition of $X$, and $\alpha_i=T^{-i}\alpha$ for $i \geq 1$. 
	Let $c_S^n=2^{-n}$ for all $S\subset n^*$. Then
	\[
	\Asc_\mu(X,\alpha,T) = \frac{1}{2}\sum_{i=1}^\infty\frac{1}{2^i}H_\mu\left(\alpha\mid\alpha_{i}\right).
	\]
\end{thm}

\section{The search for maximizing measures on subshifts}
Given a topological dynamical system $(X,T)$, one would like to find the measures that maximize $\Asc$ and $\Int$, since the nature of these measures might tell us a lot about the balance between freedom and determinism within the system. For ordinary topological entropy and topological pressure with respect to a given potential, maximizing measures (measures of maximal entropy, equilibrium states) are of great importance and are regarded as natural measures on the system. We hope that Theorem \ref{Ascseries} might be helpful in the identification of these extremal measures.
In the topological case the first-return map $T_{X \times A}$ is not continuous nor expansive nor even defined on all of $X \times A$ in general, so known results about measures of maximal entropy and equilibrium states do not apply.
To maximize $\Int$, there is the added problem of the minus sign in
\[
\Int(X,\mathscr{U},T)=2\Asc(X,\mathscr{U},T)-\htop(X,\mathscr{U},T).
\]
Maybe some modern work on local or relative variational principles, almost subadditive potentials, equilibrium states for shifts with infinite alphabets, etc. could apply? (See \cite{Falconer1988, Barreira1996, Barreira2010, Mummert2006, CaoFengHuang2008, Yayama2011, IommiYayama2012, HuangYeZhang2006, HuangMaassRomagnoliYe2004, ChengZhaoCao2012} etc.)

\medskip
\begin{prop}
	When $T: X \to X$ is a homeomorphism on a compact metric space (e.g., $(X,T)$ is a subshift on  finite alphabet) and $\alpha$ is a Borel measurable finite partition of $X$, 
	$\Asc_\mu (X,T,\alpha)$ is an affine upper semicontinuous (in the weak* topology) function of $\mu$, so the set of maximal measures for $\Asc_\mu (X,T,\alpha)$ is nonempty, compact, and convex and contains some ergodic measures  (see \cite[p. 198 ff.]{waltersergodic}).
\end{prop}

We try now to find measures of maximal $\Asc$ or $\Int$ on SFT's, or at least maximal measures among all Markov measures of a fixed memory. Recall that a measure of maximal entropy on an SFT is unique and is a Markov measure, called the Shannon-Parry measure, denoted here by $\mu_{\max}$. Further, given a potential function $\phi$ that is a function of just two coordinates, again there is a unique {\em equilibrium} measure that maximizes
\be
P_\mu(\phi)=h_\mu(\sigma) + \int_X \phi \, d\mu.
\en
See \cite{ParryTuncel1982}.

%
%

			 A 1-step Markov measure on the full shift space $(\Sigma_n,\sigma)$ is given by s stochastic matrix $P=(P_{ij})$ and fixed probability vector $p=\left(\begin{array}{cccc}p_0&p_1&\dots&p_{n-1}\end{array}\right)$, i.e. $\sum p_i=1$ and $pP=p$.
			 The measure $\mu_{P,p}$ is defined as usual on cylinder sets by
			\[
			\mu_{p,P}[i_0i_1\dots i_k]=p_{i_0}P_{i_0i_1}\cdots P_{i_{k-1}i_k}.
			\]

	\begin{example}[$1$-step Markov measure on the golden mean shift]
		Denote by $P_{00}\in[0,1]$ the probability of going from $0$ to $0$ in a  sequence of $X_{\{11\}}\subset\Sigma_2$. Then
		\[
		P=\left(\begin{array}{cc}
		P_{00}&1-P_{00}\\
		1&0\end{array}\right),\quad
		p=\left(\begin{array}{cc}\frac{1}{2-P_{00}}&\frac{1-P_{00}}{2-P_{00}}\end{array}\right)
		\]
	\end{example}
	
	Using the series formula in Theorem \ref{Ascseries} and known equations for conditional entropy, we can approximate $\Asc_\mu$ and $\Int_\mu$ for Markov measures on SFTs.
  Let's look first at 1-step Markov measures.

	\begin{center}
			\begin{tabular}{c}
			\includegraphics[width=.65\textwidth,keepaspectratio]{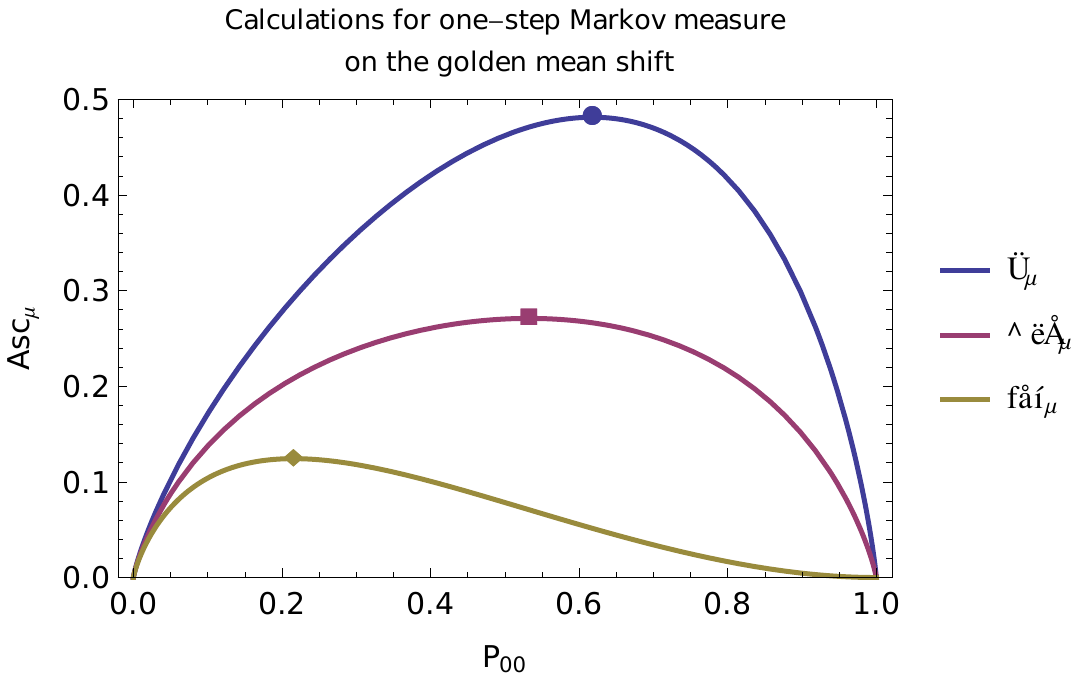}\\
			\begin{tabular}{@{}llll@{} }
				\toprule
				$P_{00}$&$h_\mu$&$\Asc_\mu$&$\Int_\mu$\\
				\midrule
				0.618&{ 0.481}&0.266&0.051\\
				0.533&0.471&{ 0.271}&0.071\\
				0.216&0.292&0.208&{ 0.124}\\
				\bottomrule\\
			\end{tabular}
		\end{tabular}
	\end{center}
	
		 Note that the maximum value of $h_\mu=\htop=\log\phi$ occurs when $P_{00}=1/\phi$;		
		 there are unique maxima among $1$-step Markov measures for $\Asc_\mu$ and $\Int_\mu$;
		 and the
		 maxima for $\Asc_\mu$, $\Int_\mu$, and $h_\mu$ are achieved by different measures.

Now let's calculate $\Asc$ and $\Int$ for various 2-step Markov measures on the golden mean SFT.

	\begin{center}
		$2$-step Markov measures on the golden mean shift
		\begin{tabular}{c}
			\includegraphics[width=.45\textwidth,keepaspectratio]{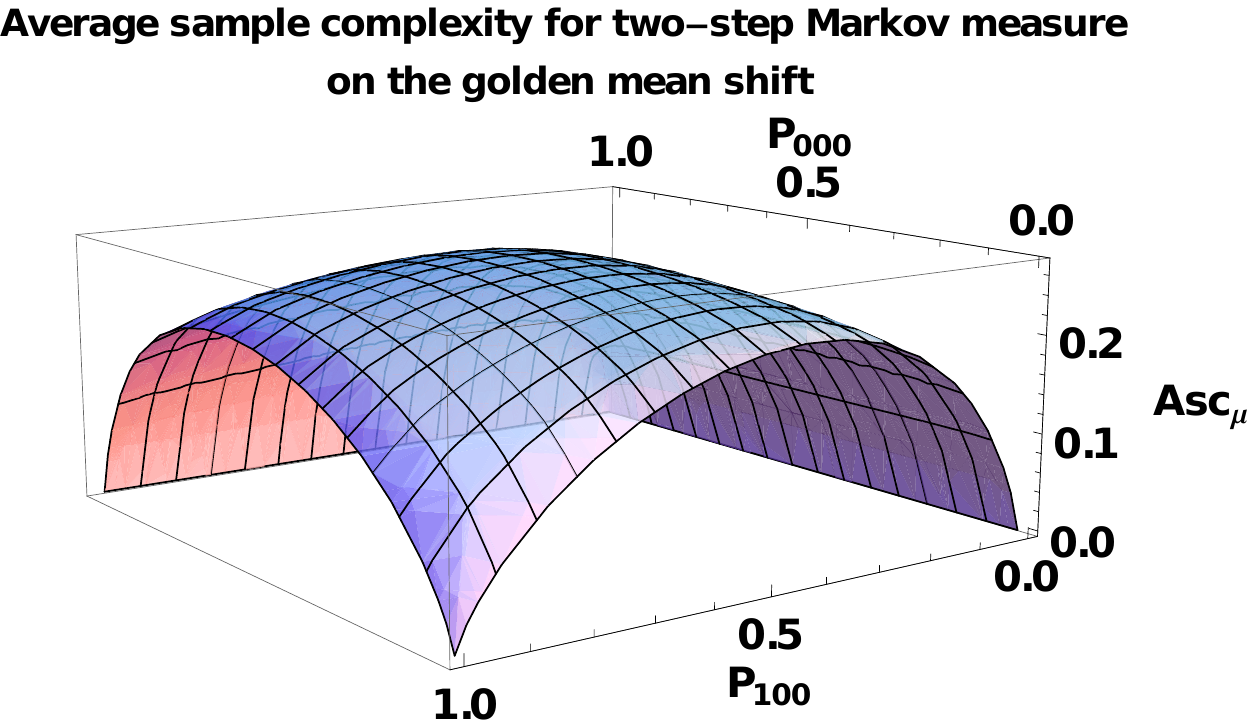}\quad
			\includegraphics[width=.45\textwidth,keepaspectratio]{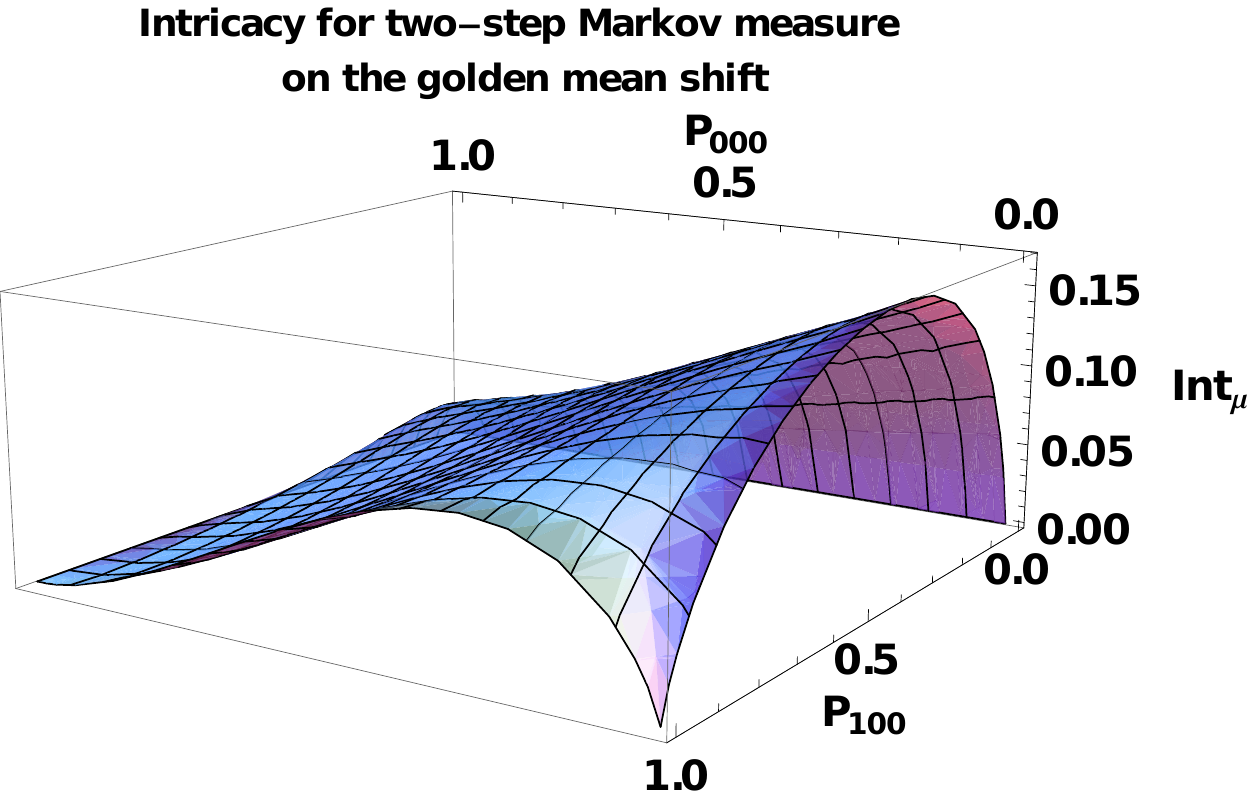}\\
			\begin{tabular}{@{}lllll@{} }
				\toprule
				$P_{000}$&$P_{100}$&$h_\mu$&$\Asc_\mu$&$\Int_\mu$\\
				\midrule
				0.618&0.618&{ 0.481}&0.266&0.051\\
				0.483&0.569&0.466&{ 0.272}&0.078\\
				0&0.275&0.344&0.221&{ 0.167}\\
				\bottomrule\\
			\end{tabular}
		\end{tabular}
	\end{center}
	
		 $\Asc_\mu$ appears to be strictly convex, so it would have a unique maximum among $2$-step Markov measures.
		 $\Int_\mu$ appears to have a unique maximum among $2$-step Markov measures on a proper subshift ($P_{000}=0$).
		 The maxima for $\Asc_\mu$, $\Int_\mu$, and $h_\mu$ are achieved by different measures, and are different from the measures that are maximal among 1-step Markov measures.

%
%
%
 Let's move from the golden mean SFT to the full 2-shift.

	\begin{center}
		$1$-step Markov measures on the full $2$-shift
		\begin{tabular}{c}
			\includegraphics[width=.45\textwidth,keepaspectratio]{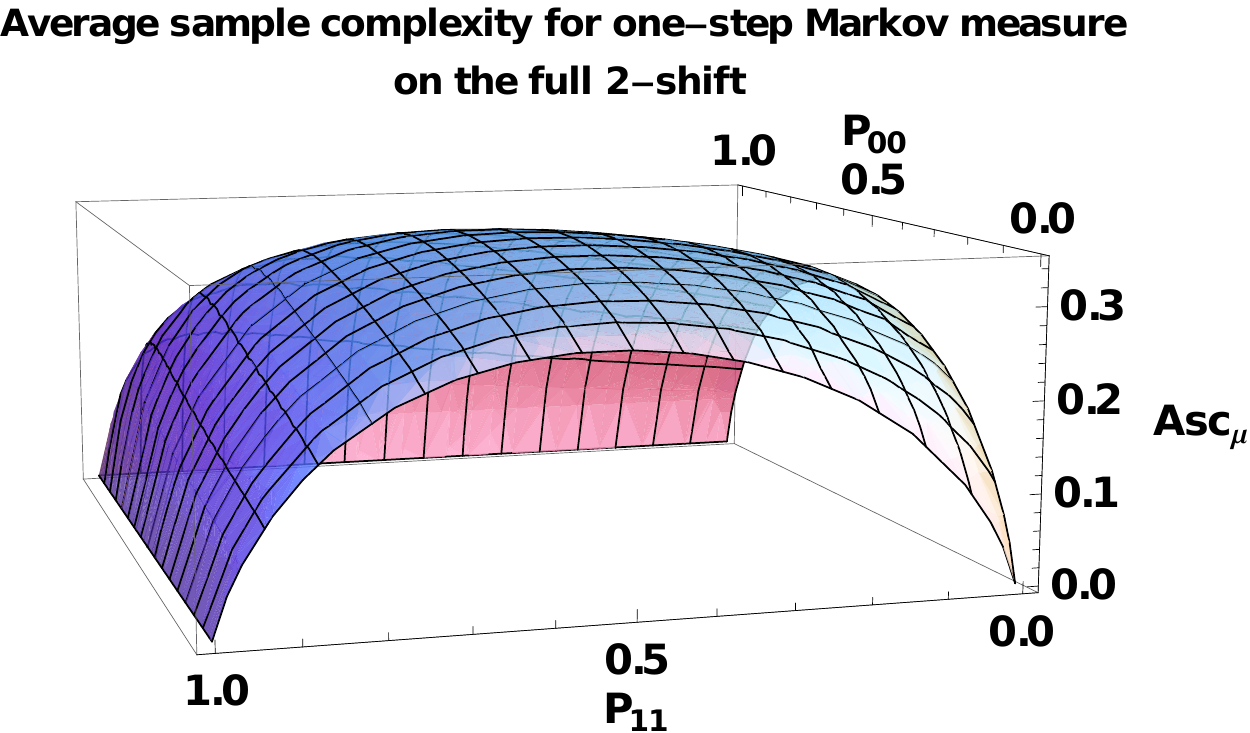}\quad\qquad
			\includegraphics[width=.45\textwidth,keepaspectratio]{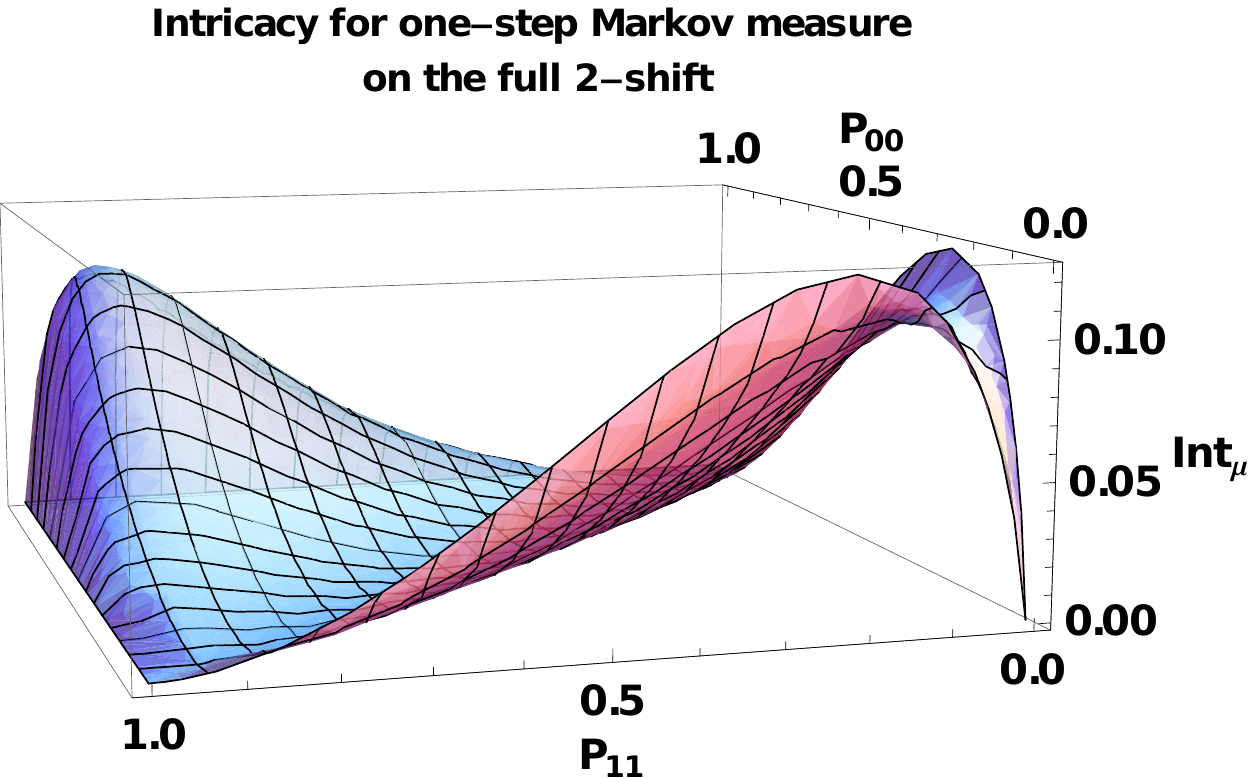}\\
			\begin{tabular}{@{}lllll@{} }
				\toprule
				$P_{00}$&$P_{11}$&$h_\mu$&$\Asc_\mu$&$\Int_\mu$\\
				\midrule
				0.5&0.5&{ 0.693}&{0.347}&0\\
				0.216&0&0.292&0.208&{0.124}\\
				0&0.216&0.292&0.208&{ 0.124}\\
				0.905&0.905&0.315&0.209&0.104\\
				\bottomrule\\
			\end{tabular}
					\end{tabular}
			\end{center}

		 $\Asc_\mu$ appears to be strictly convex, so it would have a unique maximum among $1$-step Markov measures.
		 $\Int_\mu$ appears to have {\em two} maxima among $1$-step Markov measures on proper subshifts ($P_{00}=0$ and $P_{11}=0$).
		 There seems to be a $1$-step Markov measure that is {\em fully supported} and is a {\em local maximum} for $\Int_\mu$ among all $1$-step Markov measures.
		 The maxima for $\Asc_\mu$, $\Int_\mu$, and $h_\mu$ are achieved by different measures.

%
%
%
%

	We summarize some of the questions generated above.

	\begin{conj}
		 On the golden mean SFT, for each $r$ there is a unique $r$-step Markov measure $\mu_r$ that maximizes $\Asc_\mu (X, \sigma, \alpha)$ among all $r$-step Markov measures.
		\begin{figure}
			\begin{center}
		\includegraphics[width=.65\textwidth,keepaspectratio]{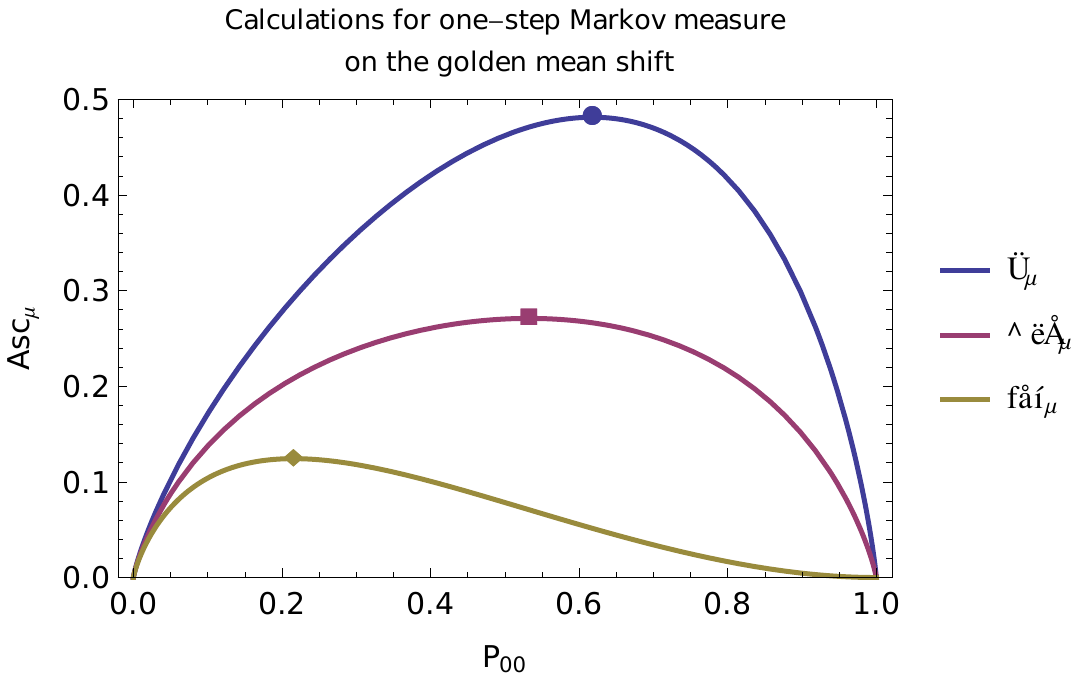}
	\end{center}
	\end{figure}
	\end{conj}

	\begin{conj} $\mu_2 \neq \mu_1$
				\begin{table}
					\begin{center}
		\begin{tabular}{@{}llll@{} }
					\toprule
			$P_{00}$&$h_\mu$&$\Asc_\mu$&$\Int_\mu$\\
			\midrule
			0.618&{ 0.481}&0.266&0.051\\
			0.533&0.471&{ 0.271}&0.071\\
			0.216&0.292&0.208&{ 0.124}\\
			\bottomrule\\
		\end{tabular}
		\end{center}
		\caption{Calculations for one-step Markov measures on the golden mean shift. Numbers in bold are maxima for the given categories. }
	\end{table}
		\begin{table}
			\begin{center}
		\begin{tabular}{@{}lllll@{} }
			\toprule
			$P_{000}$&$P_{100}$&$h_\mu$&$\Asc_\mu$&$\Int_\mu$\\
			\midrule
			0.618&0.618&{ 0.481}&0.266&0.051\\
			0.483&0.569&0.466&{ 0.272}&0.078\\
			0&0.275&0.344&0.221&{ 0.167}\\
			\bottomrule\\
		\end{tabular}
		\end{center}
		\caption{Calculations for two-step Markov measures on the golden mean shift.}
	\end{table}
	\end{conj}

	\begin{conj} On the golden mean SFT there is a unique measure that maximizes $\Asc_\mu (X,T,\alpha)$. It is not Markov of any order (and of course is not the same as $\mu_{\max}$).
	\end{conj}
	
\begin{conj} On the golden mean SFT for each $r$ there is a unique $r$-step Markov measure that maximizes $\Int_\mu (X,T,\alpha)$ among all $r$-step Markov measures.
		\begin{figure}
		\begin{center}
		\includegraphics[width=.65\textwidth,keepaspectratio]{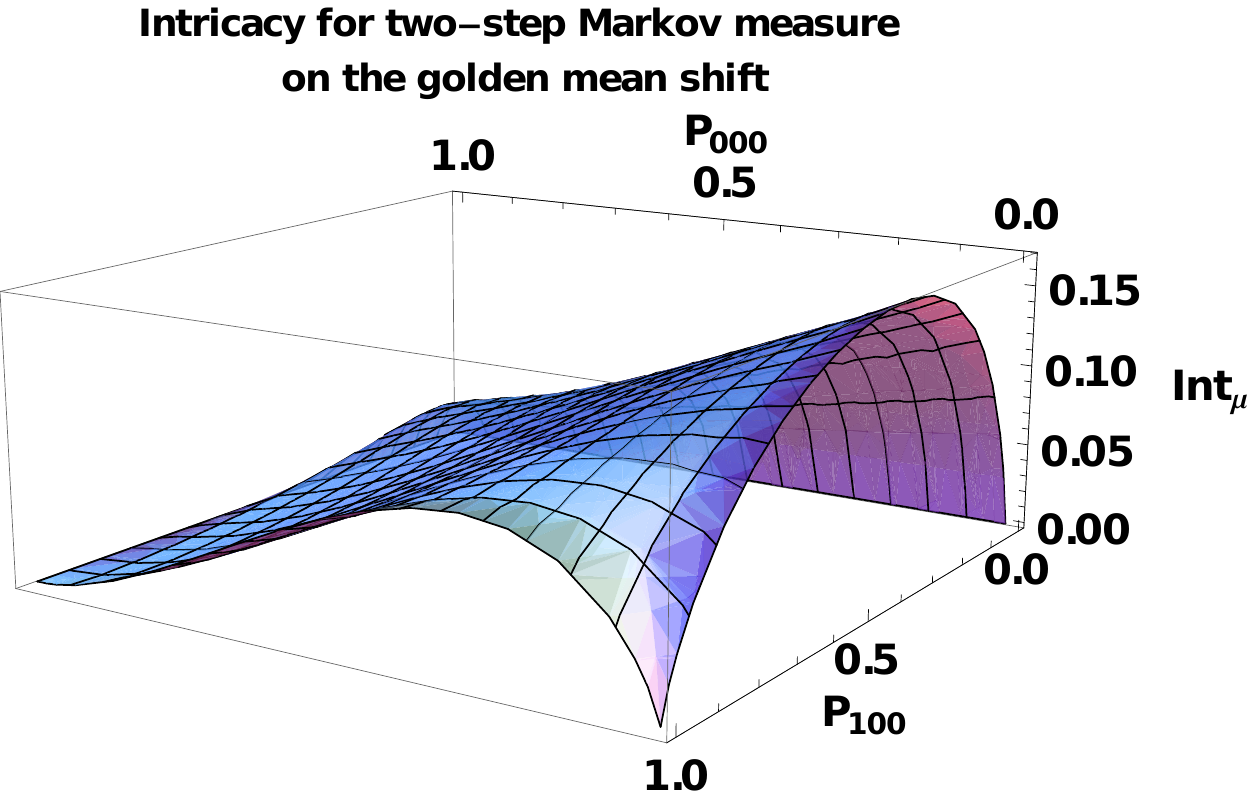}
	\end{center}
	\end{figure}
	\begin{figure}
		\begin{center}
		\includegraphics[width=.65\textwidth,keepaspectratio]{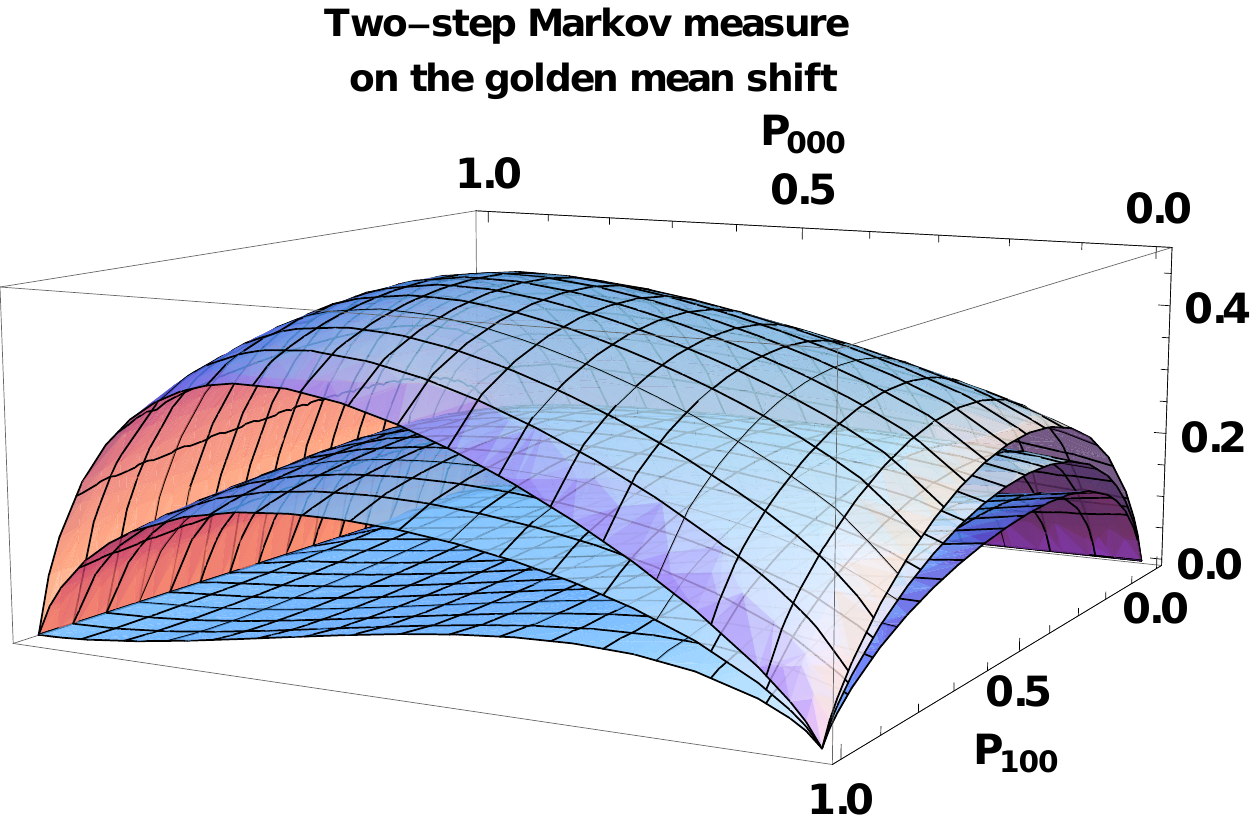}
		\caption{Combination of the plots of $h_\mu$, $\Asc_\mu$, and $\Int_\mu$ for two-step Markov measures on the golden mean shift.}
	\end{center}
	\end{figure}
\end{conj}

	\begin{conj} On the 2-shift there are {\em two} 1-step Markov measures that maximize $\Int_\mu (X,T,\alpha)$ among all 1-step Markov measures. They are supported on the golden mean SFT and its image under the dualizing map $0 \leftrightarrow 1$.
	\end{conj}

	\begin{conj} On the 2-shift there is a 1-step Markov measure that is {\em fully supported} and is a local maximum point for $\Int_\mu (X,T,\alpha)$ among all 1-step Markov measures.
		\end{conj}

		 The conjectures extend to arbitrary shifts of finite type and other dynamical systems. 	Many other natural questions suggested by the definitions and properties established so far of intricacy and average sample complexity can be found in the dissertation of Ben Wilson \cite{Wilson2015}:
				 		 \begin{enumerate}[1.]
			\item We do not know whether a variational principle $\sup_\mu \Asc_\mu (X,T,\alpha)=\Asc_{\topo}(X,\alpha,T)$ holds.
					\item Analogous definitions, results, and conjectures exist when entropy is generalized to pressure, by including a potential function which measures the energy or cost associated with each configuration.
				 First one can consider a function of just a single coordinate that gives the value of each symbol.
				Maximum intricacy may be useful for finding areas of high information activity, such as working regions in a brain (Edelman-Sporns-Tononi) or coding regions in genetic material (Koslicki-Thompson).
					\item Higher-dimensional versions, where subsets $S$ of coordinates are replaced by patterns, are naturally defined and waiting to be studied.
					\item One can define and then study average sample complexity of individual finite blocks.
				\item We need formulas for $\Asc$ and $\Int$ for more subshifts and systems.
				\item Find the subsets or patterns $S$ that maximize $\log N(\mathscr U_S)$ or
			$\log [ {N(\mathscr{U}_S)N(\mathscr{U}_{S^c})}]/{N(\mathscr{U}_{n^*}})$, and similarly for the measure-preserving case.
			\item In the topological case, what are the natures of the quantities that arise if one changes the definitions of $\Alt$ and $\Int$ by omitting the logarithms?
			\item Consider not just subsets $S$ and their complements, but partitions of $n^*$ into a finite number of subsets. For the measure-preserving case, there is a natural definition of the mutual information among a finite family of random variables on which one could base the definition of intricacy.
			\end{enumerate}
				We welcome help in resolving such questions and exploring further the ideas of intricacy, average sample complexity, and complexity in general!

\bibliographystyle{amsplain}
\bibliography{KEPSaltaBib}

\noindent{Department of Mathematics\\
		CB 3250 Phillips Hall\\
		University of North Carolina\\
		Chapel Hill, NC 27599 USA\\
	    {petersen{@}math.unc.edu}}

\end{document}